\documentclass[11pt,a4paper]{article}
\usepackage[T1]{fontenc}
\usepackage{lmodern}
\usepackage[margin=1in]{geometry}
\usepackage{amsmath,amssymb,amsthm,mathtools,bm,mathrsfs}
\usepackage{slashed}
\usepackage{microtype,array}
\usepackage[hidelinks]{hyperref}
\numberwithin{equation}{section}
\newtheorem{theorem}{Theorem}[section]
\newtheorem{proposition}[theorem]{Proposition}
\newtheorem{lemma}[theorem]{Lemma}
\newtheorem{corollary}[theorem]{Corollary}
\newtheorem{maintheorem}{Theorem}

\theoremstyle{definition}
\newtheorem{definition}[theorem]{Definition}
\theoremstyle{remark}

\newcommand{\dd}{\mathrm d}
\newcommand{\eps}{\epsilon}
\newcommand{\R}{\mathbb R}
\newcommand{\id}{\operatorname{id}}
\newcommand{\im}{\operatorname{im}}
\newcommand{\tr}{\operatorname{tr}}

\newcommand{\ad}{\operatorname{ad}}

\newcommand{\Hom}{\operatorname{Hom}}

\newcommand{\Ric}{\operatorname{Ric}}
\newcommand{\CS}{\operatorname{CS}}
\newcommand{\sg}{\mathfrak s}
\newcommand{\Bsp}{\beta_\Delta}
\newcommand{\red}[1]{{{#1}^{[0]}}}
\newcommand{\cA}{\mathcal A}

\newcommand{\cT}{\mathcal T}
\newcommand{\KHull}{K_{\mathrm{Hull}}}
\newcommand{\cL}{\mathcal L}

\newcommand{\cH}{\mathcal H}

\newcommand{\ph}{\mathrm{ph}}
\newcommand{\LC}{\mathrm{LC}}

\DeclareMathOperator{\vol}{vol}
\DeclareMathOperator{\Spin}{Spin}

\DeclareMathOperator{\Def}{Def}
\DeclareMathOperator{\Crit}{Crit}
\DeclareMathOperator{\diag}{diag}
\DeclareMathOperator{\diver}{div}
\allowdisplaybreaks[1]
\hypersetup{pdftitle={Derived deformation theory of heterotic G2 systems near the standard embedding},pdfauthor={Bram Brongers}}
\title{Derived deformation theory of heterotic $G_2$ systems near the standard embedding}
\author{Bram Brongers\\
\href{mailto:brambrongers98@gmail.com}{\texttt{brambrongers98@gmail.com}}}
\date{}
\begin{document}
\maketitle
\begin{abstract}
We study heterotic \(G_2\) structures near a torsion-free standard embedding, to first order in \(\alpha'\) and to all orders in formal deformations. 
The tangent connection is the Hull connection 
determined by the metric and flux. We construct a locally invertible family of differential operators relating the supersymmetry 
and instanton (BPS) equations to the Euler--Lagrange equations of the heterotic superpotential.

Using this comparison and the physical gauge symmetries, we construct a homological vector field locally formally
conjugate to the Hamiltonian vector field of the superpotential. On a closed manifold in the fixed gerbe sector, 
the torsion equation implies vanishing of the leading flux. We prove ellipticity of the linearized deformation complex 
with mixed differential-order weights and construct a finite-dimensional cyclic $L_\infty$ model with a formal effective potential.
\end{abstract}
\setcounter{tocdepth}{1}
\begingroup\small\tableofcontents\endgroup
\clearpage
\section{Introduction}\label{sec:introduction}

The deformation problem for heterotic $G_2$ systems asks which
deformations of a background continue to satisfy the heterotic
equations. Infinitesimal deformation theory identifies candidate
moduli, but does not determine which extend formally or how the
obstructions are organized. The heterotic $G_2$ superpotential has been
analysed through quadratic order, yielding the infinitesimal deformation
complex and its symplectic pairing \cite{QuantumG2}.
The formulation with an induced tangent connection was developed in
\cite{MSS}. The first version of this paper \cite{BrongersD11v1} imposed $H^{[0]}=0$ as an 
independent equation and resolved its compatibility identities. 
The resulting derived deformation theory differs from the one constructed here, 
although their ordinary solution spaces agree.

We study this problem near a torsion-free standard embedding, to first
order in $\alpha'$ and to all formal orders in the deformation
variables. We fix the tangent and gauge bundles on the spin seven-manifold \(Y\), 
together with their marked identification and background string, anomaly and relative gerbe data. 
The Courant formulas are used in their first-order specialization, as specified in Section 2. We restrict the
gauge sector to the rank-seven $SO(7)$ bundle identified with $TY$ at
the standard embedding. The varying structure consists of an admissible
metric, a compatible $G_2$ spinor line and a positive density.
Here an admissible metric is a positive rank-seven subbundle on which
the Courant anchor is an isomorphism onto $TY$.
The tangent connection is determined by the metric and flux through
the Hull condition. In a splitting of the Courant extension, the fields
are a positive three-form $\varphi$, a two-form potential $B$, a
dilaton $\Phi$ and a gauge connection $A$.

We construct a nonlinear local comparison between the supersymmetry
and instanton (BPS) residuals and the Euler--Lagrange equations of the
superpotential. Variation of the induced Hull connection contributes
a curvature term to the Euler--Lagrange equations. We express this
term as a differential operator on the BPS equation module.
The pairing and gauge action define a formal graded theory on fields,
ghosts and their density duals, with a square-zero vector field.
We call this the BPS--Noether enhancement.

\begin{maintheorem}\label{thm:intro-local-comparison}
Near the torsion-free standard embedding, let
$\mathcal E_{\rm BPS}=(C,\Sigma,T,I)$ be the BPS section and set
$\mathcal E_{\rm var}=-\dd W_{\rm ext}$, where $W_{\rm ext}$ is the
heterotic superpotential. There is a nonlinear family of invertible
finite-order differential operators $M_q$ from the BPS equation module
to the variational density dual such that
\[
 \mathcal E_{\rm var}(q)=-M_q\mathcal E_{\rm BPS}(q),\qquad
 M_q=\mathcal P_q-\mathcal L_q,\qquad
 M_q^{-1}=(1+\mathcal Q_q\mathcal L_q)\mathcal Q_q,
\]
where $\mathcal Q_q=\mathcal P_q^{-1}$ and
$(\mathcal Q_q\mathcal L_q)^2=0$.
The comparison extends to a local formal $Q$-isomorphism between the
BPS--Noether enhancement and the variational theory, fixing the physical fields
and all ghosts. Here a $Q$-isomorphism intertwines the
square-zero vector fields.
\end{maintheorem}

The operator comparison is proved in Section~\ref{sec:local-comparison}
using curvature pair exchange between the Hull connection and the
connection with opposite torsion. The leading flux $b=H^{[0]}$ is
closed throughout the field space, so the pair-exchange identity does
not require the BPS equations. A scalar correction to intrinsic-torsion
reconstruction expresses the curvature contribution in terms of the
BPS residuals. The correction is square-zero after composition with
$\mathcal P_q^{-1}$, giving the finite inverse.
On closed $Y$ in the fixed gerbe sector, we also prove that the torsion
equation implies $H^{[0]}=0$.

Section~\ref{sec:bps-noether} constructs the BPS--Noether enhancement.
Mixed-order ellipticity and cyclic homotopy transfer then give a
finite-dimensional cyclic formal model with effective potential
$W_{\rm eff}$.

The infinitesimal theory builds on the canonical $G_2$ complexes of
\cite{DLS}, their heterotic extension \cite{DLSHeterotic}, and the
induced-connection formulation of \cite{MSS}. The moduli metric is
studied in \cite{MSSMetric}. The superpotential Hessian and one-loop
theory of \cite{QuantumG2} and the BPS complexes of \cite{BPS} give
related variational descriptions. In the $SU(3)$ setting, a heterotic
superpotential defines holomorphic finite deformations governed by an
$L_3$ algebra, an $L_\infty$ algebra whose brackets vanish above arity
three \cite{FiniteDeformations}.

\section{Heterotic \texorpdfstring{$G_2$}{G2} geometry and string Courant conventions}
\label{sec:geometry}

\subsection{The fixed string Courant background}
\label{subsec:string-background}
Fix a spin seven-manifold \(Y\), tangent and gauge principal bundles \(P_T,P_G\), and compatible background string, 
anomaly and relative gerbe data. We use the associated real string Courant construction at nonzero anomaly coefficient \(z\). 
Whenever integration over $Y$ is used, assume that $Y$ is closed, connected and oriented.

Let $\varphi_0$ be the torsion-free background form, with
$g_0=g_{\varphi_0}$ and $\psi_0=*_{g_0}\varphi_0$.
Let $E\to Y$ be a rank-seven real vector bundle with structure group
$SO(7)$ and its associated fibre metric. Fix a spin-compatible
isometric identification $E\simeq(TY,g_0)$.
Write $\mathfrak{so}(E)$ for the bundle of skew-adjoint endomorphisms.
We use the fixed identification to express both gauge and tangent
connections on $E$.

The background connection is $A_0=\Theta_0=\Theta_{\LC}(g_0)$.
Write $\nabla^0$ for its covariant derivative. Its curvature $R_0$ satisfies
\begin{equation}\label{eq:background-identities}
 \nabla^0\varphi_0=0,\quad \dd\psi_0=0,\quad
 D_{\Theta_0}R_0=0,\quad R_0\wedge\psi_0=0,\quad \Ric(g_0)=0.
\end{equation}
On both copies of $\mathfrak{so}(E)$, use the invariant form
\begin{equation}\label{eq:invariant-form}
 P(a,b)=-\tr_7(ab).
\end{equation}
The same symbol denotes its extension to bundle-valued forms, with
exterior products in the displayed order.
For a fixed $z\in\mathbb R^\times$, the Courant extension and anomaly
pairing are
\begin{equation}\label{eq:courant-extension}
\begin{gathered}
0\longrightarrow T^*Y\xrightarrow{j}\mathbb E
 \longrightarrow\operatorname{At}(P_T\times_Y P_G)\longrightarrow0,\\
\mathfrak k=\ad P_T\oplus\ad P_G,\qquad
\mathfrak c_z((u_T,u_G),(v_T,v_G))
 =\frac z4\bigl(P(u_T,v_T)-P(u_G,v_G)\bigr).
\end{gathered}
\end{equation}
In a connection and flux splitting,
\begin{equation}\label{eq:courant-pairing}
\begin{gathered}
\mathbb E\simeq TY\oplus\mathfrak k\oplus T^*Y,\qquad
 a_{\mathbb E}(X+u+\eta)=X,\\
\langle X+u+\eta,Y+v+\zeta\rangle
 =\tfrac12\bigl(\zeta(X)+\eta(Y)\bigr)+\mathfrak c_z(u,v),\qquad
\dd H=\mathfrak c_z(F_C,F_C).
\end{gathered}
\end{equation}
Here $C=(C_T,A)$ and $F_C=(R_{C_T},F_A)$.
We use the Dorfman bracket $\circ$ recorded in
Appendix~\ref{app:courant-action}.

\begin{definition}\label{def:courant-heterotic-data}
An admissible metric is a positive rank-seven subbundle
$V_+\subset\mathbb E$ such that
$a_{\mathbb E}|_{V_+}:V_+\to TY$ is an isomorphism.
Set $V_-=V_+^\perp$ and let $\pi_-$ be the orthogonal projection
onto $V_-$.
\end{definition}

\begin{proposition}\label{prop:adapted-courant-data}
Let $v_+=(a_{\mathbb E}|_{V_+})^{-1}$.
The admissible metric determines a Riemannian metric $g$, an isotropic
splitting $s$, a principal connection $C=(C_T,A)$ and a three-form $H$ by
\begin{equation}\label{eq:adapted-courant-data}
\begin{gathered}
 g(X,Y)=\langle v_+(X),v_+(Y)\rangle,\qquad
 s(X)=v_+(X)-j(gX),\\
 H(X,Y,Z)=2\langle s(X)\circ s(Y),s(Z)\rangle.
\end{gathered}
\end{equation}
With $\sigma_\pm(X)=s(X)\pm j(gX)$,
\begin{equation}\label{eq:canonical-hull-tangent}
 a_{\mathbb E}\pi_-\bigl(\sigma_+(X)\circ\sigma_-(Y)\bigr)
 =\nabla^g_XY-\tfrac12g^{-1}H(X,Y,\cdot).
\end{equation}
\end{proposition}
\begin{proof}
The pairing gives
\[
 \langle s(X),s(Y)\rangle
 =g(X,Y)-\tfrac12g(X,Y)-\tfrac12g(X,Y)=0.
\]
The Atiyah projection of $s$ splits the Atiyah sequence and therefore
defines $C$. The Dorfman bracket gives $H$ by
\eqref{eq:adapted-courant-data}, and the Courant identity gives
$\dd H=\mathfrak c_z(F_C,F_C)$.
In the induced splitting,
\[
 V_+=\{X+gX\},\qquad V_-=\{X+u-gX\}.
\]
Substitution into the Dorfman bracket gives
\[
 g\bigl(Z,a_{\mathbb E}\pi_-(\sigma_+(X)\circ\sigma_-(Y))\bigr)
 =g(Z,\nabla^g_XY)-\tfrac12H(X,Y,Z),
\]
which is \eqref{eq:canonical-hull-tangent}.
\end{proof}

Represent the identification $TY\simeq E$ by a spin-compatible isometry
$e:(TY,g)\to E$, modulo tangent-frame changes. The Hull condition is
\begin{equation}\label{eq:intrinsic-hull}
 e^{-1}\nabla^{C_T}_X(eY)=\nabla^g_XY-\tfrac12g^{-1}H(X,Y,\cdot).
\end{equation}
A heterotic $G_2$ configuration consists of an admissible metric,
a compatible unit spinor line $[\chi]\subset S(V_+)$, a positive
density $\varrho$ and an isometry $e$ satisfying
\eqref{eq:intrinsic-hull}.

The anchor identifies $S(V_+)$ with the metric spinor bundle.
The stabilizer of a unit real spinor is $G_2$, and $\chi$ and
$-\chi$ define the same positive three-form. For the symmetric spin
pairing $\Bsp$, our Clifford and bilinear conventions are
\begin{equation}\label{eq:spinor-bilinears}
\begin{gathered}
 c_g(X)c_g(Y)+c_g(Y)c_g(X)=-2g(X,Y),\\
 \varphi(X,Y,Z)=\Bsp(\chi,c_g(X^\flat\wedge Y^\flat\wedge Z^\flat)\chi),\\
 \psi(X,Y,Z,W)=-\Bsp(\chi,c_g(X^\flat\wedge Y^\flat\wedge Z^\flat\wedge W^\flat)\chi).
\end{gathered}
\end{equation}
These forms satisfy $\psi=*_{g_\varphi}\varphi$.
The metric and volume satisfy
\begin{equation}\label{eq:stable-metric}
 \tfrac16(\iota_v\varphi)\wedge(\iota_w\varphi)\wedge\varphi
 =g_\varphi(v,w)\vol_\varphi,\qquad
 \varphi\wedge\psi=7\vol_\varphi.
\end{equation}
We write $\pi_r$ for the orthogonal projections in
\[
 \Lambda^2=\Lambda^2_7\oplus\Lambda^2_{14},\qquad
 \Lambda^3=\Lambda^3_1\oplus\Lambda^3_7\oplus\Lambda^3_{27}.
\]
Our torsion convention is
\begin{equation}\label{eq:torsion-classes}
 \dd\varphi=\tau_0\psi+3\tau_1\wedge\varphi+*\tau_3,\qquad
 \dd\psi=4\tau_1\wedge\psi+*\tau_2.
\end{equation}
The density is an independent field. Set
\begin{equation}\label{eq:native-density}
 \varrho=e^{-2\Phi}\vol_g,\qquad \zeta_\ph=e^{-\Phi}\chi.
\end{equation}
Its divergence is
\begin{equation}\label{eq:density-divergence}
 \operatorname{div}_{\varrho}(q)
 =\varrho^{-1}\mathcal L_{a_{\mathbb E}(q)}\varrho
 =\operatorname{div}_g(a_{\mathbb E}(q))-2\dd\Phi(a_{\mathbb E}(q)).
\end{equation}
Since $\operatorname{div}_{c\varrho}=\operatorname{div}_{\varrho}$
for every positive constant $c$, the divergence does not determine
the constant dilaton.

\subsection{Mixed coefficients and flux}
\label{subsec:physical-splitting}\label{subsec:eft}
We work to first order in $\alpha'$ and to all formal orders in the
field deformations. Put
\begin{equation}\label{eq:epsilon-ideal-inclusion}
 \mathbb D=\R[\eps]/(\eps^2),\qquad
 \mathbb D_0=\mathbb D/(\eps),\qquad
 \jmath_\eps:\mathbb D_0\xrightarrow{\sim}\eps\mathbb D,
 \quad [x]\mapsto\eps x.
\end{equation}
Whenever a nondegenerate Courant pairing is required, we first establish the relevant 
identity at nonzero real anomaly coefficient \(z\). We then specialize its regular splitting 
formula to \(z=\epsilon\) in the mixed coefficient modules. This specialization does not assert 
the existence of a Courant algebroid with nondegenerate pairing over \(\mathbb D\).
The geometric fields $\varphi,B,\Phi$ have coefficients in $\mathbb D$,
while $A$ has coefficients in $\mathbb D_0$.
Gauge contributions to the geometric equations lie in $\eps\mathbb D$.
The inclusion $\jmath_\eps$ is applied after reduction modulo $\eps$.
A smooth fibre map $F$ extends by
\[
 F(q^{[0]}+\eps q^{[1]})
 =F(q^{[0]})+\eps\,DF_{q^{[0]}}q^{[1]}.
\]

In a splitting of the fixed differential-gerbe sector, the coordinates
are $(\varphi,B,\Phi,A)$.
For a connection $\Theta$ and an adjoint-valued one-form $a$, define
\begin{equation}\label{eq:relative-transgression}
 \cT_3(\Theta,a)=2P(a,R_\Theta)+P(a,D_\Theta a)
                    -\tfrac23\tr_7(a\wedge a\wedge a).
\end{equation}
Our contorsion convention is
\begin{equation}\label{eq:contorsion}
 g(Y,C_g(H)_X Z)=\tfrac12H(X,Y,Z).
\end{equation}
The isometry $I_g=(g_0^{-1}g)^{-1/2}$ transports connections to the
fixed bundle $E$. Write $c_0=c_{g_0}$ and let $\sg$ be the
infinitesimal spin representation, with
\begin{equation}\label{eq:spin-action}
 [\sg(Z),c_0(X)]=c_0(ZX),\qquad \sg(Z)^*=-\sg(Z).
\end{equation}
For the spin Dirac operator $D_g$, Appendix~\ref{app:clifford} gives
$\slashed D_{\Theta_{\LC}+C_g(H)}=D_g-3c_g(H)/4$.
Set
\begin{align}
 \Theta^u&=\Theta_{\LC}(g)+C_g(\dd B),&
 a&=A-\red{\Theta^u},\label{eq:preliminary-connection}\\
 H&=\dd B-\tfrac14\jmath_\eps\cT_3(\red{\Theta^u},a),&
 \Theta_H&=\Theta_{\LC}(g)+C_g(H).\label{eq:determined-graph}
\end{align}

\begin{lemma}\label{lem:nilpotent-flux}
Equations \eqref{eq:preliminary-connection}--\eqref{eq:determined-graph}
give the unique solution of
\[
 H=\dd B-\tfrac14\jmath_\eps
           \cT_3(\red{\Theta_H},A-\red{\Theta_H}),\qquad
 \Theta_H=\Theta_{\LC}(g)+C_g(H)
\]
for nearby fields. They satisfy
\begin{equation}\label{eq:anomaly}
 \dd H=-\tfrac14\jmath_\eps
 \bigl(P(F_A,F_A)-P(R_{\red{\Theta_H}},R_{\red{\Theta_H}})\bigr).
\end{equation}
\end{lemma}
\begin{proof}
Reduction modulo $\eps$ gives
\[
 \red H=\dd\red B,\qquad
 \red{\Theta_H}=\Theta_{\LC}(g^{[0]})+C_{g^{[0]}}(\dd B^{[0]})
              =\red{\Theta^u}.
\]
Substitution determines $H$ and then $\Theta_H$ uniquely.
For $\CS_P(C)=P(C,\dd C)-2\tr_7(C^3)/3$, the transgression identity is
\begin{equation}\label{eq:transgression-identity}
 \cT_3(\Theta,a)=\CS_P(\Theta+a)-\CS_P(\Theta)+\dd P(\Theta,a).
\end{equation}
Applying $\dd$ gives \eqref{eq:anomaly}.
The Green--Schwarz transformations in Appendix~\ref{app:finite-gerbe}
make $H$ a global three-form.
\end{proof}

\subsection{The BPS section and the flux identity}

Let $\mathscr V$ be the formal neighbourhood of
$(\varphi_0,B_0,\Phi_0,A_0)$, with $\Phi_0$ constant and
$\dd B_0^{[0]}=0$. The tangent connection is given by
\eqref{eq:determined-graph}. The BPS equation module is
\[
 \mathcal R_{\rm BPS}
 = (\Omega^5\oplus\Omega^7\oplus\Omega^3)\otimes\mathbb D
   \oplus\Omega^6(\mathfrak{so}(E))\otimes\mathbb D_0.
\]
With $w=e^{-2\Phi}$, define the BPS section by
\begin{equation}\label{eq:bps-residuals}
\begin{aligned}
 \mathcal E_{\rm BPS}(q)&=(C,\Sigma,T,I),\\
 C&=\dd(w\psi),&\Sigma&=\dd\varphi\wedge\varphi,\\
 T&=H+*\dd\varphi-2*(\dd\Phi\wedge\varphi),&
 I&=F_A\wedge\red\psi.
\end{aligned}
\end{equation}
\label{par:fixed-gerbe-sector}
The gerbe and anomaly data, including the background trivialization,
are fixed. Since the Green--Schwarz corrections lie in $\eps\mathbb D$,
$\beta=(B-B_0)^{[0]}$ is a global two-form and
$b=H^{[0]}=\dd\beta$.

\begin{proposition}[Flux identity]\label{prop:compact-flux-redundancy}
For $q\in\mathscr V$, put $\phi=\varphi^{[0]}$,
$\psi=*_{\phi}\phi$, $w_0=e^{-2\Phi^{[0]}}$ and
$\delta_w=w_0^{-1}\delta_gw_0$.
Let $G_w^{(k)}$ be the Green operator of
$\Delta_w=\dd\delta_w+\delta_w\dd$ on $k$-forms, zero on
$\ker\Delta_w$. Then
\begin{equation}\label{eq:compact-flux-projector}
 T^{[0]}=b+\delta_w\psi,\qquad
 b=\dd G_w^{(2)}\delta_wT^{[0]}.
\end{equation}
These identities extend to nilpotent graded coefficients. In particular,
\begin{equation}\label{eq:physical-configurations}
 \mathcal Z_{\rm BPS}:=\mathcal E_{\rm BPS}^{-1}(0)
 \subset\{q\in\mathscr V:H^{[0]}=0\}.
\end{equation}
\end{proposition}
\begin{proof}
On four-forms in dimension seven, $\delta_g=*\dd*$.
Since $*(p\wedge\phi)=-\iota_{p^\sharp}\psi$,
\[
 \delta_w\psi
 =\delta_g\psi+2\iota_{\nabla\Phi^{[0]}}\psi
 =*\dd\phi-2*(\dd\Phi^{[0]}\wedge\phi)
 =T^{[0]}-b.
\]
With $\Pi_w^{(k)}$ the weighted harmonic projection, Hodge
decomposition gives
\[
 1=\Pi_w^{(k)}+\dd G_w^{(k-1)}\delta_w
                    +\delta_wG_w^{(k+1)}\dd.
\]
The exact form $b$ has zero harmonic component and is closed. Hence
\[
 b=\dd G_w^{(2)}\delta_wb
  =\dd G_w^{(2)}\delta_w(T^{[0]}-\delta_w\psi)
  =\dd G_w^{(2)}\delta_wT^{[0]},
\]
using $\delta_w^2=0$.
Since $\dim\ker\Delta_w^{(k)}=b_k(Y)$, the harmonic projection and
Green operator depend smoothly on the metric and weight.
Finite Taylor expansion extends \eqref{eq:compact-flux-projector}
to nilpotent coefficient ideals, also after tensoring with polynomial
differential forms on simplices.
\end{proof}

\section{The generalized Dirac functional and variational Hamiltonian theory}
\label{sec:equivalence}\label{sec:deformation-algebra}

\subsection{The generalized Dirac functional}\label{subsec:courant-functional}
Let $\slashed D^+$ be the positive generalized Dirac operator associated
with $(V_+,\operatorname{div}_{\varrho})$.
It is independent of the choice of compatible generalized connection
by \cite[Lemma 3.4]{GFDirac}. In our Clifford convention,
\begin{equation}\label{eq:canonical-generalized-dirac}
 \slashed D^+=D_g+\tfrac14c_g(H)-\tfrac12c_g(\vartheta),
 \qquad \vartheta=2\dd\Phi.
\end{equation}
This splitting formula is established at nonzero real anomaly coefficient and is regular at \(z=0\). 
In the first-order theory, \(\not D^+\) and its expectation value denote the corresponding specializations at \(z=\epsilon\).
Define
\begin{equation}\label{eq:courant-functional}
 W_{\mathrm{Courant},G_2}
 =\int_Y\varrho\,\Bsp(\chi,\slashed D^+\chi),
\end{equation}
and, in the variables $(g,\zeta_\ph,B,A)$,
\begin{equation}\label{eq:native-source}
 W_{\rm ext}=\int_Y\bigl[\Bsp(\zeta_\ph,D_g\zeta_\ph)
 +\tfrac14\Bsp(\zeta_\ph,H\cdot_g\zeta_\ph)\bigr]\vol_g.
\end{equation}

\begin{proposition}\label{prop:courant-functional}
For configurations satisfying the Hull condition,
$W_{\mathrm{Courant},G_2}=W_{\rm ext}$.
\end{proposition}
\begin{proof}
Skew-adjointness of one-form Clifford multiplication gives
\[
 \Bsp(\chi,c_g(\alpha)\chi)=0
 \qquad(\alpha\in\Omega^1(Y)).
\]
Since
\[
 D_g(e^{-\Phi}\chi)
 =e^{-\Phi}\bigl(D_g\chi-c_g(\dd\Phi)\chi\bigr),
\]
we obtain
\[
 \Bsp(\zeta_\ph,D_g\zeta_\ph)
 +\tfrac14\Bsp(\zeta_\ph,H\cdot_g\zeta_\ph)
 =e^{-2\Phi}\Bsp(\chi,\slashed D^+\chi).
\]
Multiplication by $\vol_g$ gives the identity of integrands.
\end{proof}

\begin{proposition}\label{prop:spinorial-source}
For every nearby $\mathbb D$-valued three-form $\varphi$ with positive
reduction $\varphi^{[0]}$,
\begin{equation}\label{eq:form-source}
 W_{\rm ext}=\frac14\int_Y e^{-2\Phi}
 \bigl(H\wedge\psi-\tfrac12\dd\varphi\wedge\varphi\bigr).
\end{equation}
\end{proposition}
\begin{proof}
The spinor identities of Lemma~\ref{lem:spinor-source-identities} are
\[
 \dd\varphi\wedge\varphi=-8\Bsp(\chi,D_g\chi)\vol_g,
 \qquad H\wedge\psi=\Bsp(\chi,H\cdot\chi)\vol_g.
\]
Substitution into \eqref{eq:native-source}, with the preceding
skew-adjointness calculation, gives \eqref{eq:form-source}.
\end{proof}

Equation~\eqref{eq:form-source} fixes our normalization of the
heterotic $G_2$ superpotential of \cite{G2Superpotential}.
We use the Euler--Lagrange sign convention
\begin{equation}\label{eq:euler-sign-convention}
 \mathcal E_{\rm EL}=\dd W_{\rm ext},\qquad
 \mathcal E_{\rm var}=-\dd W_{\rm ext},
\end{equation}
where $\dd W_{\rm ext}$ is the density covector obtained by integration
by parts.

\subsection{Local formal maps and mixed duality}\label{subsec:mixed-variational-duals}
Local functionals are horizontal densities modulo total derivatives:
\[
 \mathcal H_{\rm loc}=\Omega^7_{\rm hor}/\dd_Y\Omega^6_{\rm hor}.
\]
Evolutionary derivations act on jets by prolongation and commute with
total derivatives. Formal completion is taken in the field
displacements, independently of $\eps$.

\begin{definition}[Local formal maps]\label{def:real-local-formal}
A local formal map over $\mathbb R$ is a formal series whose $n$th
Taylor coefficient is a real multilinear differential operator of
finite order in each variable. The order may depend on $n$.
A local formal isomorphism has an inverse of this type.
A local formal $Q$-theory is a formal graded space with a local
evolutionary vector field $Q$ of ghost number one, satisfying $Q^2=0$,
and a marked zero of $Q$. A $Q$-map intertwines these vector fields.
\end{definition}

\begin{definition}[Mixed local fields and functions]
\label{def:mixed-hamiltonian-category}\label{def:mixed-local-functions}
Let
\[
 \mathcal E=(E_{\rm fr}\otimes\mathbb D)
             \oplus(E_0\otimes\mathbb D_0)
\]
be a graded coefficient bundle, with $E_{\rm fr}$ and $E_0$ finite-rank
real bundles. In coordinates $x=x^{[0]}+\eps x^{[1]}$ on the first
summand and $a$ on the second, the mixed local functions have the form
\begin{equation}\label{eq:mixed-local-function}
 F=f_0(x^{[0]})
   +\eps\bigl(Df_0(x^{[0]})[x^{[1]}]+f_1(x^{[0]},a)\bigr).
\end{equation}
The formal series $f_0,f_1$ have finite jet order in each Taylor degree.
Mixed maps have $\mathbb D$-multilinear Taylor coefficients preserving
the ideal filtration. Their $\mathbb D_0$-valued components depend
only on $(x^{[0]},a)$.
\end{definition}

\begin{definition}[Variational dual]\label{def:mixed-variational-dual}
The variational dual is
\[
 \mathcal E^\vee=\Hom_{\mathbb D}(\mathcal E,\mathbb D)
                         \otimes\operatorname{Dens}(Y).
\]
In particular,
\[
 \Hom_{\mathbb D}(\mathbb D_0,\mathbb D)=\eps\mathbb D,
\]
so a $\mathbb D_0$-valued field pairs with an $\eps\mathbb D$-valued
density covector. Integration by parts defines $\dd F$ by
\[
 \delta F=\int_Y\langle\dd F,\delta q\rangle.
\]
\end{definition}

The duality follows from
$\operatorname{Ann}_{\mathbb D}(\eps)=\eps\mathbb D$.
The coefficient trace $\tau_{\mathbb D}(a+\eps b)=b$ gives the real
pairings
\[
 \tau_{\mathbb D}\bigl((a_0+\eps a_1)(b_0+\eps b_1)\bigr)
 =a_0b_1+a_1b_0,\qquad
 \tau_{\mathbb D}(\jmath_\eps(ab))=ab.
\]

A coordinate of ghost number $g$ has cochain degree $1-g$.
Fields, equations and Noether identities have cochain degrees $1,2,3$.
Scalar reducibility has degree $-1$ and its dual has degree $4$.
The suspension convention is
\begin{equation}\label{eq:grading-suspension}
 |\mathbf sx|=|x|-1.
\end{equation}
For a horizontal form $\alpha$, the extension of $Q$ satisfies
\begin{equation}\label{eq:horizontal-vertical-interface}
 \widehat Q(\alpha a)=(-1)^{\deg_{\rm hor}\alpha}\alpha Q(a).
\end{equation}
The vertical differential and contraction have bidegrees
\begin{equation}\label{eq:vertical-bigrading}
 \deg\delta=(0,1),\qquad\deg\iota_X=(|X|,-1).
\end{equation}
Exchanging bidegrees $(g,v)$ and $(g',v')$ gives $(-1)^{gg'+vv'}$.
Horizontal exterior signs are applied separately.

\subsection{The shifted cotangent construction}

\begin{definition}[Mixed symplectic form]\label{def:mixed-symplectic}
A closed variational two-form $\omega$ of ghost number $-1$ is mixed
symplectic if $X\mapsto\iota_X\omega$ is an isomorphism from the
tangent module to its shifted variational dual, with a local inverse.
Hamiltonian vector fields and brackets are defined by
\[
 \iota_{X_F}\omega=-\delta F,\qquad \{F,G\}=X_FG
\]
for local functionals of the form \eqref{eq:mixed-local-function}.
\end{definition}

For a base coordinate $q$, the conjugate momentum $p_q$ has ghost
number $-1-|q|$. With covectors written first, the cotangent potential
and symplectic form are
\begin{equation}\label{eq:canonical-cotangent-convention}
 \vartheta=\int_Y\sum_q\langle p_q,\delta q\rangle,\qquad
 \omega=\delta\vartheta.
\end{equation}
For a base vector field $f$ of ghost number one, its cotangent
Hamiltonian is
\begin{equation}\label{eq:cotangent-hamiltonian}
 H_f=\iota_f\vartheta=\sum_q(-1)^{|p_q|}p_q[fq].
\end{equation}
A momentum conjugate to a $\mathbb D_0$-valued coordinate is written
$p=\jmath_\eps(\bar p)$, with $\bar p\in\mathbb D_0$.

\begin{proposition}[Mixed cotangent bracket]
\label{prop:mixed-hamiltonian-domain}
Every mixed local functional $F$ has a unique local Hamiltonian vector
field $X_F$. The bracket $\{F,G\}=X_FG$ is a $\mathbb D$-bilinear
graded Lie bracket of degree one.
\end{proposition}
\begin{proof}
The cotangent equations are invertible on the free summand.
On the quotient summand, invertibility follows from
$\jmath_\eps:\mathbb D_0\xrightarrow{\sim}\eps\mathbb D$.
For
\[
 F=f_0^{\mathbb D}+\jmath_\eps(f_1),\qquad
 G=g_0^{\mathbb D}+\jmath_\eps(g_1),
\]
where $f_0^{\mathbb D}$ denotes the dual-number extension in
\eqref{eq:mixed-local-function}, the bracket is
\[
 \{F,G\}
 =\{f_0,g_0\}_{\rm fr}^{\mathbb D}
 +\jmath_\eps\bigl(
    \{f_0,g_1\}_{\rm fr}
   +\{f_1,g_0\}_{\rm fr}
   +\{f_1,g_1\}_0\bigr).
\]
The subscripts denote cotangent brackets in the free and quotient
coordinates. Since $\eps^2=0$, this expression has the form
\eqref{eq:mixed-local-function}.
Graded antisymmetry and Jacobi follow from the cotangent identities.
Integration by parts gives the same identities for local functionals.
\end{proof}

\begin{lemma}[Cotangent master equation]\label{lem:cotangent-master}
Let $f$ be a homological vector field on a graded base and let $W$ be
a local functional of ghost number zero with $fW=0$.
Then $\mathcal S=W+\iota_f\vartheta$ satisfies
$\{\mathcal S,\mathcal S\}=0$, and its Hamiltonian vector field
satisfies $Q^2=0$.
\end{lemma}
\begin{proof}
Cotangent lifts preserve commutators and their Hamiltonians vanish on
the zero section. Therefore
\[
 \tfrac12\{H_f,H_f\}=H_{f^2},\qquad
 \{H_f,W\}=fW,\qquad \{W,W\}=0.
\]
The bracket is symmetric on ghost-number-zero Hamiltonians, so
\[
 \tfrac12\{\mathcal S,\mathcal S\}=H_{f^2}+fW=0,
 \qquad Q^2=\tfrac12X_{\{\mathcal S,\mathcal S\}}=0.
\]
\end{proof}

\begin{lemma}[Affine cotangent lift]\label{lem:affine-cotangent-q}
Let $Q_{\rm cot}$ be the shifted cotangent lift of a homological vector
field $f$ defining a gauge action on the physical fields.
Let $T_\alpha$ translate the field momenta by a covector $\alpha(q)$
depending only on the ghost-number-zero fields.
If $\alpha$ is covariant and $\alpha(G_v)=0$ for every infinitesimal
gauge transformation $G_v$, then $Q_{\rm cot}+T_\alpha$ is homological.
\end{lemma}
\begin{proof}
We have $Q_{\rm cot}^2=T_\alpha^2=0$.
The commutator on field momenta is the covariance defect of $\alpha$.
On gauge momenta it is the evaluation of $\alpha$ on the gauge action,
and on scalar momenta it vanishes. The hypotheses therefore give
\[
 [Q_{\rm cot},T_\alpha]=0,\qquad
 (Q_{\rm cot}+T_\alpha)^2=0.
\]
\end{proof}

\subsection{The physical variational theory}
\label{subsec:physical-variational}

Let $\mathscr B$ be the formal graded base with coordinates
\[
 q=(\varphi,B,\Phi,A),\qquad \gamma=(\xi,s,\Lambda),\qquad c
\]
of ghost numbers $0,1,2$, respectively.
The ghosts $\xi,\Lambda,c$ are $\mathbb D$-valued and $s$ is
$\mathbb D_0$-valued.
The symmetric isometry $I_g=(g_0^{-1}g)^{-1/2}$ determines the
infinitesimal frame rotation
\[
 j_g(\xi)=I_g^{-1}NI_g
    +I_g^{-1}(DI)_g(N^Tg+gN),\qquad N=\nabla^0\xi.
\]
Set $k_g(\xi)=(j_g(\xi)-\iota_\xi\Theta_0)^{[0]}$.
The relative gauge ghost is
\begin{equation}\label{eq:native-ghost-shift}
 \lambda=s-j_g(\xi)^{[0]}-\iota_{\xi^{[0]}}(A-\Theta_0).
\end{equation}
With $a=A-\Theta_H^{[0]}$, put
\[
 b_\lambda=\tfrac12P(\lambda,F_A)-\tfrac14P(a,D_A\lambda).
\]
The gauge vector field $f$ is defined by
\begin{equation}\label{eq:physical-gauge-action}
\begin{aligned}
 f\varphi&=-\mathcal L_\xi\varphi,
 & f\Phi&=-\xi(\Phi),\\
 fB&=-\mathcal L_\xi B-\dd\Lambda-\jmath_\eps b_\lambda,
 & fA&=-\iota_{\xi^{[0]}}F_A-D_As,\\
 f\xi&=-\tfrac12[\xi,\xi],
 & fs&=\tfrac12[s,s]-\tfrac12F_A(\xi^{[0]},\xi^{[0]}),\\
 f\Lambda&=\dd c-\mathcal L_\xi\Lambda
                    -\tfrac14\jmath_\eps P(\lambda,D_A\lambda),
 & fc&=-\xi(c)+\tfrac1{24}\jmath_\eps P(\lambda,[\lambda,\lambda]).
\end{aligned}
\end{equation}
Curvature evaluation is extended with the ghost coefficients in the
displayed order. The vector field acts from the left with the
horizontal convention \eqref{eq:horizontal-vertical-interface}.

\begin{lemma}\label{lem:native-invariance}
The vector field $f$ satisfies $f^2=0$ and $fW_{\rm ext}=0$.
\end{lemma}
\begin{proof}
The Courant and Green--Schwarz identities give $f^2=0$ on all
coordinates, including $c$, as proved in Appendix~\ref{app:native-closure}.
They also give $fH=-\mathcal L_\xi H$.
For the Dirac density
$\mathcal I=\varrho\,\Bsp(\chi,\slashed D^+\chi)$, covariance yields
\[
 f\mathcal I=-\mathcal L_\xi\mathcal I=-\dd\iota_\xi\mathcal I,
 \qquad fW_{\rm ext}=-\int_Y\dd\iota_\xi\mathcal I=0.
\]
\end{proof}

\begin{theorem}[Variational Hamiltonian theory]\label{thm:main-a}
The graded base $\mathscr B$ defines a local Hamiltonian theory
\[
 \mathcal M_{\rm var}=T^*[-1]\mathscr B,\qquad
 \mathcal S_{\rm var}=W_{\rm ext}+\iota_f\vartheta_{\rm var},\qquad
 \iota_{Q_{\rm var}}\omega_{\rm var}=-\delta\mathcal S_{\rm var},
\]
with
\[
 \{\mathcal S_{\rm var},\mathcal S_{\rm var}\}=0,\qquad
 Q_{\rm var}^2=0.
\]
\end{theorem}
\begin{proof}
The flux and gauge transformations depend on finite jets.
The cotangent form uses the duality of
Definition~\ref{def:mixed-variational-dual}.
Lemma~\ref{lem:native-invariance} verifies the hypotheses of
Lemma~\ref{lem:cotangent-master}.
\end{proof}

Let $(e,z,\tau)$ be the momenta dual to $(q,\gamma,c)$, of ghost
numbers $(-1,-2,-3)$.
Write $G_uq$ for the infinitesimal gauge action, so $fq=-G_\gamma q$,
and put $L_u=D_qG_u$.
Let $X_u$ be the vector component of $u$ and let $b(u,v)$ be the
polarization of the quadratic part of $f\gamma$.
Writing
\[
 \mathfrak b_\gamma=\tfrac12b(\gamma,\gamma),\qquad
 \mathfrak c_\gamma=\tfrac1{24}\jmath_\eps
                         P(\lambda,[\lambda,\lambda]),
\]
define
\[
 \mathfrak b'_{\gamma,q}=D_q\mathfrak b_\gamma,\qquad
 \mathfrak c'_{\gamma,q}=D_q\mathfrak c_\gamma,\qquad
 \mathfrak c'_\gamma(u)=D_\gamma\mathfrak c_\gamma[u].
\]
The ghost variation in the last expression is placed on the right.
Let $\lambda_u$ be the relative parameter obtained from
\eqref{eq:native-ghost-shift} by replacing $\gamma$ with $u$. Then
$\mathfrak c'_\gamma(u)=\jmath_\eps P(\lambda_u,[\lambda,\lambda])/8$.
The momentum components of $Q=Q_{\rm var}$ are
\begin{equation}\label{eq:rv-E-28}
\begin{aligned}
 Qe[V]&=\mathcal E_{\rm var}[V]-e[L_\gamma V]
       -z[\mathfrak b'_{\gamma,q}(V)]+\tau[\mathfrak c'_{\gamma,q}(V)],\\
 Qz[u]&=-e[G_u]-z[b(\gamma,u)]-\tau[X_u(c)]
                                      +\tau[\mathfrak c'_\gamma(u)],\\
 Q\tau[h]&=-z[(0,0,\dd h)]-\tau[\xi(h)].
\end{aligned}
\end{equation}
The covector-first convention gives
$-e[L_{\vartheta u}V]=\vartheta(L_u^\dagger e)[V]$
for an odd coefficient $\vartheta$.
The derivatives and adjoints in \eqref{eq:rv-E-28} are computed in
\eqref{eq:rv-E-26}--\eqref{eq:rv-E-30}.

\subsection{Field charts and formal deformation functors}

\begin{definition}[Hamiltonian $Q$-isomorphism]\label{def:local-hamiltonian-equivalence}
A local Hamiltonian $Q$-isomorphism is an invertible local formal map
$U$ satisfying
\[
 U^*\omega_2=\omega_1,\qquad
 U^*\mathcal S_2=\mathcal S_1,\qquad U_*Q_1=Q_2,
\]
and preserving the inclusion of the physical fields.
\end{definition}

\begin{proposition}\label{prop:native-form-equivalence}
The spinorial and positive-form presentations of the variational theory
are locally Hamiltonian $Q$-isomorphic.
\end{proposition}
\begin{proof}
Let $\Psi$ be the change from weighted-spinor to positive-form
coordinates, including the relative gauge parameter
\eqref{eq:native-ghost-shift}.
The chart and its inverse are constructed in Appendix~\ref{app:momenta}.
Use subscripts $\mathrm{sp}$ and $\varphi$ for the functional, gauge
vector field and cotangent data in the two charts.
The spinor identities and gauge transformations give
\[
 \Psi^*W_\varphi=W_{\rm sp},\qquad
 D\Psi\,f_{\rm sp}=f_\varphi\circ\Psi.
\]
The cotangent lift $\widehat\Psi$ is defined by
$\pi=(D\Psi)^\dagger\Pi$.
The Jacobian and inverse in
\eqref{eq:full-triangular-jacobian}--\eqref{eq:equivalence-momenta}
give, modulo total derivatives,
\[
 \widehat\Psi^*\vartheta_\varphi=\vartheta_{\rm sp},\qquad
 \widehat\Psi^*\omega_\varphi=\omega_{\rm sp},\qquad
 \widehat\Psi^*\mathcal S_\varphi=\mathcal S_{\rm sp}.
\]
Nondegeneracy implies $\widehat\Psi_*Q_{\rm sp}=Q_\varphi$.
Both maps and their adjoints have finite differential order and
preserve the coefficient modules.
\end{proof}

Let $\mathfrak b_0$ have the standard-embedding field coordinates and
zero ghosts and momenta. At this point the BPS residuals vanish, so
\eqref{eq:physical-P-explicit} gives $\dd W_{\rm ext}|_0=0$.
All other components of $Q_{\rm var}$ vanish at zero ghosts and momenta.
If $\ell_n$ are the Taylor brackets at $\mathfrak b_0$, then
\begin{equation}\label{eq:reduced-stationary-points}
 Q_{\rm var}(\mathfrak b_0)=0,\qquad \ell_0=0.
\end{equation}

All dg-Artin algebras below are augmented over $\mathbb R$ and have
nilpotent augmentation ideals.

\begin{definition}[Formal deformation functor]
\label{def:simplicial-mc}\label{def:formal-deformation-functor}
\label{def:variational-deformation-functor}\label{subsec:formal-deformation-functors}
Let $L$ be the Taylor $L_\infty$ algebra at a marked zero.
For such an algebra $R$ with augmentation ideal $\mathfrak m_R$, define
\[
 \Def_L(R)_n=\operatorname{MC}\bigl(
 L\otimes\mathfrak m_R\otimes\Omega^\bullet_{\rm poly}(\Delta^n)\bigr).
\]
Here $\Omega^\bullet_{\rm poly}(\Delta^n)$ is the algebra of polynomial
differential forms on the simplex, and the differential is the sum of
$\ell_1$, $\dd_R$ and the polynomial de Rham differential with the
graded tensor-product signs.
\end{definition}

Write $L_{\rm var}$ for the Taylor algebra of $Q_{\rm var}$,
$C^\bullet_{\rm var}=(L_{\rm var},\ell_1)$ for its tangent complex,
and $\Def_{\rm var}=\Def_{L_{\rm var}}$.

\section{The local BPS--Euler comparison}
\label{sec:local-comparison}

We construct a local differential-operator isomorphism from the
physical BPS equation module
\[
 \mathcal R_{\rm BPS}
 =(\Omega^5\oplus\Omega^7\oplus\Omega^3)\otimes\mathbb D
 \oplus\Omega^6(\mathfrak{so}(E))\otimes\mathbb D_0
\]
to the variational density dual. The isomorphism is defined on
the full module, not merely on the BPS section.
All pairings are understood modulo integration by parts.

Put $\phi=\varphi^{[0]}$, $\psi=*_\phi\phi$,
$w=e^{-2\Phi}$ and $w_0=w^{[0]}$.
The inclusion $\jmath_\eps$ is always applied after reduction
modulo $\eps$.

\subsection{The physical pairing}
\label{subsec:physical-pairing}

The leading flux and Hull connection are
\[
 b=H^{[0]}=\dd(B-B_0)^{[0]},\qquad K=C_{g_\phi}(b),\qquad
 \theta=\Theta_{\LC}(g_\phi)+K.
\]
Thus $\dd b=0$ on $\mathscr V$.
Write $a=A-\theta$ and
$V=(v,\beta,\sigma,k)=(\delta\varphi,\delta B,\delta\Phi,\delta A)$.
The induced connection variation is
\[
 \kappa_V=D_V\theta=\mathsf K_\varphi(v^{[0]})+
 \mathsf K_B(\beta^{[0]}),\qquad \mathsf K_B=C_{g_\phi}\dd.
\]
The operator $\mathsf K_\varphi$ differentiates
$\Theta_{\LC}(g_\phi)+C_{g_\phi}(b)$ in the isometry $I_g$.

Let $\mathscr R_{0,q}$ denote the curvature reconstruction obtained
from the pointwise intrinsic-torsion inverse. Its defining formulas
are \eqref{eq:curvature-equation-reconstruction} and
\eqref{eq:homogeneous-curvature}.
For $E=(E_C,E_\Sigma,E_T,E_I)\in\mathcal R_{\rm BPS}$, define
\[
 \begin{aligned}
 J_\varphi&=\tfrac43\pi_1+\pi_7-\pi_{27},&
 B_\varphi(N)&=\tfrac13(u\mapsto u\wedge\psi_\varphi)^\dagger N,\\
 G_q(E)&=J_\varphi E_T-\tfrac13(*_\varphi E_\Sigma)\varphi
       +\tfrac32\iota_{B_\varphi(w^{-1}E_C)^\sharp}\psi_\varphi,\\
 V_q(E)&=3E_\Sigma-2E_T\wedge\psi_\varphi,\\
 A_q(E)&=\tfrac14a\wedge E_C^{[0]}-\tfrac12w_0E_I,&
 H_q(E)&=\tfrac14a\wedge E_C^{[0]}
       +\tfrac12w_0\mathscr R_{0,q}(E^{[0]})\wedge\psi.
 \end{aligned}
\]
The adjoints in these pointwise formulas use the metric $g_\varphi$.
The equation-to-density pairing is
\begin{equation}\label{eq:rp-P-5}
 \boxed{\begin{aligned}
 4(\mathcal P_qE)[V]=\int_Y\{&
 -\beta\wedge E_C+w\,v\wedge *_\varphi G_q(E)+w\sigma V_q(E)\\
 &+\jmath_\eps[P(k,A_q(E))+P(\kappa_V,H_q(E))]\}.
 \end{aligned}}
\end{equation}

Let $D_{\LC}$ be the exterior covariant derivative induced by
$\nabla^{\LC}=\nabla^{g_\phi}$ on $\mathfrak{so}(TY)$.
For $\eta\in\Omega^3(Y)$, put
\begin{equation}\label{eq:rp-P-1}
 \Gamma_q\eta=
 \left(D_{\LC}C_{g_\phi}(\eta)+\tfrac12[K,C_{g_\phi}(\eta)]
       -\mathscr R_{0,q}(0,0,\eta,0)\right)\wedge\psi.
\end{equation}
Applying the intrinsic-torsion reconstruction to $C^{[0]}$ and
$T^{[0]}-b$ gives $\nabla^{\LC}\phi$.
Levi--Civita pair symmetry and the curvature-difference formula therefore give
\begin{equation}\label{eq:rp-P-3}
 \mathscr R_{0,q}(\mathcal E_{\rm BPS}^{[0]})\wedge\psi+
 \Gamma_qb=R_\theta\wedge\psi.
\end{equation}
Define
\begin{equation}\label{eq:rp-P-4}
 (\mathcal J_q\eta)[V]=-\tfrac18\int_Y
            \jmath_\eps\bigl(w_0P(\kappa_V,\Gamma_q\eta)\bigr).
\end{equation}
Variation of the Green--Schwarz transgression and the superpotential gives
\begin{equation}\label{eq:rp-P-20}
 \boxed{\mathcal E_{\rm var}(q)
 =-\mathcal P_q\mathcal E_{\rm BPS}(q)+\mathcal J_qb.}
\end{equation}
The calculation, including the sign of the Hull contribution, is given
in Appendix~\ref{app:first-variation}.

\subsection{A two-sided inverse of the physical pairing}
\label{subsec:unrestricted-pairing-inverse}

Represent a density covector by $(f_B,f_\varphi,f_\Phi,f_A)$, with
evaluation $\int_Y(\beta\wedge f_B+v\wedge f_\varphi+
\sigma f_\Phi+\jmath_\eps P(k,f_A))$. The gauge covector is $\jmath_\eps f_A$, with
$f_A$ valued in $\mathbb D_0$. Define the density adjoints by
\[
 \int_Y P(\mathsf K_B\beta,N)=\int_Y\beta\wedge\mathsf K_B^\top N,
 \qquad
 \int_Y P(\mathsf K_\varphi v,N)=\int_Yv\wedge\mathsf K_\varphi^\top N.
\]
Integration by parts in \eqref{eq:rp-P-5} gives
\begin{equation}\label{eq:rp-P-16}
 \begin{aligned}
 (4\mathcal P_qE)_B&=-E_C+\jmath_\eps\mathsf K_B^\top H_q(E),\\
 (4\mathcal P_qE)_\varphi&=w*_\varphi G_q(E)+\jmath_\eps\mathsf K_\varphi^\top H_q(E),\\
 (4\mathcal P_qE)_\Phi&=wV_q(E),\\
 (4\mathcal P_qE)_A&=\tfrac14a\wedge E_C^{[0]}-\tfrac12w_0E_I.
 \end{aligned}
\end{equation}
Set
\[
 \mathcal S_qf=(f_B,f_\varphi,f_\Phi,f_A+a\wedge f_B^{[0]}/4)
\]
and
\[
 \mathcal D_qE=(-E_C,w*_\varphi G_q(E),wV_q(E),-\tfrac12w_0E_I).
\]
To compute $E=\mathcal D_q^{-1}F$, write
$F=(F_B,F_\varphi,F_\Phi,F_A)$ and set
\[
 E_C=-F_B,\quad
 U=w^{-1}*_\varphi F_\varphi-\tfrac32
 \iota_{B_\varphi(w^{-1}E_C)^\sharp}\psi_\varphi,\quad
 a_s=\tfrac17\langle U,\varphi\rangle,\quad r_s=*_\varphi(w^{-1}F_\Phi).
\]
The remaining components are
\[
 \begin{aligned}
 E_T&=-\tfrac12(r_s+9a_s)\varphi+\pi_7U-\pi_{27}U,\\
 E_\Sigma&=(-2r_s-21a_s)\vol_\varphi,
 &E_I&=-2w_0^{-1}F_A.
 \end{aligned}
\]
For $t_s=\langle E_T,\varphi\rangle/7$ and
$\sigma_s=*_\varphi E_\Sigma$, the scalar equations are
\[
 3a_s=4t_s-\sigma_s,\qquad r_s=3\sigma_s-14t_s.
\]
The eigenvalues on $\Lambda^3_7$ and $\Lambda^3_{27}$ are $1$ and $-1$.
Thus $\mathcal D_q$ is pointwise invertible.

Put
\[
 \mathcal C_q^P E=(\jmath_\eps\mathsf K_B^\top H_q(E),
 \jmath_\eps\mathsf K_\varphi^\top H_q(E),0,0),\qquad
 \mathcal U_q=\mathcal D_q^{-1}\mathcal C_q^P.
\]
Since $\mathcal C_q^PE$ has zero reduction and zero gauge component,
$\mathcal S_q\mathcal C_q^P=\mathcal C_q^P$.
The gauge term $a\wedge E_C^{[0]}/4$ cancels under $\mathcal S_q$, giving
\begin{equation}\label{eq:rp-P-17}
 \begin{gathered}
 \mathcal S_q(4\mathcal P_q)=\mathcal D_q(1+\mathcal U_q),
 \qquad \mathcal U_q^2=0,\\
 \boxed{\mathcal Q_q=4(1-\mathcal U_q)\mathcal D_q^{-1}\mathcal S_q
              =\mathcal P_q^{-1}.}
 \end{gathered}
\end{equation}
The image of $\mathcal U_q$ lies in $\eps\mathcal R_{\rm BPS}$,
whereas $H_q$ depends only on the reduction of the geometric components.
Therefore $\mathcal U_q^2=0$ and
\[
 (1-\mathcal U_q)(1+\mathcal U_q)
 =(1+\mathcal U_q)(1-\mathcal U_q)=1.
\]

\subsection{Opposite torsion and the local multiplier}

Let $p_q(E)\in\Omega^1(Y,\Lambda^3T^*Y)$ be the intrinsic-torsion
reconstruction of \eqref{eq:curvature-equation-reconstruction},
evaluated on $E_C^{[0]}$ and $E_T^{[0]}$.
Define $\Xi\in\Omega^1(Y,\Lambda^3T^*Y)$ by
$\Xi(X)=\iota_X\psi$, and put
\[
 f_q(E)=\langle E_T^{[0]},\phi\rangle-*_\phi E_\Sigma^{[0]},
 \qquad
 \widetilde p_q(E)=p_q(E)-\tfrac14 f_q(E)\Xi,
 \qquad
 \nabla^-=\nabla^{\LC}-K.
\]

For $\eta\in\Omega^3(Y)$, define
\[
 h_\eta(X,Y)=\langle\iota_X\eta,\iota_Y\phi\rangle.
\]
The intrinsic-torsion reconstruction of Appendix~
\ref{app:curvature-factorization} gives
\begin{equation}\label{eq:pure-torsion-reconstruction}
 t_q(0,0,\eta,0)
 =-\tfrac12 h_\eta
  +\tfrac14\langle\eta,\phi\rangle g_\phi.
\end{equation}
For the induced action $\rho$ of $\mathfrak{so}(TY)$ on forms,
\[
 \rho(C_{g_\phi}(\eta)_X)\phi
 =\tfrac12\iota_{h_\eta(X,\cdot)^\sharp}\psi.
\]
Using
\[
 \langle*\dd\phi,\phi\rangle=*(\dd\phi\wedge\phi),
 \qquad
 \langle*(\alpha\wedge\phi),\phi\rangle=0
 \quad\text{for }\alpha\in\Omega^1(Y),
\]
we obtain
\begin{equation}\label{eq:opposite-torsion-reconstruction}
 \begin{aligned}
 p_q(\mathcal E_{\rm BPS})
   &=\nabla^-\phi+\tfrac14\langle b,\phi\rangle\Xi,\\
 f_q(\mathcal E_{\rm BPS})
   &=\langle b,\phi\rangle,\\
 \widetilde p_q(\mathcal E_{\rm BPS})
   &=\nabla^-\phi.
 \end{aligned}
\end{equation}

Identify $Z\in\mathfrak{so}(TY)$ with the two-form
$(U,V)\mapsto g_\phi(U,ZV)$.
Curvature is then a section of
$\Lambda^2T^*Y\otimes\Lambda^2T^*Y$.
Let $\mathsf T$ interchange these factors and let $R^-$ be the curvature
of $\nabla^-$.
We view a four-form as a double two-form by evaluating it on the two
pairs of vectors.
The curvature identities of Appendix~\ref{app:opposite-curvature}
give
\begin{equation}\label{eq:opposite-curvature-exchange}
 R_\theta-\mathsf T R^-
 =\tfrac12\dd b=0.
\end{equation}
Here the quadratic contorsion term $K\wedge K$ is invariant
under $\mathsf T$, and $\dd b=0$ holds throughout the ambient
field space.

On $\Lambda^3T^*Y$-valued forms, the exterior covariant
derivative satisfies
\[
 \dd_{\nabla^-}
 =\dd_{\nabla^{\LC}}-\rho(K)\wedge.
\]
Define the contraction
\[
 \mathcal C_\phi(u)=\tfrac1{12}\iota_{\Xi^\dagger u}\phi,
 \qquad
 \mathcal C_\phi:\Lambda^3T^*Y\longrightarrow\Lambda^2_7T^*Y,
\]
where $\Xi^\dagger$ is the metric adjoint of
$\Xi:TY\to\Lambda^3T^*Y$.
Set
\begin{equation}\label{eq:local-curvature-replacement}
 \begin{aligned}
 \widehat r_q(E)
   &=\mathsf T\mathcal C_\phi
       \bigl(\dd_{\nabla^-}\widetilde p_q(E)\bigr)
       \wedge\psi,\\
 \Delta_q(E)
   &=\widehat r_q(E)
     -\mathscr R_{0,q}(E^{[0]})\wedge\psi.
 \end{aligned}
\end{equation}
Since $\mathcal C_\phi(\rho(Z)\phi)=\pi_7Z$ and
$\dd_{\nabla^-}(\nabla^-\phi)=\rho(R^-)\phi$,
equations \eqref{eq:opposite-torsion-reconstruction} and
\eqref{eq:opposite-curvature-exchange} give
\[
 \begin{aligned}
 \widehat r_q(\mathcal E_{\rm BPS})
 &=\mathsf T\bigl((1\otimes\pi_7)R^-\bigr)\wedge\psi\\
 &=(\pi_7\otimes1)R_\theta\wedge\psi\\
 &=R_\theta\wedge\psi.
 \end{aligned}
\]
The last equality follows from
$\Lambda^2_{14}\wedge\psi=0$.
Together with \eqref{eq:rp-P-3}, this proves
\[
 \Delta_q(\mathcal E_{\rm BPS})=\Gamma_qb.
\]

Define
\begin{equation}\label{eq:local-correction}
 \boxed{
 (\mathcal L_qE)[V]
 =-\tfrac18\int_Y
 \jmath_\eps\bigl(w_0P(\kappa_V,\Delta_q(E))\bigr).
 }
\end{equation}
Expanding the covariant derivatives gives
\begin{equation}\label{eq:local-correction-expanded}
 \begin{aligned}
 \Delta_q(E)
 =\mathsf T\mathcal C_\phi\Bigl(
   &-\rho(K)\wedge p_q(E)\\
   &-\tfrac14\dd_{\nabla^{\LC}}(f_q(E)\Xi)\\
   &+\tfrac14 f_q(E)\rho(K)\wedge\Xi
 \Bigr)\wedge\psi.
 \end{aligned}
\end{equation}
Thus $\Delta_q$ has differential order at most one in $E$,
and $\mathcal L_q$ has order at most two after integration
by parts.

\begin{theorem}[Local BPS--Euler comparison]
\label{thm:local-bps-euler}
For every $q\in\mathscr V$, the operator
$M_q=\mathcal P_q-\mathcal L_q$ satisfies
\begin{equation}\label{eq:local-multiplier}
 \boxed{
 \mathcal E_{\rm var}=-M_q\mathcal E_{\rm BPS},
 \qquad
 M_q^{-1}=(1+\mathcal Q_q\mathcal L_q)\mathcal Q_q.
 }
\end{equation}
The inverse is two-sided on $\mathcal R_{\rm BPS}$, and both
operators have finite differential order.
\end{theorem}

\begin{proof}
The curvature reconstruction gives
$\mathcal L_q\mathcal E_{\rm BPS}=\mathcal J_qb$.
Equation \eqref{eq:rp-P-20} therefore implies
$\mathcal E_{\rm var}=-M_q\mathcal E_{\rm BPS}$.

Since $\kappa_V$ is independent of $\delta A$,
\[
 \im\mathcal L_q\subset\eps(T_q\mathscr V)^\vee.
\]
The $\mathbb D$-linearity of $\mathcal Q_q$ gives
\[
 N_q:=\mathcal Q_q\mathcal L_q,\qquad
 \im N_q\subset\eps\mathcal R_{\rm BPS},\qquad
 N_q(\eps\mathcal R_{\rm BPS})
   =\eps N_q(\mathcal R_{\rm BPS})=0.
\]
Hence
\[
 N_q^2=0,\qquad
 M_q=\mathcal P_q(1-N_q),\qquad
 M_q^{-1}=(1+N_q)\mathcal Q_q.
\]
Multiplication verifies both inverse identities.

The pointwise inverses have invertible reductions, and all
operator coefficients depend on finite jets of $q$.
The identities therefore extend to the local formal completion.
\end{proof}

\subsection{Specialization at the standard embedding}
\label{subsec:local-unary}

At the torsion-free standard embedding, $K=0$ and
$\nabla^{\LC}\psi_0=0$.
Define
\[
 W:\Omega^1(Y)\longrightarrow
   \Omega^6(Y,\mathfrak{so}(TY))
\]
by
\[
 g_0\bigl(U,(*_0W(\nu))(X)V\bigr)
 =\tfrac14(X^\flat\wedge\nu)(U,V),
\]
where $X^\flat$ is taken with respect to $g_0$.
The contraction in Appendix~\ref{app:opposite-curvature}
gives
\begin{equation}\label{eq:linear-corrected-pairing}
 \begin{gathered}
 \Delta_0(E)=W\dd f_0(E),
 \qquad
 M_0=\mathcal P_0-\mathcal L_0,\\
 (\mathcal L_0E)[V]
 =-\tfrac18\int_Y
   \jmath_\eps\bigl(
     w_0P(\kappa_V,W\dd f_0(E))
   \bigr).
 \end{gathered}
\end{equation}
Equations \eqref{eq:linear-corrected-pairing} hold on
$\mathcal R_{\rm BPS}$.

\section{The physical BPS--Noether theory}
\label{sec:bps-noether}

The pairing $M_q$ identifies the BPS equation module with the density
covectors dual to the physical fields. We combine this pairing with
the cotangent lift of the gauge action to construct a formal graded
theory whose equation section is $\mathcal E_{\rm BPS}$.

\subsection{The physical covector and its Noether identities}

Let $f$ be the homological vector field on the base $\mathscr B$ of
Section~\ref{sec:equivalence}, with coordinates
$q=(\varphi,B,\Phi,A)$, $\gamma=(\xi,s,\Lambda)$ and $c$.
For a field variation $V$, define
\begin{equation}\label{eq:noether-physical-covector}
 \alpha_q[V]=\bigl(M_q\mathcal E_{\rm BPS}(q)\bigr)[V].
\end{equation}
We use the notation $\phi=\varphi^{[0]}$, $\psi=*_{\phi}\phi$,
$w_0=e^{-2\Phi^{[0]}}$, $\theta=\Theta_H^{[0]}$ and $a=A-\theta$.
For $E\in\mathcal R_{\rm BPS}$, put
\begin{equation}\label{eq:noether-bianchi}
\begin{aligned}
 J_A(E)&=D_A(w_0E_I)-F_A\wedge E_C^{[0]},\\
 \widehat K_\theta(E)&=D_\theta(w_0\widehat r_q(E))
                         -R_\theta\wedge E_C^{[0]},\\
 \widehat Z(E)&=\tfrac12\bigl(\widehat K_\theta(E)-J_A(E)
                                    -a\wedge\dd E_C^{[0]}\bigr).
\end{aligned}
\end{equation}
Since $\widehat r_q(\mathcal E_{\rm BPS})=R_\theta\wedge\psi$, the
Bianchi identities give
\[
 \begin{aligned}
 D_A(w_0I)&=F_A\wedge\dd(w_0\psi)=F_A\wedge C^{[0]},\\
 D_\theta\bigl(w_0\widehat r_q(\mathcal E_{\rm BPS})\bigr)
   &=R_\theta\wedge\dd(w_0\psi)=R_\theta\wedge C^{[0]}.
 \end{aligned}
\]
Consequently,
\begin{equation}\label{eq:noether-actual-bianchi}
 \dd C=0,\qquad
 J_A(\mathcal E_{\rm BPS})=0,\qquad
 \widehat K_\theta(\mathcal E_{\rm BPS})=0,\qquad
 \widehat Z(\mathcal E_{\rm BPS})=0.
\end{equation}

\begin{lemma}\label{lem:physical-noether-covector}
The covector $\alpha$ is covariant under the physical gauge action and
satisfies
\[
 \alpha_q[G_vq]=0
\]
for every gauge parameter $v$.
\end{lemma}
\begin{proof}
Write $G_\Lambda$, $G_\zeta$ and $G_X$ for the gerbe, relative gauge
and diffeomorphism generators. Substitution of their actions into
$4\alpha$, using the Green--Schwarz transgression, gives
\[
\begin{aligned}
 4\alpha[G_\Lambda]
   &=-\int_Y\Lambda\wedge\dd C,\\
 4\alpha[G_\zeta]
   &=\tfrac12\int_Y\jmath_\eps
       P\bigl(\zeta,J_A(\mathcal E_{\rm BPS})\bigr),\\
 4\alpha[G_X]
   &=-\int_Y\jmath_\eps
       P\bigl(k_g(X),\widehat Z(\mathcal E_{\rm BPS})\bigr)\\
   &\quad+\int_Y\dd\iota_X
       \bigl(w(H\wedge\psi_\varphi
                       -\tfrac12\dd\varphi\wedge\varphi)\bigr).
\end{aligned}
\]
Here $k_g$ is the frame rotation defined in
Section~\ref{subsec:physical-variational}.
Equation~\eqref{eq:noether-actual-bianchi} and Stokes' theorem imply
$\alpha[G_vq]=0$.

For $V=(v,\beta,\sigma,k)$, put
$t_{\xi,V}=D_gk_g(\xi)[D_\phi g(v^{[0]})]$.
The covariance defects of $4M_q$ under diffeomorphisms and relative
gauge transformations are
\begin{equation}\label{eq:noether-independent-defects}
 -\int_Y\jmath_\eps P(t_{\xi,V},\widehat Z(E)),
 \qquad
 -\tfrac12\int_Y\jmath_\eps P(k,\zeta\,\dd E_C^{[0]}),
\end{equation}
as computed in Appendix~\ref{app:rp-independent}.
Both vanish at $E=\mathcal E_{\rm BPS}$ by
\eqref{eq:noether-actual-bianchi}. The gerbe action has zero
covariance defect, so $\alpha$ is covariant.
\end{proof}

\subsection{The Noether vector field}

On $T^*[-1]\mathscr B$, let $(u,m,t)$ be the density covectors dual
to $(q,\gamma,c)$, of ghost numbers $(-1,-2,-3)$.
Let $Q_{\rm cot}$ be the shifted cotangent lift of $f$ and define
the vertical vector field $T_\alpha$ by
\[
 T_\alpha u=\alpha,\qquad
 T_\alpha q=T_\alpha\gamma=T_\alpha c
             =T_\alpha m=T_\alpha t=0.
\]

\begin{definition}[Physical BPS--Noether enhancement]
\label{def:physical-noether}
The physical BPS--Noether enhancement is the formal graded space
$\mathcal M_{\mathrm{BPS},N}=T^*[-1]\mathscr B$ with vector field
$Q_{\mathrm{BPS},N}=Q_{\rm cot}+T_\alpha$.
Writing $Q=Q_{\mathrm{BPS},N}$, its action on density covectors is
\begin{equation}\label{eq:noether-density-q}
\begin{aligned}
 Qu[V]&=\alpha[V]-u[L_\gamma V]
       -m[\mathfrak b'_{\gamma,q}(V)]+t[\mathfrak c'_{\gamma,q}(V)],\\
 Qm[v]&=-u[G_v]-m[b(\gamma,v)]-t[X_v(c)]
                                  +t[\mathfrak c'_\gamma(v)],\\
 Qt[h]&=-m[(0,0,\dd h)]-t[\xi(h)].
\end{aligned}
\end{equation}
The operators $L_\gamma$, $b$, $\mathfrak b'_{\gamma,q}$,
$\mathfrak c'_{\gamma,q}$ and $\mathfrak c'_\gamma$ are those of
\eqref{eq:rv-E-28}. The BPS equation coordinate $r$ is defined by
\[
 u=M_qr,\qquad r\in\mathcal R_{\rm BPS}.
\]
We denote the associated simplicial deformation functor by
$\Def_{\mathrm{BPS},N}$.
\end{definition}

This construction uses the identities dual to the gauge action.
It need not resolve all differential relations among the BPS equations.

\begin{proposition}\label{prop:physical-noether-square}
The vector field $Q_{\mathrm{BPS},N}$ is homological. In the coordinates
$(q,\gamma,c,r,m,t)$, its action on $r,m,t$ is
\begin{equation}\label{eq:physical-noether-q}
\begin{aligned}
 Qr={}&\mathcal E_{\rm BPS}
 +M_q^{-1}\bigl(
 -(M_qr)\circ L_\gamma
 -m\circ\mathfrak b'_{\gamma,q}
 +t\circ\mathfrak c'_{\gamma,q}
 -(D_qM_q[fq])r\bigr),\\
 Qm[v]={}&-(M_qr)[G_v]-m[b(\gamma,v)]-t[X_v(c)]
                                  +t[\mathfrak c'_\gamma(v)],\\
 Qt[h]={}&-m[(0,0,\dd h)]-t[\xi(h)].
\end{aligned}
\end{equation}
In particular,
$Qr|_{\gamma=c=r=m=t=0}=\mathcal E_{\rm BPS}(q)$.
\end{proposition}
\begin{proof}
Lemma~\ref{lem:physical-noether-covector} verifies the hypotheses of
Lemma~\ref{lem:affine-cotangent-q}. Thus
\[
 Q_{\mathrm{BPS},N}^2
 =Q_{\rm cot}^2+[Q_{\rm cot},T_\alpha]+T_\alpha^2=0.
\]
The coordinate change $u=M_qr$ is invertible by
Theorem~\ref{thm:local-bps-euler}. Its chain rule is
\[
 Qu=(D_qM_q[fq])r+M_qQr.
\]
Substitution of \eqref{eq:noether-density-q} gives
\eqref{eq:physical-noether-q}.
\end{proof}

\subsection{Local equivalence and linearization}

\begin{theorem}[Local BPS--Noether equivalence]
\label{thm:bps-noether-equivalence}
The map
\begin{equation}\label{eq:noether-local-isomorphism}
 U_N(q,\gamma,c,r,m,t)
       =(q,\gamma,c,-M_qr,-m,-t)
\end{equation}
is a local formal $Q$-isomorphism
\[
 U_N:\mathcal M_{\mathrm{BPS},N}\xrightarrow{\sim}
      \mathcal M_{\rm var}.
\]
It fixes $(q,\gamma,c)$ and has a finite-order differential inverse
in the density coordinates. For every augmented real dg-Artin
algebra $R$, it induces a natural isomorphism
\[
 \Def_{\mathrm{BPS},N}(R)\cong\Def_{\rm var}(R).
\]
\end{theorem}
\begin{proof}
Theorem~\ref{thm:local-bps-euler} gives
\[
 \alpha=\dd W_{\rm ext}=-\mathcal E_{\rm var}.
\]
Under $(e,z,\tau)=(-u,-m,-t)$, equations
\eqref{eq:noether-density-q} become \eqref{eq:rv-E-28}.
Together with $u=M_qr$, this proves
\[
 DU_N\,Q_{\mathrm{BPS},N}=Q_{\rm var}\circ U_N,
 \qquad
 U_N^{-1}(q,\gamma,c,e,z,\tau)
   =(q,\gamma,c,-M_q^{-1}e,-z,-\tau).
\]
The locality and formal dependence of both maps follow from
Theorem~\ref{thm:local-bps-euler}.

Put $R_n=R\otimes\Omega^\bullet_{\rm poly}(\Delta^n)$ and let
$d_n$ be its tensor-product differential. The chain rule gives
\[
\begin{aligned}
 d_nU_N(x)&=DU_N(x)d_nx,\\
 d_nU_N(x)+Q_{\rm var}(U_N(x))
   &=DU_N(x)\bigl(d_nx+Q_{\mathrm{BPS},N}(x)\bigr).
\end{aligned}
\]
Taylor evaluation is finite on
$\mathfrak m_R\otimes\Omega^\bullet_{\rm poly}(\Delta^n)$.
The same identities hold for $U_N^{-1}$, giving inverse maps on
Maurer--Cartan solutions. Both maps commute with morphisms of $R$
and pullbacks of polynomial forms on simplices.
\end{proof}

\begin{proposition}[Linearized BPS--Noether comparison]
\label{prop:linear-bps-noether-comparison}\label{prop:derived-hull-graph}
Let $E=D\mathcal E_{\rm BPS}|_0$ and
$\mathcal H=D^2W_{\rm ext}|_0$. Then
\[
 \mathcal H=M_0E.
\]
With $\mathsf R$ and $G$ the scalar-reducibility and gauge maps at
the background, the BPS--Noether tangent complex is
\begin{equation}\label{eq:noether-unary-complex}
 C_{\rm var}^{-1}\xrightarrow{\mathsf R}C_{\rm var}^0
 \xrightarrow{G}C_{\rm var}^1\xrightarrow{E}\mathcal R_{\rm BPS}
 \xrightarrow{G^\dagger M_0}(C_{\rm var}^0)^\vee
 \xrightarrow{-\mathsf R^\dagger}(C_{\rm var}^{-1})^\vee.
\end{equation}
The degreewise maps $(1,1,1,-M_0,-1,-1)$ give a differential-operator
chain isomorphism with $C_{\rm var}^\bullet$.
There is also a differential-operator chain isomorphism with the
Hull subcomplex of the three-row heterotic complex of
\cite{QuantumG2}. This subcomplex imposes the induced tangent-connection
variation and its corresponding relations in the equation and
identity terms, as defined in Appendix~\ref{app:independent-hull-graph}.
Both chain isomorphisms have finite-order differential inverses.
\end{proposition}
\begin{proof}
Since $\mathcal E_{\rm BPS}(0)=0$, differentiation of
$\dd W_{\rm ext}=M\mathcal E_{\rm BPS}$ gives
\[
 \mathcal H=M_0E.
\]
The operator $M_0$ is given by \eqref{eq:linear-corrected-pairing}.
The gauge and Noether identities give
\[
 G\mathsf R=0,\qquad EG=0,\qquad
 G^\dagger M_0E=G^\dagger\mathcal H=0,\qquad
 \mathsf R^\dagger G^\dagger=(G\mathsf R)^\dagger=0.
\]
The chain isomorphism with $C_{\rm var}^\bullet$ is the derivative
of $U_N$.
For the Hull subcomplex, Appendix~\ref{app:independent-hull-graph}
constructs maps $J_p$ with finite-order inverses. With
$d_{\mathrm{BPS},N}$ the differential in
\eqref{eq:noether-unary-complex}, they satisfy
\[
 d_{\rm ind}J_p=J_{p+1}d_{\mathrm{BPS},N}
 \qquad(-1\le p\le3).
\]
These maps include the tangent Bianchi identity in degree three.
\end{proof}

\subsection{Shifted symplectic structure}

\begin{corollary}\label{cor:noether-hamiltonian}
The forms and functional
\[
 \omega_N=U_N^*\omega_{\rm var},\qquad
 \mathcal S_N=U_N^*\mathcal S_{\rm var}
\]
define a local Hamiltonian structure on $\mathcal M_{\mathrm{BPS},N}$.
The form $\omega_N$ is mixed symplectic of ghost number $-1$, and
\[
 \iota_{Q_{\mathrm{BPS},N}}\omega_N=-\delta\mathcal S_N,
 \qquad \{\mathcal S_N,\mathcal S_N\}_N=0.
\]
For another locally invertible pairing $M'$ satisfying
$M'\mathcal E_{\rm BPS}=\alpha$, the coordinate change
\begin{equation}\label{eq:noether-choice-equivalence}
 T_{M,M'}:\quad
 r'=M'^{-1}Mr,\qquad m'=m,\qquad t'=t
\end{equation}
fixes $(q,\gamma,c)$ and identifies the pulled-back symplectic forms
and Hamiltonians. These maps satisfy
$T_{M',M''}\circ T_{M,M'}=T_{M,M''}$.
\end{corollary}
\begin{proof}
Pullback by the $Q$-isomorphism $U_N$ gives
\[
 \iota_{Q_{\mathrm{BPS},N}}\omega_N
 =U_N^*\bigl(\iota_{Q_{\rm var}}\omega_{\rm var}\bigr)
 =-\delta\mathcal S_N,
 \qquad
 \{\mathcal S_N,\mathcal S_N\}_N
 =U_N^*\{\mathcal S_{\rm var},\mathcal S_{\rm var}\}=0.
\]
The inverse of $U_N$ preserves nondegeneracy.
Write $U_M$ for \eqref{eq:noether-local-isomorphism} with pairing $M$.
Then
\[
 U_{M'}\circ T_{M,M'}=U_M,
 \qquad M''^{-1}M'M'^{-1}M=M''^{-1}M.
\]
The first identity proves preservation of the Hamiltonian structures.
The second proves the cocycle identity.
\end{proof}

In the BPS coordinate $r$, the symplectic form includes the field
derivative of $M_q$ through
$\delta(M_qr)=M_q\delta r+(D_qM_q[\delta q])r$.

\section{Ellipticity and finite-dimensional cyclic models}
\label{sec:finite-moduli}\label{sec:physical-minimal}

Throughout this section, $Y$ is smooth and closed and the background is
the torsion-free standard embedding with constant dilaton. The gerbe and
anomaly data remain fixed. We prove ellipticity of the tangent complex
and transfer the variational $L_\infty$ structure to its cohomology.

\subsection{The variational tangent complex}
\label{subsec:variational-unary}

The tangent complex $C^\bullet_{\rm var}=(L_{\rm var},\ell_1)$ has
terms
\[
\begin{aligned}
 C^{-1}_{\rm var}&=\Omega^0\otimes\mathbb D,\\
 C^0_{\rm var}&=(\Omega^1\oplus\Omega^1)\otimes\mathbb D
              \oplus\Omega^0(\mathfrak{so}(TY))\otimes\mathbb D_0,\\
 C^1_{\rm var}&=(\Omega^2_7\oplus\Omega^1(T^*Y)\oplus\Omega^0)
                 \otimes\mathbb D
              \oplus\Omega^1(\mathfrak{so}(TY))\otimes\mathbb D_0,\\
 C^{3-p}_{\rm var}&=(C^p_{\rm var})^\vee
                    \qquad(p=-1,0,1).
\end{aligned}
\]
Write $H^p_{\rm var}=H^p(C^\bullet_{\rm var})$.
In the spinor chart, a field variation is $x=(z,M,\alpha,l)$.
The components $z$, $\alpha$ and $l$ vary the spinor line, gauge
connection and logarithmic density. The symmetric and skew parts of
$M$ give the metric and two-form variations.

Put $m=M^{[0]}$ and $\nabla=\nabla^{\LC}_{g_0}$.
For $m^T(U,V)=m(V,U)$, the linearized Hull connection is
\begin{equation}\label{eq:linear-hull-graph}
\begin{aligned}
 g_0\bigl(U,(\KHull m)(X)V\bigr)
 ={}&(\nabla_Vm)(X,U)-(\nabla_Um)(X,V)\\
    &+\tfrac12\bigl(\nabla_X(m-m^T)\bigr)(U,V).
\end{aligned}
\end{equation}
Let $C_{G_2}:\Lambda^2_7\to TY$ satisfy
$C_{G_2}(\iota_v\varphi_0)=3v$, and let $\mathsf r$ be the derivative
of the pullback action of $\mathrm{GL}(TY)$ on forms.
The map to positive-form coordinates is
\begin{equation}\label{eq:linear-field-map}
\begin{gathered}
 h_g=M+M^T,\qquad \beta=M-M^T,\qquad
 v=C_{G_2}(z+\tfrac12\pi_7\beta),\\
 u=\delta\varphi=\mathsf r(g_0^{-1}h_g/2)\varphi_0+\iota_v\psi_0,
 \qquad p=\delta\Phi=\tfrac14\tr_{g_0}h_g-\tfrac12l,\\
 a=\alpha-\KHull M^{[0]}.
\end{gathered}
\end{equation}
Its inverse has differential order at most one and is given in
Appendix~\ref{app:elliptic-cyclic}.

The derivative of $\varphi\mapsto *_\varphi\varphi$ at $\varphi_0$ is
\[
 \mathcal J_{\varphi_0}=*_0J,\qquad
 J=J_{\varphi_0}=\tfrac43\pi_1+\pi_7-\pi_{27}.
\]
With $w=e^{-2\Phi_0}$, set
\begin{equation}\label{eq:linear-four-form}
 \dot\psi_w=\mathcal J_{\varphi_0}u-2p\psi_0,
 \qquad k=\KHull M^{[0]}.
\end{equation}
The quadratic part of the superpotential is
\begin{equation}\label{eq:quadratic-native-source}
\begin{aligned}
 W_{\mathrm{ext},2}(x)
 ={}&\frac w4\int_Y
 \bigl(\dd\beta\wedge\dot\psi_w+p\,\dd u\wedge\varphi_0
                                    -\tfrac12\dd u\wedge u\bigr)\\
 &-\frac{\eps w}{16}\int_Y
 \bigl\{2P(\alpha-k,R_0)\wedge\dot\psi_w
       +\bigl[P(\alpha,D\alpha)-P(k,Dk)\bigr]\wedge\psi_0\bigr\},
\end{aligned}
\end{equation}
where $D=D_{\Theta_0}$ and the coefficient of $\eps$ is evaluated
modulo $\eps$. Since $\dd W_{\rm ext}|_0=0$, the Hessian transforms
by the linear field map \eqref{eq:linear-field-map}.

Let $\mathsf R$ be the scalar-reducibility map and $G$ the negative
linearization of the gauge action $fq$. The Hessian is characterized by
\[
 (\mathcal Hx_1)[x_2]=D^2W_{\rm ext}|_0(x_1,x_2).
\]
The cochain pairing is
\[
 B_C(x,y)=(-1)^{|x|}\omega_{\rm var}(\mathbf sx,\mathbf sy),\qquad
 B=\tau_{\mathbb D}B_C.
\]
We use $\dagger$ for the formal adjoint under the density pairing and
$\sharp$ for the cyclic adjoint, characterized by
\[
 B(Ax,y)=(-1)^{|A||x|}B(x,A^\sharp y).
\]

\begin{proposition}[Variational tangent complex]
\label{prop:unary-variational}
The differential is
\begin{equation}\label{eq:unary-differential}
 d_C^{-1}=\mathsf R,\quad d_C^0=G,\quad d_C^1=-\mathcal H,\quad
 d_C^2=-G^\sharp,\quad d_C^3=\mathsf R^\sharp,\quad d_C^4=0.
\end{equation}
It satisfies
\[
 (d_C)^2=0,\qquad
 B_C(d_Cx,y)+(-1)^{|x|}B_C(x,d_Cy)=0.
\]
The derivative of the Hamiltonian coordinate change in
Proposition~\ref{prop:native-form-equivalence} is a local
$\mathbb D$-linear chain isomorphism between the spinorial and
positive-form complexes.
\end{proposition}
\begin{proof}
Scalar reducibility gives
\[
 \mathsf Rc=(-\dd c/\sqrt2,\dd c/\sqrt2,0),\qquad
 (\xi,\Lambda,\lambda)=(0,\dd c,0),\qquad G\mathsf R=0.
\]
Differentiating gauge invariance at the stationary background yields
\[
 0=D_q\bigl(\dd W_{\rm ext}(q)[G_\gamma q]\bigr)_0[x]
   =D^2W_{\rm ext}|_0(x,G\gamma),
 \qquad \mathcal HG=0.
\]
Hessian symmetry and formal adjunction give the remaining compositions.
The cotangent sign convention gives
\[
 B_C(x_p,y_{3-p})=(-1)^{p+1}y_{3-p}[x_p]
 \qquad(p=-1,0,1).
\]
Thus the upper maps in density coordinates are
$G^\dagger$ and $-\mathsf R^\dagger$, and the displayed cyclicity
identity follows from adjunction.
Differentiating the Hamiltonian $Q$-isomorphism and its inverse proves
the chain-isomorphism assertion.
\end{proof}

\subsection{Mixed-order ellipticity and Hodge contraction}
\label{subsec:variational-elliptic}\label{subsec:common-elliptic-contractions}

The variation $\KHull$ has order one. Its contribution to the Hessian,
followed by its formal adjoint, has order at most three. The coefficient
filtration separates these higher-order terms from the first-order
geometric and gauge operators.

Write $C^\bullet_{\rm var}=\Gamma(Y,\mathscr E^\bullet_{\rm var})$.
Multiplication by $\eps$ defines the filtration
\[
 0\subset\eps\mathscr E_{\rm var}\subset\ker\eps
 \subset\mathscr E_{\rm var}.
\]
Set
\[
 \mathsf G_0=\mathscr E_{\rm var}/\ker\eps,\qquad
 \mathsf A=\ker\eps/\eps\mathscr E_{\rm var},\qquad
 \mathsf G_1=\eps\mathscr E_{\rm var}.
\]
Assign weights in cochain degree $p$ by
\begin{equation}\label{eq:unary-spatial-weights}
 s_p(\mathsf G_0)=p-1,\qquad
 s_p(\mathsf A)=p,\qquad
 s_p(\mathsf G_1)=p+1.
\end{equation}
In a local splitting of this filtration, the component
$d_p^{T,S}:S^p\to T^{p+1}$ satisfies
\[
 \operatorname{ord}(d_p^{T,S})\le s_{p+1}(T)-s_p(S).
\]
The weighted principal symbol consists of the homogeneous terms of
these degrees. The order bounds are proved in
Appendix~\ref{app:linear-symbol}.

\begin{proposition}[Mixed-order ellipticity]\label{prop:unary-ellipticity}
The weighted principal-symbol sequences of the spinorial and
positive-form variational complexes are exact at every nonzero real
covector. Their adjoint symbol sequences are also exact.
\end{proposition}
\begin{proof}
Fix a nonzero real covector. The coefficient filtration gives short
exact sequences of symbol complexes
\[
 0\longrightarrow\mathsf G_1\longrightarrow\ker\eps
   \longrightarrow\mathsf A\longrightarrow0,
 \qquad
 0\longrightarrow\ker\eps\longrightarrow\mathscr E_{\rm var}
   \longrightarrow\mathsf G_0\longrightarrow0.
\]
The calculation in Appendix~\ref{app:linear-symbol} proves exactness
of the symbol complexes on $\mathsf G_0$, $\mathsf A$ and
$\mathsf G_1$. The long exact cohomology sequences give exactness on
$\mathscr E_{\rm var}$.
The field maps and their inverses have weighted order zero, so their
symbols identify the two presentations. The adjoint symbol is the
signed transpose at the opposite covector under the real density
pairing and is therefore exact.
\end{proof}

\begin{proposition}[Mixed-order Hodge theory]\label{prop:mixed-order-hodge}
Let $(C^\bullet,d)$ be a finite complex of smooth sections of finite-rank
real bundles over a closed manifold. Suppose its components admit
weights $s_{p,a}$ such that
\[
 \operatorname{ord}(d_{p;ba})\le s_{p+1,b}-s_{p,a},
\]
and its weighted symbol sequence is exact away from the zero section.
On the component Sobolev spaces $H^{k-s_{p,a}}$, the differential has
closed range and finite-dimensional cohomology represented by smooth
sections. There is a continuous contraction on smooth sections
\begin{equation}\label{eq:variational-sdr}
 (H^\bullet(C),0)\underset{p_0}{\overset{i_0}{\rightleftarrows}}
 (C^\bullet,d),\qquad \mathsf h_0,
\end{equation}
with
\begin{equation}\label{eq:resolved-sdr-identities}
 p_0i_0=1,\qquad d\mathsf h_0+\mathsf h_0d=1-i_0p_0,\qquad
 p_0\mathsf h_0=\mathsf h_0i_0=\mathsf h_0^2=0.
\end{equation}
\end{proposition}
\begin{proof}
Choose positive bundle metrics and positive invertible elliptic
operators $\Lambda_a$ of order one. The order-reducing maps
\[
 T_p=\diag(\Lambda_a^{s_{p,a}}),\qquad
 D_p=T_{p+1}^{-1}d_pT_p
\]
give a bounded order-zero complex with exact principal symbol.
Its Laplacians
\[
 \Delta_p=D_{p-1}D_{p-1}^*+D_p^*D_p
\]
are elliptic. Let $P_p$ be the harmonic projection and $G_p$ the
Green operator. Then
\[
 \Delta_pG_p=G_p\Delta_p=1-P_p,\qquad
 \widetilde h_p=D_{p-1}^*G_p,
\]
and
\[
 D\widetilde h+\widetilde hD=1-P,\qquad
 P\widetilde h=\widetilde hP=\widetilde h^2=0.
\]
Conjugating by $T$ gives \eqref{eq:resolved-sdr-identities}.
The elliptic estimates, smoothness of the harmonic spaces and
continuity on the stated Sobolev scales are proved in
Appendix~\ref{subsec:variational-hodge-proof}.
\end{proof}

\begin{corollary}\label{cor:unary-fredholm}
The variational tangent complex has closed range and finite-dimensional
cohomology. It admits the contraction \eqref{eq:variational-sdr}.
\end{corollary}
\begin{proof}
Apply Propositions~\ref{prop:unary-ellipticity} and
\ref{prop:mixed-order-hodge} with the weights
\eqref{eq:unary-spatial-weights}.
\end{proof}

The contraction is $\mathbb R$-linear and need not preserve the
$\mathbb D$-module structure or its filtration.
Writing $J=(DU_N)_0$, the corresponding contraction of the
BPS--Noether tangent complex is
\[
 i_N=J^{-1}i_0,\qquad p_N=p_0J,\qquad
 \mathsf h_N=J^{-1}\mathsf h_0J.
\]

\subsection{The cyclic minimal model and effective potential}
\label{subsec:cyclic-kuranishi}

The real pairing $B=\tau_{\mathbb D}B_C$ is nondegenerate of degree
$-3$ and satisfies
\[
 B(d_Cx,y)+(-1)^{|x|}B(x,d_Cy)=0.
\]
It induces a pairing $B_H$ on $H_{\rm var}=H^\bullet(C_{\rm var})$.
Appendix~\ref{subsec:cyclic-transfer} proves that $B_H$ is perfect
and constructs a cyclic contraction $(i_c,p_c,\mathsf h_c)$ with
\[
 B_H(a,b)=B(i_ca,i_cb),\qquad
 B_H(p_cx,a)=B(x,i_ca),\qquad
 \mathsf h_c^\sharp=\mathsf h_c.
\]

Choose formal coordinates $\mathbf q$ centred at the background on
$\mathscr B$, with cotangent coordinates $\pi$.
For a formal variable $t$, set
\[
 W_n(\mathbf q)=[t^n]W_{\rm ext}(t\mathbf q),\qquad
 f_n(\mathbf q)=[t^n]f(t\mathbf q).
\]
The homogeneous Hamiltonians and vector fields are
\begin{equation}\label{eq:reduced-coefficient-extraction}
 \mathcal S_{n+1}=W_{n+1}+\iota_{f_n}\vartheta_{\rm var},
 \qquad Q_n=X_{\mathcal S_{n+1}}.
\end{equation}
The cotangent term has degree $n+1$ in $(\mathbf q,\pi)$.

\begin{theorem}[Finite-dimensional cyclic Kuranishi model]
\label{thm:finite-cyclic-kuranishi}
The real graded space $H_{\rm var}$ admits a minimal cyclic
$L_\infty$ structure
\[
 \mu_n:\Lambda^nH_{\rm var}\longrightarrow H_{\rm var},
 \qquad |\mu_n|=2-n,\qquad \mu_1=0,
\]
and an $L_\infty$ quasi-isomorphism
$I_\infty:H_{\rm var}\to L_{\rm var}$.
For every augmented real dg-Artin algebra $R$, it induces natural
weak equivalences
\begin{equation}\label{eq:finite-noether-equivalence}
 \operatorname{MC}_\bullet(H_{\rm var}\otimes\mathfrak m_R)
 \simeq\Def_{\rm var}(R)\cong\Def_{\mathrm{BPS},N}(R).
\end{equation}
The pairing $B_H$ gives isomorphisms
\[
 H^p_{\rm var}\cong(H^{3-p}_{\rm var})^*,
 \qquad -1\le p\le4.
\]
The minimal $L_\infty$-isomorphism class and the associated simplicial
deformation functor are independent of the contraction.
\end{theorem}
\begin{proof}
The Hamiltonian Taylor operations of
\eqref{eq:reduced-coefficient-extraction} are cyclic by
Proposition~\ref{prop:cyclic-conventions}.
Cyclic homotopy transfer along $(i_c,p_c,\mathsf h_c)$ gives
$\mu_n$ and $I_\infty$.
In terms of the induced coderivations on the suspended symmetric
coalgebras, the transfer identities are
\[
 \boldsymbol\mu^2=0,\qquad
 \boldsymbol\ell\,\boldsymbol I
      =\boldsymbol I\,\boldsymbol\mu,\qquad
 (I_\infty)_1=i_c.
\]
The recursion and its extension to dg-Artin coefficients are proved in
Appendix~\ref{subsec:kuranishi-transfer}.
They give the first equivalence in
\eqref{eq:finite-noether-equivalence}. Theorem~\ref{thm:bps-noether-equivalence}
gives the second. Perfectness of $B_H$ gives the stated dualities.

For two contractions, compose the inclusion for the first minimal
model with the reverse $L_\infty$ morphism for the second.
Its linear term is the identity on cohomology, so it has an
$L_\infty$ inverse by recursion in symmetric degree.
\end{proof}

Define the formal Kuranishi map and effective potential by
\[
 \begin{aligned}
 \kappa(x)&=\sum_{n\ge2}\frac1{n!}\mu_n(x,\ldots,x)
                   \in H^2_{\rm var},\\
 W_{\rm eff}(x)&=-\sum_{n\ge2}\frac1{(n+1)!}
                 B_H\bigl(x,\mu_n(x,\ldots,x)\bigr),
 \qquad x\in H^1_{\rm var}.
 \end{aligned}
\]
The series $W_{\rm eff}$ has degree at least three and need not be
polynomial. For an Artin algebra $R$ concentrated in degree zero,
the Maurer--Cartan equation on $H^1_{\rm var}\otimes\mathfrak m_R$
is $\kappa(x)=0$.
The groups $H^0_{\rm var}$ and $H^{-1}_{\rm var}$ describe
infinitesimal automorphisms and higher automorphisms, respectively.

\begin{corollary}[Effective potential]\label{cor:kuranishi-effective-potential}
For $x,y\in H^1_{\rm var}$,
\[
 \dd W_{\rm eff}(x)[y]=-B_H(y,\kappa(x)),
 \qquad \Crit(W_{\rm eff})=\kappa^{-1}(0).
\]
The simplicial Maurer--Cartan functor of the full graded minimal model
is equivalent to $\Def_{\mathrm{BPS},N}$ by
\eqref{eq:finite-noether-equivalence}.
\end{corollary}
\begin{proof}
Since $|x|=|y|=1$, cyclicity and graded skew-symmetry give
\[
 D_x\bigl[B_H(x,\mu_n(x,\ldots,x))\bigr][y]
  =(n+1)B_H(y,\mu_n(x,\ldots,x)).
\]
Therefore
\[
 \dd W_{\rm eff}(x)[y]
 =-\sum_{n\ge2}\frac1{n!}B_H(y,\mu_n(x,\ldots,x))
 =-B_H(y,\kappa(x)).
\]
Perfectness of $B_H:H^1_{\rm var}\times H^2_{\rm var}\to\mathbb R$
gives the equality of formal loci.
\end{proof}

\section{Discussion}\label{sec:discussion}

In the first version of this work \cite{BrongersD11v1},
the condition $H^{[0]}=0$ was imposed as an independent
equation together with its compatibility identities.
The resulting derived theory has nontrivial relative
cohomology, and its pointed relative fibre at the standard
embedding is homotopy abelian. In the present formulation,
the BPS--Euler comparison holds without this restriction.
On closed $Y$ in the fixed gerbe sector,
Proposition~\ref{prop:compact-flux-redundancy} shows that
$H^{[0]}=0$ already follows from $T=0$.
The two constructions therefore have the same ordinary
solutions under these hypotheses, but different derived
deformation theories.

This distinction raises the question of which differential
identities should enter the derived deformation problem.
The BPS--Noether enhancement constructed here uses the
Noether identities associated with the gauge action.
We have not shown that these generate all differential
relations among the BPS equations.

The finite-dimensional cyclic model identifies the formal
solution locus with $\operatorname{Crit}(W_{\rm eff})$.
Determining the transferred brackets for particular
backgrounds would give explicit obstruction maps and
conditions for unobstructed deformations. In particular,
the higher brackets may distinguish the local geometry of
the moduli space from its infinitesimal deformation space.

The results are formal near a fixed standard embedding, 
with fixed marked bundles and relative string and gerbe data, 
using the first-order Courant conventions of Section~\ref{sec:geometry}.
An analytic Kuranishi description and a global derived
moduli space remain to be constructed. The latter would
also provide a setting in which to study orientations and
possible enumerative invariants. A separate question is
whether the local BPS--Euler comparison extends beyond
first order in $\alpha'$.

\section*{Acknowledgements}
\addcontentsline{toc}{section}{Acknowledgements}
Selected tensor, sign, and coefficient identities were checked by exact
symbolic computation. OpenAI ChatGPT/Codex assisted with checking
derivations, developing verification scripts, and drafting
the manuscript. Responsibility for the arguments and mathematical claims
rests with the author.

\appendix
\section{Clifford algebra, intrinsic torsion and curvature}
\label{app:clifford}\label{app:physical-first-variation}

\subsection{Clifford and Dirac normalizations}

We use the real irreducible Clifford module in dimension seven with
$\gamma_a\gamma_b+\gamma_b\gamma_a=-2\delta_{ab}$ and symmetric
invariant pairing $\Bsp$. The orientation and irreducible Clifford
extension are fixed so that the positive form determined by a unit
spinor has components
\begin{equation}\label{eq:spinor-bilinears-app}
 \varphi_{abc}=\Bsp(\chi,\gamma_{abc}\chi),\qquad
 \psi=*\varphi,\qquad
 \Bsp(\chi,\gamma_{abcd}\chi)=-\psi_{abcd}.
\end{equation}
Repeated Clifford indices below are summed over all ordered values.
Antisymmetrized Clifford products have unit antisymmetrization weight.
In particular a three-form acts by
$H\mathbin{\cdot}=H_{abc}\gamma^{abc}/6$. Clifford generators are
skew-adjoint, whereas the Clifford action of a three-form is symmetric. Thus
\begin{equation}\label{eq:spinor-one-bilinear}
 \Bsp(\zeta,v\mathbin{\cdot}\zeta)=0
\end{equation}
for every commuting real spinor $\zeta$ and vector $v$.

\begin{lemma}\label{lem:spinor-source-identities}
For a unit spinor and its induced positive form, without a torsion
restriction,
\begin{equation}\label{eq:spinor-source-identities}
 \dd\varphi\wedge\varphi=-8\Bsp(\chi,D_g\chi)\vol_g,
 \qquad H\wedge\psi=\Bsp(\chi,H\mathbin\cdot\chi)\vol_g.
\end{equation}
\end{lemma}
\begin{proof}
The second equality is the component pairing of $H$ with
\eqref{eq:spinor-bilinears-app}. For the first, differentiate the
spinor bilinear in a normal orthonormal frame. Since $\chi$ has unit
length, $\nabla_a\chi=T_a{}^b\gamma_b\chi$ for a unique tensor $T$. The seven spinors $\gamma_b\chi$ span the orthogonal complement of
$\chi$. Clifford anticommutation and the bilinear formula give
 \[
 \Bsp(\chi,D\chi)=-\tr T,\qquad
 \langle\dd\varphi,\psi\rangle=8\tr T.
 \]
More explicitly, the four-form bilinear gives
$\nabla_i\varphi_{jkl}=2T_i{}^a\psi_{ajkl}$. Therefore
\[
 \langle\dd\varphi,\psi\rangle
 =\tfrac8{24}T_i{}^a\psi_{ajkl}\psi^{ijkl}=8\tr T,
\]
using $\psi_{ajkl}\psi_i{}^{jkl}=24\delta_{ai}$. Since
$\dd\varphi\wedge\varphi=\langle\dd\varphi,\psi\rangle\vol_g$,
the first identity follows.
\end{proof}

For $\zeta_\ph=e^{-\Phi}\chi$, the derivative of
$\Phi$ drops out by \eqref{eq:spinor-one-bilinear}, giving the
factor $1/4$ in \eqref{eq:form-source}.

For $Z_{ab}=g(e_a,Ze_b)$ the spin action is
$\sg(Z)=-Z_{ab}\gamma^a\gamma^b/4$. Indeed
\[
 [\gamma_a\gamma_b,\gamma_c]
 =2\delta_{ac}\gamma_b-2\delta_{bc}\gamma_a,
 \qquad[\sg(Z),\gamma_c]=Z_{ac}\gamma_a.
\]
In the transported metric frame, $(C_gH)_{iab}=H_{iab}/2$; equivalently
\[
 g(C_g(H)_XY,Z)=-H(X,Y,Z)/2.
\]
Hence
\[
 \gamma^i\sg(C_g(H)_i)
 =-\tfrac18H_{iab}\gamma^i\gamma^a\gamma^b
 =-\tfrac34H\mathbin\cdot.
\]
Repeated-index contractions vanish because $H$ is alternating.

\paragraph{Canonical generalized Dirac coefficient.}
\label{app:generalized-dirac-coefficient}
The canonical generalized Dirac formula of
\cite[Lemma 2.10]{CoupledG2} uses the tangent connection of torsion
$H/3$. In our exterior-to-Clifford convention its spin connection is
$\nabla^{g,S}_X+c_g(\iota_XH)/12$. The Clifford exterior identity gives
\[
 \sum_i c_g(e^i)c_g(\iota_{e_i}H)=3c_g(H).
\]
Indeed the contraction part has repeated contraction with the same
vector and vanishes, while the exterior part is the degree operator
on a three-form. Clifford contraction of the spin connection is
therefore $D_g+c_g(H)/4$. In the
half-evaluation convention of \eqref{eq:courant-pairing},
$\operatorname{div}_g-\operatorname{div}_{\varrho}
=\langle4j(\dd\Phi),\cdot\rangle$; the positive projection has
anchor $2(\dd\Phi)^\sharp$. Thus $\vartheta=2\dd\Phi$ in
\eqref{eq:canonical-generalized-dirac}, with the displayed sign and
coefficient. Its one-form self-pairing vanishes by
\eqref{eq:spinor-one-bilinear}.

\subsection{Canonical representation algebra}
\label{app:representation}

Let $I_\varphi$ be the ideal generated by
$\Lambda^2_{14}(\varphi)$. The surviving representations of
$\Lambda^\bullet/I_\varphi$ are
$\R,\Lambda^1,\Lambda^2_7,\Lambda^3_1$; higher exterior degrees
vanish. The projections are
\begin{equation}\label{eq:representation-projectors}
 \pi_7b=\tfrac13(b+*_g(\varphi\wedge b)),\qquad
 \pi_1u=\tfrac17\langle u,\varphi\rangle_g\varphi.
\end{equation}
Wedge product followed by passage to the quotient gives its
multiplication. At the background its Frobenius trace is
\begin{equation}\label{eq:frobenius-trace}
 \operatorname{Tr}_{\cA}(ab)=\int_Y w_0a\wedge b\wedge\psi_0.
\end{equation}
\begin{equation}\label{eq:trace-units}
 e^i\wedge\iota_{e_j}\phi_0\wedge\psi_0=3\delta^i_j\vol_0,
 \quad\phi_0\wedge\psi_0=7\vol_0.
\end{equation}
These identities prove nondegeneracy on complementary representations.
The coefficient trace uses $g_0$ on $T^*Y$ and $\eps P$ on the
gauge reduction. If $b\in\Omega^2_{14}$, then
$b\wedge\psi_0=0$ and $\dd\psi_0=0$ imply that $\dd b$ has zero
singlet projection. The derivation property gives
$\dd I_{\phi_0}\subset I_{\phi_0}$. Thus
$\check\dd=(\dd,\pi_7\dd,\pi_1\dd)$ squares to zero; the same
argument applies to a coefficient connection whose curvature has
exterior type fourteen, preserving its coefficient action.

Write $\cA_0=\Lambda^\bullet/I_{\phi_0}$. The trace identifies
density duals in the graded carrier
\begin{equation}\label{eq:native-full-carrier}
 C^p=\cA_0^{p+1}\mathbb D\oplus
 \cA_0^p(T^*Y\otimes\mathbb D\oplus\mathfrak{so}(E)\otimes\mathbb D_0)
 \oplus\cA_0^{p-1}\mathbb D.
\end{equation}
The momentum of cochain degree $p$ has cochain degree $3-p$.
Gauge duals use $\Hom_{\mathbb D}(\mathbb D_0,\mathbb D)=\eps\mathbb D$.
For the exterior realizations of degrees $0,1,6,7$, the maps are
$J_0=J_1=\id$, $J_2b=b\wedge\psi_0$, $J_3u=u\wedge\psi_0$, with
$J_2^{-1}\xi=\iota_{(*\xi)^\sharp}\phi_0/3$ and
$J_3^{-1}\nu=(*\nu)\phi_0/7$. They act only on exterior indices
and commute with background covariant differentiation.

\subsection{Intrinsic-torsion reconstruction on independent arguments}
\label{app:curvature-factorization}

All tensors in this subsection use a varying positive form. In an
orthonormal frame define
\[
 A_\varphi(Q)_i=\tfrac14\langle Q,e_i^\flat\wedge\varphi\rangle,
 \qquad B_\varphi(N)_i=\tfrac13\langle N,e_i^\flat\wedge\psi_\varphi\rangle,
\]
and the intrinsic-torsion map
\[
 \mathcal D_\varphi(t)=
 \left(\sum_i e^i\wedge t_i{}^a\iota_{e_a}\psi_\varphi,
       -\sum_i e^i\wedge t_i{}^a e_a^\flat\wedge\varphi\right).
\]
Its domain has dimension $49$ and target
$\Lambda^4\oplus\Lambda^5$ has dimension $56$.

\begin{lemma}\label{lem:geometric-residual-curvature}
The map $\mathcal D_\varphi$ is injective and
\begin{equation}\label{eq:intrinsic-torsion-image}
 \im\mathcal D_\varphi=
 \{(Q,N):4A_\varphi(Q)-3B_\varphi(N)=0\}.
\end{equation}
For every independent equation argument the following formulas are
therefore defined:
\begin{equation}\label{eq:curvature-equation-reconstruction}
 \begin{aligned}
 v_E&=\tfrac32B_\varphi(w^{-1}E_C)-2A_\varphi(*E_T),\\
 Q_E&=*E_T+2v_E\wedge\varphi,&
 N_E&=w^{-1}E_C+2v_E\wedge\psi_\varphi,\\
 t_E&=(\mathcal D_\varphi^\dagger\mathcal D_\varphi)^{-1}
                      \mathcal D_\varphi^\dagger(Q_E,N_E),&
 p_{E,i}&=t_{E,i}{}^a\iota_{e_a}\psi_\varphi.
 \end{aligned}
\end{equation}
Subtracting the actual leading flux from $E_T$ reconstructs
$\nabla^{\LC}\phi$ on every ambient physical field.
\end{lemma}
\begin{proof}
The pointwise identities
$\langle e^i\wedge\varphi,e^j\wedge\varphi\rangle=4\delta_{ij}$
and $\langle e^i\wedge\psi,e^j\wedge\psi\rangle=3\delta_{ij}$
give $A_\varphi(v\wedge\varphi)=B_\varphi(v\wedge\psi)=v$.
Writing $\eta_k=t_i{}^a\varphi_{kia}$, contraction gives
\[
 A_\varphi(\mathcal D_\varphi(t)_1)=-\eta/2,
 \qquad B_\varphi(\mathcal D_\varphi(t)_2)=-2\eta/3.
\]
The contractions used are
$\langle e^i\wedge\iota_{e_a}\psi,e^k\wedge\varphi\rangle
=-2\varphi_{kia}$ and
$\langle e^i\wedge e^a\wedge\varphi,e^k\wedge\psi\rangle
=2\varphi_{kia}$, proving image containment.

For injectivity let $b_t$ be the skew part of $t$, as a two-form.
Then $\mathcal D_\varphi(t)_2=-2b_t\wedge\varphi$, and
$*(b_t\wedge\varphi)=2\pi_7b_t-\pi_{14}b_t$ is invertible.
A tensor in the kernel is consequently symmetric, say $h$.
Set $h\cdot\varphi=\sum h_i{}^a e^i\wedge\iota_{e_a}\varphi$.
The identities
\[
 *\mathcal D_\varphi(h)_1=(\tr h)\varphi-h\cdot\varphi,
 \qquad(h\cdot\varphi)_{iab}\varphi_j{}^{ab}
             =4h_{ij}+2(\tr h)g_{ij}
\]
imply $h=(\tr h)g$ and then $h=0$. The compatibility map in
\eqref{eq:intrinsic-torsion-image} has rank seven, since its value
on $(v\wedge\varphi,0)$ is $4v$. Dimension now proves equality
of the image. The definitions in
\eqref{eq:curvature-equation-reconstruction} give
$4A_\varphi(Q_E)-3B_\varphi(N_E)
=4A_\varphi(*E_T)-3B_\varphi(w^{-1}E_C)+2v_E=0$.
Only matrices with invertible real reduction are inverted.

For the actual residuals with $H^{[0]}$ subtracted from $T^{[0]}$,
the torsion classes in \eqref{eq:torsion-classes} give
$A_\phi(\dd\phi)=3\tau_1$, $B_\phi(\dd\psi)=4\tau_1$ and
$v_E=\dd\Phi^{[0]}$. Thus $(Q_E,N_E)=(\dd\phi,\dd\psi)$.
A compatible positive form has a unique expression
$\nabla_i^{\LC}\phi=t_i{}^a\iota_{e_a}\psi$.
Since $*\iota_{e_a}\psi=-e_a^\flat\wedge\phi$, its alternations
are precisely $\mathcal D_\phi(t)=(\dd\phi,\dd\psi)$.
Injectivity proves the last assertion without an equation assumption.
\end{proof}

For a pure $E_T=\eta$ input, put $s=\langle\eta,\phi\rangle$.
Substitution of $t_{ia}=-\eta_{ijk}\phi_{ajk}/4+sg_{ia}/4$ into
$\mathcal D_\phi$ gives
\[
 \mathcal D_\phi(t)_1=*\eta-4A_\phi(*\eta)\wedge\phi,
 \qquad\mathcal D_\phi(t)_2=-4A_\phi(*\eta)\wedge\psi.
\]
These contractions use
$\phi_{iab}\phi_j{}^{ab}=6g_{ij}$ and the four-form bilinear
contractions above; they are pointwise identities in an orthonormal
frame. The displayed pair equals $(Q_E,N_E)$ because
$v_E=-2A_\phi(*\eta)$. Injectivity proves
\eqref{eq:pure-torsion-reconstruction} independently of a torsion
restriction on the moving form.

\paragraph{Curvature action and exact contraction.}
Our convention is $[\nabla_i,\nabla_j]v^a=R_{ij}{}^a{}_bv^b$,
$R_{ijab}=g_{ac}R_{ij}{}^c{}_b$. On a three-form,
\begin{equation}\label{eq:curvature-action-phi}
 (\rho_\varphi(R_{ij}))_{abc}
 =-R_{ij}{}^d{}_a\varphi_{dbc}
  -R_{ij}{}^d{}_b\varphi_{adc}
  -R_{ij}{}^d{}_c\varphi_{abd}.
\end{equation}
For $\chi_i=\iota_{e_i}\psi_\varphi$, define
\[
 (\mathcal C_\varphi u)_{ab}
 =\tfrac1{12}g^{ci}\langle u,\chi_i\rangle\varphi_{cab}.
\]
Direct contraction gives $\rho_\varphi(\iota_v\varphi)=3\iota_v\psi$,
$\rho_\varphi(Z_{14})=0$ and
$\langle\iota_v\psi,\iota_u\psi\rangle=4g(v,u)$.
Thus $\mathcal C_\varphi\rho_\varphi(Z)=\pi_7Z$ with the stated
coefficient. Since $\pi_7(C_g(\eta)_i)
=\eta_{ijk}\phi_{ajk}\iota_{e_a}\phi/12$, this also proves
$\rho(C_g(\eta)_i)\phi=\eta_{ijk}\phi_{ajk}\chi_a/4$.
The original curvature reconstruction is
\begin{equation}\label{eq:homogeneous-curvature}
 \mathscr R_\varphi(E)=\operatorname{swap}\mathcal C_\varphi
                          (\dd_{\nabla^{\LC}}p_E),
 \qquad\mathscr R_{0,q}=\mathscr R_\varphi^{[0]}.
\end{equation}
$\mathcal C_\varphi$ acts on the coefficient three-form, keeping
the exterior pair fixed; $\operatorname{swap}$ then actually exchanges
the pairs. The identity
$\dd_{\nabla^{\LC}}\nabla^{\LC}\phi=\rho(R_{\LC})\phi$
and Levi--Civita pair symmetry imply
$\mathscr R_{0,q}(E^{[0]}-(0,0,b,0))\wedge\psi=R_{\LC}\wedge\psi$
on the actual section, proving \eqref{eq:rp-P-3}.

\subsection{Opposite curvature and the scalar coefficient}
\label{app:opposite-curvature}

Both connections $\theta=\Theta_{\LC}+K$ and
$\nabla^-=\nabla^{\LC}-K$ are metric, with
\[
 R_\theta=R_{\LC}+D_{\LC}K+K\wedge K,
 \qquad R^-=R_{\LC}-D_{\LC}K+K\wedge K.
\]
In an orthonormal frame $(K_i)_{ab}=b_{iab}/2$, so
\[
 \begin{aligned}
 (D_{\LC}K)_{ijab}+(D_{\LC}K)_{abij}
 &=\tfrac12(\nabla_ib_{jab}-\nabla_jb_{iab}
                 +\nabla_ab_{bij}-\nabla_bb_{aij})
 \\&=\tfrac12(\dd b)_{ijab},\\
 (K\wedge K)_{ijab}
 &=-\tfrac14(b_{iac}b_{jbc}-b_{jac}b_{ibc})
 =(K\wedge K)_{abij}.
 \end{aligned}
\]
The last equality follows by interchanging the two skew pairs and
using alternation of $b$; all quadratic terms are retained.
Together with $R_{\LC,ijab}=R_{\LC,abij}$ these formulas prove
\eqref{eq:opposite-curvature-exchange}. No individual pair symmetry
of $R_\theta$ or $R^-$ is assumed.

At the marked background $\mathcal C_{\phi_0}(\chi_j)_{ab}
=\phi_{0,jab}/3$. Actual pair exchange in
\eqref{eq:local-correction-expanded} gives the two-form-valued tensor
\[
 \left[-\tfrac14\operatorname{swap}\mathcal C_{\phi_0}
                 (\dd f\wedge\chi)\right]_{ijab}
 =-\tfrac1{12}\{(\partial_af)\phi_{0,bij}
                         -(\partial_bf)\phi_{0,aij}\}.
\]
The Hodge contraction of an exterior two-form with $\psi_0$ is
$(*(U\wedge\psi_0))_k=U_{ij}\phi_{0,kij}/2$.
Since $\phi_{0,bij}\phi_{0,kij}=6g_{0,bk}$, its value is
$(g_{0,ka}\partial_bf-g_{0,kb}\partial_af)/4$.
This proves the coefficient and sign of $W\dd f$ in
\eqref{eq:linear-corrected-pairing} on every independent input.
The same pure-torsion substitution before imposing closure gives
\begin{equation}\label{eq:linear-scalar-flux-defect}
 \Gamma_0\eta=\tfrac12(\dd\eta)^{\rm pair}\wedge\psi_0
                    +W\dd\langle\eta,\phi_0\rangle.
\end{equation}
Indeed its nonscalar terms are
$D_{\LC}C_g(\eta)+\operatorname{swap}D_{\LC}C_g(\eta)
=(\dd\eta)^{\rm pair}/2$ and its scalar term is the contraction
just computed. This calculation commutes no derivatives.

\subsection{The unrestricted physical first variation}
\label{app:first-variation}

Write $a=A-\theta$, $k=\delta A$, $\kappa=\delta\theta$.
For the relative transgression of \eqref{eq:relative-transgression},
direct differentiation and the invariant trace give
\[
 \delta\cT_3(\theta,a)=2P(k,F_A)-2P(\kappa,R_\theta)
                       -\dd P(a,\kappa+k).
\]
For clarity, in relative variations $k_{\rm rel}=\delta a$ the
same formula is
$2P(k_{\rm rel},F_A)+2P(\kappa,F_A-R_\theta)
-\dd P(a,2\kappa+k_{\rm rel})$; substitution
$k_{\rm rel}=k-\kappa$ gives the absolute formula above.
Consequently
\begin{equation}\label{eq:absolute-flux-variation}
 \delta H=\dd\{\beta+\tfrac14\jmath_\eps P(a,\kappa+k)\}
       -\tfrac12\jmath_\eps\{P(k,F_A)-P(\kappa,R_\theta)\}.
\end{equation}
The implicit dependence of the full Hull connection causes no
additional term: inside the displayed ideal only its reduced
derivative $\kappa=D_V\theta$ is used.

Differentiating \eqref{eq:form-source}, with
$\delta\psi_\varphi=*J_\varphi v$ and $\delta w=-2w\sigma$, gives
\[
 \begin{aligned}
 4\dd W_{\rm ext}[V]=\int_Y\{&w\delta H\wedge\psi_\varphi
 +wH\wedge*J_\varphi v
 -\tfrac12w(\dd v\wedge\varphi+\dd\varphi\wedge v)\\
 &-2w\sigma(H\wedge\psi_\varphi
                         -\tfrac12\dd\varphi\wedge\varphi)\}.
 \end{aligned}
\]
Finite integration by parts turns the $\dd\beta$ term into
$-\beta\wedge C$ and the $v$ terms into
$wv\wedge*(J_\varphi H-*\dd\varphi+*(\dd\Phi\wedge\varphi))$.
The torsion classes give
\[
 J_\varphi(*\dd\varphi)-\tfrac13*(\dd\varphi\wedge\varphi)\varphi
 +\tfrac32\iota_{B_\varphi(\dd\psi_\varphi)^\sharp}\psi_\varphi
 =-*\dd\varphi.
\]
Indeed the scalar, seven and twenty-seven components are respectively
$-\tau_0\varphi$, $3\iota_{\tau_1^\sharp}\psi_\varphi$ and
$-\tau_3$. Using $*(p\wedge\varphi)=-\iota_{p^\sharp}\psi_\varphi$
then identifies this $v$ coefficient with $G_q(\mathcal E_{\rm BPS})$.
The scalar coefficient is
$\Sigma-2H\wedge\psi_\varphi=V_q(\mathcal E_{\rm BPS})$.
Finally the transgression term in \eqref{eq:absolute-flux-variation}
gives $P(k,A_q)$ and
$P(\kappa,a\wedge C^{[0]}/4+w_0R_\theta\wedge\psi/2)$.
Thus the complete physical first variation is
\begin{equation}\label{eq:physical-P-explicit}
 \begin{aligned}
 4\dd W_{\rm ext}[V]=\int_Y\{&-\beta\wedge C
 +wv\wedge*G_q(\mathcal E_{\rm BPS})
 +w\sigma V_q(\mathcal E_{\rm BPS})\\
 &+\jmath_\eps[P(k,A_q(\mathcal E_{\rm BPS}))
 +P(\kappa,a\wedge C^{[0]}/4+w_0R_\theta\wedge\psi/2)]\}.
 \end{aligned}
\end{equation}
Replacing $R_\theta\wedge\psi$ by \eqref{eq:rp-P-3} yields
$\dd W_{\rm ext}=\mathcal P_q\mathcal E_{\rm BPS}-\mathcal J_qb$.
The convention $\mathcal E_{\rm var}=-\dd W_{\rm ext}$ proves
\eqref{eq:rp-P-20}, including the minus sign of the gravitational
correction. Neither the base field nor its variation has been
restricted by $b=0$.

\section{Mixed duality, physical symmetry and density covariance}
\label{app:gauge-noether}\label{app:gauge}\label{app:mixed}\label{app:weighted-spinor-coordinates}\label{app:resolved-physical}

\subsection{The Hamiltonian domain and signs}\label{app:mixed-hamiltonian-domain}
For free coordinates $q$ and normalized quotient momenta $\bar p_a$,
the actual mixed potential is
\begin{equation}\label{eq:mixed-potential-small}
\vartheta=\int_Y\left(\sum_{\rm free}p_q\,\delta q
+\jmath_\eps\sum_{\rm quotient}P(\bar p_a,\delta a)\right),\qquad\omega=\delta\vartheta.
\end{equation}
The bar is the ideal-duality identification, not reduction of the
embedded covector, whose leading quotient is zero. A Hamiltonian has
\begin{equation}\label{eq:mixed-functional-domain}
F=f_0^{\mathbb D}(q,p)+\jmath_\eps f_1(q^{[0]},p^{[0]},a,\bar p_a).
\end{equation}
Here $f_0^{\mathbb D}$ is its canonical dual-number extension.
The quotient derivative is uniquely determined by
$\delta_aF=\jmath_\eps\langle F_a^{(1)},\delta a\rangle$.
The normalized quotient Hamiltonian equation uses $F_a^{(1)}$ through
the isomorphism $\jmath_\eps$, without multiplying by $\eps^{-1}$.
For $F=f_0^{\mathbb D}+\jmath_\eps f_1$ and
$G=g_0^{\mathbb D}+\jmath_\eps g_1$ the bracket is
\[
\{F,G\}=\{f_0,g_0\}_{\rm fr}^{\mathbb D}
+\jmath_\eps\bigl(\{f_0,g_1\}_{\rm fr}+\{f_1,g_0\}_{\rm fr}
+\{f_1,g_1\}_0\bigr).
\]
The subscript $0$ means the ordinary real quotient cotangent bracket
holding free coordinates fixed. This proves closure of the domain;
finite integration by parts replaces derivatives by their Euler derivatives.
The covector-first contraction gives
\begin{equation}\label{eq:canonical-euler-rows}
\delta F\equiv\sum_q(F_q\,\delta q+F_{p_q}\,\delta p_q),\qquad
X_Fq=F_{p_q},\quad X_Fp_q=-F_q,\quad\{q,p\}=-1.
\end{equation}
Indeed $\iota_X(p\,\delta q)=(-1)^{|X||p|}pXq$ and
$\iota_X(\delta p\,\delta q)=(Xp)\delta q-(Xq)\delta p$,
using $|p|=-1-|q|$. Hence $X_{H_f}q=fq$, including quotient blocks.
For any differential field map the cotangent lift means
\begin{equation}\label{eq:actual-cotangent-lift}
\pi=(DT)^\dagger\pi',\qquad
\int\langle\pi',\delta T(q)\rangle\equiv\int\langle\pi,\delta q\rangle.
\end{equation}

\subsection{Courant action, frame and Green--Schwarz closure}
\label{app:courant-action}\label{app:native-closure}
For $C=(C_T,A)$, $F=F_C$ and $\nabla=D_C$, the marked Dorfman bracket is
\begin{equation}\label{eq:string-dorfman}
\begin{split}
(X+u+\eta)\circ(Y+v+\zeta)
={}&[X,Y]+\nabla_Xv-\nabla_Yu-[u,v]-F(X,Y)\\
&+\mathcal L_X\zeta-\iota_Yd\eta+\iota_Y\iota_XH\\
&+2\mathfrak c_z(\nabla u,v)+2\mathfrak c_z(\iota_XF,v)
-2\mathfrak c_z(\iota_YF,u).
\end{split}
\end{equation}
For a connection difference $v_C$ and a two-form $b_2$, the map from
new splitting coordinates to old ones is
\begin{equation}\label{eq:string-splitting-change}
X+u+\eta\longmapsto X+(u-v_C(X))+\eta+\iota_Xb_2
 +2\mathfrak c_z(v_C,u)-\mathfrak c_z(v_C(X),v_C).
\end{equation}
Here the new connection is $C+v_C$ and the new flux is
$H+\mathcal T_{\mathfrak c_z}(C,v_C)+db_2$, where
\[
\mathcal T_{\mathfrak c_z}(C,v_C)
=2\mathfrak c_z(v_C,F_C)+\mathfrak c_z(v_C,D_Cv_C)
 +\tfrac13\mathfrak c_z(v_C,[v_C,v_C]).
\]
The linear and quadratic pairing terms cancel directly in
\eqref{eq:string-splitting-change}. Substituting into
\eqref{eq:string-dorfman} gives the displayed flux change.
Its derivative is
$\mathfrak c_z(F_{C+v_C},F_{C+v_C})-\mathfrak c_z(F_C,F_C)$,
which checks the sign and the direction of the map.

Its skew bracket $[e_1,e_2]_C=e_1\circ e_2-jd\langle e_1,e_2\rangle$
has Jacobiator $jdT_C$, where
\[
 T_C=\sum_{\rm cyc}\langle[e_1,e_2]_C,e_3\rangle/3.
\]
The corresponding left CE rules are
\begin{equation}\label{eq:string-ce}
Q\gamma=jd\beta_C-\tfrac12[\gamma,\gamma]_C,\qquad
Q\beta_C=-\tfrac12a_{\mathbb E}(\gamma)(\beta_C)-\tfrac16T_C(\gamma,\gamma,\gamma).
\end{equation}
Since $\langle jdh,e\rangle=a_{\mathbb E}(e)h/2$, constant scalar
reducibilities are retained.

In the symmetric metric frame $I=I_g$, put $N^a{}_b=\nabla_b^0\xi^a$.
The skew frame derivative and its cocycle are
\begin{equation}\label{eq:native-frame-lift}
j_g(\xi)=I^{-1}NI+I^{-1}(DI)_g(N^Tg+gN).
\end{equation}
\begin{equation}\label{eq:native-frame-cocycle}
\begin{gathered}
J(g,F)_x=I_g(Fx)^{-1}dF_xI_{F^*g}(x),\qquad
J(g,F\circ G)=G^*J(g,F)J(F^*g,G).
\end{gathered}
\end{equation}
The spin lift acts on the fixed-frame spinor $\eta=I_g^*\chi$ by
$\Spin(J)^{-1}F^*\eta$, giving
\begin{equation}\label{eq:native-base-action}
fg=-\mathcal L_\xi g,\qquad
f\eta=-\nabla_\xi^0\eta+\sg(j_g(\xi))\eta,\qquad
f\varrho=-\mathcal L_\xi\varrho.
\end{equation}
Set $k_g(\xi)=(j_g(\xi)-\iota_\xi\Theta_0)^{[0]}$,
$c_A=s-\iota_{\xi^{[0]}}A=\lambda+k_g(\xi)$ and $a=A-\theta$,
where $\theta=\Theta_H^{[0]}$. Write $f$ for the lower BRST differential. In absolute coordinates its rules are
\begin{equation}\label{eq:rv-E-25}
\begin{aligned}
Q\varphi&=-\mathcal L_\xi\varphi,&Q\Phi&=-\xi(\Phi),\\
QB&=-\mathcal L_\xi B-d\Lambda-\jmath_\epsilon b_\lambda,&
QA&=-\iota_{\xi^{[0]}}F_A-D_As,\\
Q\xi&=-\tfrac12[\xi,\xi],&
Qs&=\tfrac12[s,s]-\tfrac12\xi^{[0]i}\xi^{[0]j}F_{ij},\\
Q\Lambda&=dc-\mathcal L_\xi\Lambda-\tfrac14\jmath_\epsilon P(\lambda,D_A\lambda),&
Qc&=-\xi(c)+\tfrac1{24}\jmath_\epsilon P(\lambda,[\lambda,\lambda]),
\end{aligned}
\end{equation}
where
\[
\lambda=s-k_g(\xi)-\iota_{\xi^{[0]}}A,\qquad
b_\lambda=\tfrac12P(\lambda,F_A)-\tfrac14P(a,D_A\lambda).
\]
All \(\lambda\) and connection expressions in ideal terms are \(\mathbb D_0\)-valued.
Products have their ghost signs in addition to their horizontal
exterior signs. The differential acts from the left and commutes
with $\dd$. In a flat reference presentation of the extension the
dictionary to \eqref{eq:string-ce} is
\begin{equation}\label{eq:string-reduced-dictionary}
\begin{gathered}
u_T=j_g(\xi)-\iota_\xi\Theta_0,\quad t=u_T^{[0]},\quad
u_G=t+\lambda=s-\iota_\xi A,\\
B_a=B-\tfrac14\jmath_\eps P(\theta,a),\qquad
\eta_C=-\Lambda-\tfrac12\jmath_\eps P(A,\lambda),\\
\beta_C=-c+\tfrac14\jmath_\eps P(t,\lambda)-\tfrac12\iota_\xi\eta_C,
\qquad \gamma=\xi+u_T+u_G+\eta_C.
\end{gathered}
\end{equation}
It is triangular and invertible. The absolute representatives
$\Lambda_a=\Lambda-\jmath_\eps P(t,a)/2$ and
$c_a=c-\jmath_\eps P(t,\lambda)/4$ give
$H=dB_a+\CS_{\mathfrak c}(C)$, $\mathfrak c=\jmath_\eps(P_T-P_G)/4$.
Differentiating yields $fC=-\mathcal L_\xi C-D_Cu$ and
$fB_a=-\mathcal L_\xi B_a+d\eta_C+\mathfrak c(C,du)$;
substitution gives \eqref{eq:rv-E-25}.

\begin{proof}[Proof of lower square-zero and flux invariance]
The frame cocycle gives the composition law for the joint metric-spinor
action, hence square-zero. In the absolute gauge chart,
$fA=-\mathcal L_\xi A-D_Ac_A$ and
$fc_A=c_A^2-\xi^i\partial_ic_A$. Vector-field and gauge Jacobi give
square-zero. The left chain rule is
\begin{equation}\label{eq:native-covariant-ghost-chain}
fs=fc_A+\iota_{f\xi}A-\iota_\xi fA
=s^2-\tfrac12\xi^i\xi^jF_{ij}.
\end{equation}
Differentiating the frame cocycle, including its metric input, gives
$fk_g=k_g^2-\xi^i\partial_i k_g$, and therefore
\begin{equation}\label{eq:native-relative-ghost}
f\lambda=\tfrac12[\lambda,\lambda]-\nabla_\xi^0\lambda+[\lambda,j_g(\xi)].
\end{equation}
Relative transgression varies by
$\delta_\lambda\cT_3=d\{2P(\lambda,F_A)-P(a,D_A\lambda)\}$.
It cancels the compensated variation of $dB$, proving $fH=-\mathcal L_\xi H$.
After removing diffeomorphism and common-frame transport the remaining
gauge differential has $qA=-D_A\lambda$, $q\theta=0$,
$q\lambda=\lambda^2$, $q(D_A\lambda)=0$, $qF_A=-[F_A,\lambda]$.
Invariance of $P$ and graded Jacobi give
\begin{equation}\label{eq:native-direct-descent}
\begin{aligned}
qb_\lambda&=-\tfrac12P(\lambda^2,F_A)+\tfrac14P(D_A\lambda,D_A\lambda)
=\tfrac14dP(\lambda,D_A\lambda),\\
qP(\lambda,D_A\lambda)&=P(\lambda^2,D_A\lambda)
=\tfrac16dP(\lambda,[\lambda,\lambda]),\\
qP(\lambda,[\lambda,\lambda])&=0.
\end{aligned}
\end{equation}
These cancel $f^2B$ and $f^2\Lambda$. Naturality restores all mixed
transport terms. Scalar coherence holds before applying $d$: for
$T=P(\lambda,[\lambda,\lambda])$, invariance and Jacobi give $fT=-\xi(T)$,
so $f^2c=[\xi,\xi](c)/2-\xi(\xi(c))+\jmath_\eps(\xi(T)-\xi(T))/24=0$.
Thus every base generator, including constant scalar automorphisms,
has zero square.
\end{proof}

\subsection{Finite relative gerbe representatives}\label{app:finite-gerbe}
For $u_t=\exp(t\ell)$, $v_t=u_t^{-1}\partial_tu_t$ and
$\vartheta_t=u_t^{-1}d_Yu_t$, define
$W_2(u)=-\int_0^1\tr_7(v_t\vartheta_t\vartheta_t)dt$.
Maurer--Cartan gives $dW_2=-\tr_7((u^{-1}du)^3)/3$; hence
$\Omega_2(C,u)=-P(du\,u^{-1},C)-W_2(u)$ satisfies
$d\Omega_2=\CS_P(C^u)-\CS_P(C)$.
For independent gauge and tangent-frame changes $u,v$,
\begin{equation}\label{eq:finite-relative-gerbe}
\begin{aligned}
A'&=A^u,\quad\theta'=\theta^v,\quad a'=A'-\theta',\\
B'&=B+d\Lambda_{\rm fin}+\tfrac14\jmath_\eps\{
\Omega_2(A,u)-\Omega_2(\theta,v)+P(\theta',a')-P(\theta,a)\}.
\end{aligned}
\end{equation}
The transgression identity proves preservation of $H$; on common frame
changes the $W_2$ and remaining two-form terms cancel. Infinitesimally
$\omega_{\rm cov}=2P(\lambda,F_A)-P(a,D_A\lambda)
=\omega_{\rm fin}+2dP(\lambda,A)$, so
\begin{equation}\label{eq:gerbe-parameter-change}
\Lambda_{\rm fin}=\Lambda_{\rm cov}+\tfrac12\jmath_\eps P(\lambda,A).
\end{equation}
Both field charts use $\Lambda_{\rm cov}$ and the same scalar $c$.

\subsection{Full lower derivatives and all density coefficients}
We derive the action on equation covectors by differentiating the
physical field action and preserving the density pairing. Use an ordinary even parameter \(u=(\xi,s_g,\Lambda)\) in the absolute parameter chart, and set
\[
\lambda=s_g-k_g(\xi)-\iota_{\xi^{[0]}}A,\qquad
b_\lambda=\tfrac12P(\lambda,F_A)-\tfrac14P(a,D_A\lambda).
\]
The exact physical action is
\begin{equation}\label{eq:rv-N0-14}
G_uq=(\mathcal L_\xi\varphi,\
\mathcal L_\xi B+d\Lambda+\jmath_\epsilon b_\lambda,\
\xi(\Phi),\ \iota_{\xi^{[0]}}F_A+D_As_g).
\end{equation}
For fixed \(u\), its field derivative \(L_u=D_qG_u\) acts on \(V=(v,\beta,\sigma,k)\) by
\begin{equation}\label{eq:rv-N0-15}
\begin{aligned}
(L_uV)_\varphi&=\mathcal L_\xi v,&
(L_uV)_\Phi&=\xi(\sigma),\\
(L_uV)_A&=\iota_{\xi^{[0]}}D_Ak+[k,s_g],\\
(L_uV)_B&=\mathcal L_\xi\beta+\jmath_\epsilon\bigl\{
\tfrac12P(\dot\lambda,F_A)+\tfrac12P(\lambda,D_Ak)
-\tfrac14P(k-\kappa,D_A\lambda)\\
&\hspace{43mm}
-\tfrac14P(a,D_A\dot\lambda+[k,\lambda])\bigr\},\\
\dot\lambda&=-D_gk_g(\xi)[m]-\iota_{\xi^{[0]}}k,\qquad
m=D_\phi g[v^{[0]}],\quad \kappa=D_V\theta .
\end{aligned}
\end{equation}

The metric-frame derivative in \eqref{eq:rv-N0-15} can be written
as follows. Put \(I=(g_0^{-1}g)^{-1/2}\), \(N=\nabla^0\xi\), \(M=N^Tg+gN\), \(\dot I=(DI)_g[m]\). The frame formula of Appendix~\ref{app:gauge} gives
\begin{equation}\label{eq:rv-N0-16}
D_gk_g(\xi)[m]
=-I^{-1}\dot I\,j_g(\xi)+I^{-1}N\dot I
+I^{-1}\{(D^2I)_g[m,M]+(DI)_g[N^Tm+mN]\}.
\end{equation}
These derivatives are taken in the formal matrix function
$I=(g_0^{-1}g)^{-1/2}$. Similarly, before common-frame transport, the induced derivative is
\begin{equation}\label{eq:rv-N0-17}
\delta\Theta^u=\delta\Gamma_{\rm LC}+D_gC_g(dB^{[0]})[m]+C_g(d\beta^{[0]}),
\quad
(\delta\Gamma_{\rm LC})^k{}_{ij}
=\tfrac12g^{k\ell}
(\nabla_i m_{j\ell}+\nabla_jm_{i\ell}-\nabla_\ell m_{ij}).
\end{equation}
The derivative of \(C_g\) differentiates its single inverse-metric contraction in the \(C_g(H)_{iab}=H_{iab}/2\) convention. In the common frame \eqref{eq:rv-N0-17} is conjugated by \(I\) and augmented by \(D_\theta(I^{-1}\dot I)\).  Equations \eqref{eq:rv-N0-15}–\eqref{eq:rv-N0-17} have at most first horizontal order in \(V\). Their finite density adjoints are local.

Polarization of the exact absolute ghost rule gives the bilinear bracket coefficient
\begin{equation}\label{eq:rv-E-18}
\begin{split}
b_\xi(u,v)&=-[X_u,X_v],\\
b_s(u,v)&=[s_u,s_v]-X_u^{[0]i}X_v^{[0]j}F_{ij},\\
b_\Lambda(u,v)&=-\mathcal L_{X_u}\Lambda_v+\mathcal L_{X_v}\Lambda_u\\
&\quad-\tfrac14\jmath_\epsilon\{
P(\lambda_u,D_A\lambda_v)-P(\lambda_v,D_A\lambda_u)\}.
\end{split}
\end{equation}

For \(V=(v,\beta,\sigma,k)\) put
\(\dot\lambda_V=-D_gk_g(\xi)[D_\phi g(v^{[0]})]-\iota_{\xi^{[0]}}k\).
The field derivatives of the quadratic gauge-ghost and cubic
scalar-ghost terms are
\begin{equation}\label{eq:rv-E-26}
\begin{split}
\mathfrak b_{\gamma,q}'(V)_\xi&=0,\\
\mathfrak b_{\gamma,q}'(V)_s
&=-\tfrac12\xi^{[0]i}\xi^{[0]j}(D_Ak)_{ij},\\
\mathfrak b_{\gamma,q}'(V)_\Lambda
&=-\tfrac14\jmath_\epsilon\{
P(\dot\lambda_V,D_A\lambda)
+P(\lambda,D_A\dot\lambda_V+[k,\lambda])\},\\
\mathfrak c_{\gamma,q}'(V)
&=\tfrac18\jmath_\epsilon P(\dot\lambda_V,[\lambda,\lambda]).
\end{split}
\end{equation}
For a gauge parameter $u$, define
\begin{equation}\label{eq:rv-E-27}
\mathfrak c_\gamma'(u)=
\tfrac18\jmath_\epsilon P(\lambda_u,[\lambda,\lambda]).
\end{equation}
The \(1/8\) is the derivative of the displayed alternating cubic \(1/24\) term with the ghost variation placed on the right. All three terms agree with this sign. Let \(b(\gamma,u)\) denote the extension of the already defined
parameter bracket with one ghost argument.

The complete cotangent Hamiltonian is
\begin{equation}\label{eq:rv-E-30}
\begin{split}
H_f={}&e_\varphi[\mathcal L_\xi\varphi]
+e_B[\mathcal L_\xi B+d\Lambda+\jmath_\eps b_\lambda]
+e_\Phi[\xi(\Phi)]+e_A[\iota_{\xi^{[0]}}F_A+D_As]\\
&-\tfrac12z_\xi[[\xi,\xi]]
+z_s[\tfrac12[s,s]-\tfrac12\xi^{[0]i}\xi^{[0]j}F_{ij}]\\
&+z_\Lambda[dc-\mathcal L_\xi\Lambda-\tfrac14\jmath_\eps P(\lambda,D_A\lambda)]\\
&+\tau[\xi(c)-\tfrac1{24}\jmath_\eps P(\lambda,[\lambda,\lambda])].
\end{split}
\end{equation}
Varying and placing coefficients before variational one-forms gives
\eqref{eq:rv-E-28}: $Qe=\mathcal E_{\rm var}-(H_f)_q$,
$Qz=-(H_f)_\gamma$, $Q\tau=-(H_f)_c$. Differentiating the descent
identities proves coherence of every bracket and cubic adjoint term.

\subsection{Weighted-spinor field chart and cotangent inverse}\label{app:momenta}
In a background orthonormal atlas use
$x=(z,M,\alpha,l)$ with
$z\in\Omega^2_7\otimes\mathbb D$, $M\in\Omega^1(T^*Y)\otimes\mathbb D$,
$\alpha\in\Omega^1(\mathfrak{so}(E))\otimes\mathbb D_0$, $l\in\Omega^0\otimes\mathbb D$.
Set
\begin{equation}\label{eq:exact-native-chart}
\begin{gathered}
h=M+M^T,\quad\beta=M-M^T,\quad L=e^{h/2},\quad
g=g_0(L\cdot,L\cdot),\quad B-B_0=\beta,\\
\varrho=\varrho_0e^l,\quad A=\Theta_0+\alpha,\qquad
\Phi=\Phi_0+\tfrac14\tr h-\tfrac12l.
\end{gathered}
\end{equation}
The normalized positive generalized frame has blocks
$K_+=(L^{-1}+L-\beta L^{-1})/2$, $K_-=(L^{-1}-L+\beta L^{-1})/2$,
with $K_+^TK_+-K_-^TK_-=1$. Its left polar rotation is
\begin{equation}\label{eq:positive-polar}
P_+=(K_+K_+^T)^{1/2},\qquad O=P_+^{-1}K_+,\qquad
\widehat O=\exp(\sg(\log O)).
\end{equation}
For $p_a=z_{ij}\varphi_{0,aij}/4$ and $\gamma(p)=c_0(p^\flat)$, the spinor is
\begin{equation}\label{eq:spinor-rational-chart}
\chi_+(z)=\frac{\chi_0+\gamma(p)\chi_0}{\sqrt{1+|p|_0^2}},\qquad
\chi_g=I_g\widehat O^{-1}\chi_+(z).
\end{equation}

\begin{lemma}\label{lem:native-chart-inverse}
This field chart, including $a=A-\red{\Theta^u(g,B)}$, has a
two-sided local formal inverse.
\end{lemma}

\begin{proof}
Recover $g=g_\varphi$, $h=\log(g_0^{-1}g)$,
$M=(h+\beta)/2$, $l=\tr h/2-2(\Phi-\Phi_0)$ and
$\widehat\varphi=\Lambda^3L^{-1}\varphi$. In the fixed Clifford frame
$\Pi_{\widehat\varphi}=(1+\widehat\varphi\cdot)/8$ is the rank-one
spin-line projector: check it in the reference spinor and transport
by $\Spin(7)$. Thus
$\widehat\chi=\Pi_{\widehat\varphi}\chi_0/
\sqrt{\Bsp(\chi_0,\Pi_{\widehat\varphi}\chi_0)}$ selects the local unit
representative. Reconstruct $O$ and put $q_+=\widehat O\widehat\chi$.
Then
\begin{equation}\label{eq:inverse-spinor-chart}
p_i=\frac{\Bsp(\gamma_i\chi_0,q_+)}{\Bsp(\chi_0,q_+)},\qquad
z=\tfrac23\iota_p\varphi_0,\qquad
\alpha=a+\red{\Theta^u(g,B)}-\Theta_0.
\end{equation}
All denominators have nonzero reduction. The vector/scalar ratio
inverts \eqref{eq:spinor-rational-chart}; the last connection substitution
is triangular and differentiates $g,B$ only once.
\end{proof}

\label{app:cotangent-lift}
For $\Psi(x,\gamma,c)=(u(x),\mathsf P_x\gamma,c)$,
\begin{equation}\label{eq:full-triangular-jacobian}
D\Psi=\begin{pmatrix}\mathsf A&0&0\\\mathsf C&\mathsf P&0\\0&0&1\end{pmatrix},\qquad
\mathsf A=Du,\quad \mathsf C(V)=(D_x\mathsf P[V])\gamma.
\end{equation}
Its cotangent lift and inverse are
\begin{equation}\label{eq:full-cotangent-lift}
\begin{aligned}
p_\gamma&=(\mathsf P^{-1})^\dagger\pi_\gamma,&
p_x&=(\mathsf A^{-1})^\dagger(\pi_x-\mathsf C^\dagger p_\gamma),&p_c&=\pi_c.
\end{aligned}
\end{equation}
\begin{equation}\label{eq:equivalence-momenta}
\begin{aligned}
\pi_\gamma&=\mathsf P^\dagger p_\gamma,&
\pi_x&=\mathsf A^\dagger p_x+\mathsf C^\dagger p_\gamma,&\pi_c&=p_c.
\end{aligned}
\end{equation}
For source $(\xi,s,\Lambda)$ and target $(\xi,\lambda,\Lambda)$, define
$\mathcal L_xX=(j_g(X)+\iota_X(A-\Theta_0))^{[0]}$,
$\mathcal J_\xi m=(D_gj)_g(m;\xi)^{[0]}$, $\mathcal I_\xi k=\iota_{\xi^{[0]}}k$.
Then the actual parameter blocks are
\begin{equation}\label{eq:native-parameter-jacobian}
\mathsf P=\begin{pmatrix}1&0&0\\-\mathcal L_x&1&0\\0&0&1\end{pmatrix},\quad
\mathsf P^{-1}=\begin{pmatrix}1&0&0\\\mathcal L_x&1&0\\0&0&1\end{pmatrix},\quad
\mathsf C(V)=\begin{pmatrix}0\\-\mathcal J_\xi\delta g-\mathcal I_\xi\delta A\\0\end{pmatrix}.
\end{equation}
Pairing these blocks with arbitrary variations proves the potential
identity directly. The mixed adjoint of a free-to-quotient derivative
maps its ideal covector to an ideal free covector. Coefficient
derivatives and finite integration by parts are part of every $\dagger$.

\subsection{Independent-argument physical covariance}\label{app:rp-independent}
Use the corrected physical pairing with reconstructed curvature
$\widehat r_q(E)$ from Section~\ref{sec:local-comparison}. Its Bianchi
combination is $\widehat Z(E)$ of Section~\ref{sec:bps-noether}.
Writing
\[
 \widehat H_q(E)=\tfrac14a\wedge E_C^{[0]}+\tfrac12w_0\widehat r_q(E),
\]
the identity $D_Aa+D_\theta a=2(F_A-R_\theta)$ gives, on every
independent argument,
\begin{equation}\label{eq:rp-P-6}
 D_AA_q(E)+D_\theta\widehat H_q(E)
 =\tfrac12\{\widehat K_\theta(E)-J_A(E)-a\wedge dE_C^{[0]}\}
 =\widehat Z(E).
\end{equation}
Indeed the two derivatives of $a\wedge E_C^{[0]}/4$ give
$(F_A-R_\theta)\wedge E_C^{[0]}/2-a\wedge dE_C^{[0]}/2$;
the remaining two derivatives are those in
\eqref{eq:noether-bianchi}.
For $V=(v,\beta,\sigma,k)$, put
$t_{\xi,V}=D_gk_g(\xi)[D_\phi g(v^{[0]})]$.
All tensorial terms in the common-frame covariance defect are Lie
derivatives of densities. The remaining terms are $D_At_{\xi,V}$ and
$D_\theta t_{\xi,V}$; their finite adjoints give the unscaled defects
\begin{align}
-\int_Y\jmath_\eps P(t_{\xi,V},\widehat Z(E)),\label{eq:rp-P-8}\\
-\tfrac12\int_Y\jmath_\eps P(k,\zeta\,dE_C^{[0]}).\label{eq:rp-P-9}
\end{align}
for diffeomorphisms and pure gauge transformations respectively.
The gauge defect follows by expanding $D_Vb_\zeta$ in
\eqref{eq:rv-N0-15}: induced-connection terms cancel, leaving
$-dP(\zeta,k)/2$ multiplied by $E_C^{[0]}$. Gerbe transformations
give zero. The density multiplier is one quarter of this pairing.
For the explicit compensator, let
$\mathsf J_{\xi,q}(v)=t_{\xi,V}$ and
$U_\xi=-w_0^{-1}*\mathsf J_{\xi,q}^{\top}\widehat Z(E)$,
$a_\xi=\langle U_\xi,\phi\rangle/7$. In relative parameters it is
\begin{equation}\label{eq:rp-P-10}
\begin{aligned}
\widehat K_C&=0,&
\widehat K_\Sigma&=-21\jmath_\eps(a_\xi\vol_\phi),\\
\widehat K_T&=\jmath_\eps(-\tfrac92a_\xi\phi+\pi_7U_\xi-\pi_{27}U_\xi),&
\widehat K_I&=w_0^{-1}\zeta\,dE_C^{[0]}.
\end{aligned}
\end{equation}
The singlet coefficients solve $3a=4t-\sigma$, $0=3\sigma-14t$.
These compensators vanish on the actual section by the exterior,
gauge and tangent Bianchi identities. On independent equation
generators they remain. Passing to absolute parameters differentiates
\eqref{eq:native-ghost-shift} using \eqref{eq:rv-N0-15}--\eqref{eq:rv-N0-17};
the differentiated residual chart in Section~\ref{sec:bps-noether} includes
these terms and every coefficient derivative of the moving multiplier.

\subsection{Cyclic Taylor conventions}\label{app:cyclic-conventions}
For the cochain module $L$, set
\begin{equation}\label{eq:cyclic-pairing-sign}
\langle x,y\rangle_B=(-1)^{|x|}\omega(\mathbf sx,\mathbf sy).
\end{equation}
It has degree $-3$ and is graded symmetric. Put $a_r=|x_r|-1$,
$A_n=\sum_ra_r$, $b_n=\sum_{r<t}a_ra_t$ and
$e_n=\binom n2+\sum_r(n-r)|x_r|$. For raw left directional tensors
$\widehat Q_n=(D^nQ)_0$, without a factorial,
\begin{equation}\label{eq:cyclic-polarization}
\ell_n(x_1,\ldots,x_n)=(-1)^{e_n+A_n+b_n}\mathbf s^{-1}
\widehat Q_n(\mathbf sx_1,\ldots,\mathbf sx_n).
\end{equation}
The operations have degree $2-n$, are graded skew and use $1/n!$
in the vector field. Their cyclic tensors obey
\begin{equation}\label{eq:cyclic-rotation}
\langle\ell_n(x_1,\ldots,x_n),x_{n+1}\rangle_B
=(-1)^{n+|x_1|\sum_{r=2}^{n+1}|x_r|}
\langle\ell_n(x_2,\ldots,x_{n+1}),x_1\rangle_B.
\end{equation}

\begin{proposition}[Cyclic Taylor operations]\label{prop:cyclic-conventions}
For a constant shifted symplectic form and Hamiltonian $Q$, these
Taylor operations are cyclic and satisfy the $L_\infty$ identities.
\end{proposition}

\begin{proof}
Negative graded duality between $Q$ and its suspended coderivation
contributes $(-1)^{A_n+b_n}$, and desuspension contributes $(-1)^{e_n}$.
The equation $\iota_{Q_n}\omega=-\delta S_{n+1}$ makes the Hamiltonian
tensor graded symmetric. Desuspending and moving its first entry to
the end gives \eqref{eq:cyclic-rotation}. The homogeneous Taylor
coefficients of $Q^2=0$ give the $L_\infty$ identities.
\end{proof}

\section{The variational symbol, Hodge contraction and cyclic transfer}
\label{app:elliptic-cyclic}\label{app:elliptic-transfer}

All spaces are real, and the background is the torsion-free standard
embedding with constant \(w=e^{-2\Phi_0}>0\). We give the actual field
and adjoint components underlying the mixed symbol calculation, followed
by the analytic and cyclic contractions. The two-form variation is
unrestricted throughout.

\subsection{Unary field and covector formulas}
\label{app:linear-fields}

For a three-form \(u\), define
\[
 H(u)_{ij}=\frac14(u_{imn}\varphi_{0,jmn}+u_{jmn}\varphi_{0,imn})
                  -\frac1{18}u_{mnp}\varphi_{0,mnp}g_{0,ij},\qquad
 V(u)_a=\frac1{24}u_{ijk}\psi_{0,aijk}.
\]
Then \(u=\mathsf r(H(u)/2)\varphi_0+\iota_{V(u)}\psi_0\). The inverse
of \eqref{eq:linear-field-map} is
\begin{equation}\label{eq:linear-inverse}
\begin{gathered}
 M=(H(u)+\beta)/2,\quad
 z=\tfrac13\iota_{V(u)}\varphi_0-\tfrac12\pi_7\beta,\\
 l=\tfrac12\tr H(u)-2p,\qquad \alpha=a+\KHull M^{[0]}.
\end{gathered}
\end{equation}
To verify the Hodge-star variation, differentiate the
\(\mathrm{GL}(TY)\)-equivariance of \(\varphi\mapsto *_\varphi\varphi\):
\[
 \mathcal J_{\varphi_0}(\mathsf r(A)\varphi_0)=\mathsf r(A)\psi_0.
\]
For symmetric \(A\), the star variation gives
\(\mathsf r(A)\psi_0=*_0((\tr A)\varphi_0-\mathsf r(A)\varphi_0)\).
Scalar and trace-free symmetric \(A\) give \(4/3\) and \(-1\) on
\(\Lambda^3_1\) and \(\Lambda^3_{27}\). For skew \(A\), orthogonal
equivariance gives the factor \(1\) on \(\Lambda^3_7\). This proves the
operator \(J\) used in \eqref{eq:linear-four-form}.

The parameter dictionary and its inverse are
\begin{equation}\label{eq:linear-parameter-map}
 \xi^\flat=(r_++r_-)/\sqrt2,\quad
 \Lambda=(r_--r_+)/\sqrt2,\quad
 \kappa(\xi)_{ab}=(\nabla_b\xi_a-\nabla_a\xi_b)/2,\quad
 \lambda=s-\kappa(\xi)^{[0]}.
\end{equation}
Thus \(r_\pm=(\xi^\flat\mp\Lambda)/\sqrt2\) and
\(s=\lambda+\kappa(\xi)^{[0]}\). The linearized gauge action is
\begin{equation}\label{eq:linear-gauge-complete}
\begin{aligned}
 G_z&=\pi_7\dd r_+/\sqrt2-\tfrac{\eps}{4}\pi_7P(\lambda,R_0),\\
 (G_M)_{ia}&=(\nabla_i(r_-)_a+\nabla_a(r_+)_i)/\sqrt2
                       +\tfrac{\eps}{4}P(\lambda,R_{0,ia}),\\
 G_\alpha&=\iota_\xi R_0+Ds,\qquad
 G_l=(\diver r_++\diver r_-)/\sqrt2.
\end{aligned}
\end{equation}
The scalar map is \(\mathsf Rc=(-\dd c/\sqrt2,\dd c/\sqrt2,0)\).
The connection identity needed to pass to relative parameters is
\begin{equation}\label{eq:linear-frame-identity}
 \gamma(\cL_\xi g_0)_{iab}
 =\nabla_i\kappa(\xi)_{ab}+\xi^qR_{0,qiab},\quad
 \gamma(h)_{iab}=\tfrac12(\nabla_bh_{ia}-\nabla_ah_{ib}).
\end{equation}
Substitution of \(h_{ia}=\nabla_i\xi_a+\nabla_a\xi_i\) leaves the
sum of three covariant commutators. The Riemann symmetries and algebraic
Bianchi identity reduce it to the displayed curvature term. Consequently
\(Ga=D\lambda\) and
\(G\beta=\dd\Lambda+\eps P(\lambda,R_0)/2\).

First vary \((u,\beta,p,\alpha,k)\) as separate inputs in
\eqref{eq:quadratic-native-source}. Their ordinary form covectors are
\begin{equation}\label{eq:hessian-euler-inputs}
\begin{aligned}
 E_u={}&\tfrac w4[J(\dd\beta)-*\dd u+*(\dd p\wedge\varphi_0)]
                   -\tfrac{\eps w}{8}JP(a,R_0),\\
 E_\beta={}&-\tfrac w4*\dd\dot\psi_w,\\
 E_\Phi={}&\tfrac w4*(\dd u\wedge\varphi_0-2\dd\beta\wedge\psi_0)
                +\tfrac{\eps w}{4}*(P(a,R_0)\wedge\psi_0),\\
 E_\alpha={}&-\tfrac w8*(R_0\wedge\dot\psi_w+D\alpha\wedge\psi_0),\\
 E_k={}&+\tfrac w8*(R_0\wedge\dot\psi_w+Dk\wedge\psi_0).
\end{aligned}
\end{equation}
The last two entries are normalized gauge covectors: their evaluation
has an external factor \(\eps P\). All their inputs, and all inputs in
an explicit \(\eps\) term, are reduced modulo \(\eps\).
The dependence on \((M,z,l)\) is included by the chain rule. Explicitly,
put
\[
 I(q)_{ia}=\tfrac14(q_{imn}\varphi_{0,amn}+q_{amn}\varphi_{0,imn}),
 \qquad (\KHull^\dagger b)_{ia}=-2\nabla_b b_{iab}-\nabla_jb_{jia}.
\]
In density covectors the spinorial Hessian is
\begin{equation}\label{eq:hessian-native-covectors}
\begin{aligned}
 L_z(\cH x)&=4\iota_{V(E_u)}\varphi_0,\\
 L_M(\cH x)&=I(E_u)+2\iota_{V(E_u)}\varphi_0+E_\beta
                +\tfrac12E_\Phi g_0+\eps\KHull^\dagger E_k,\\
 L_\alpha(\cH x)&=E_\alpha,\qquad L_l(\cH x)=-\tfrac12E_\Phi.
\end{aligned}
\end{equation}
Two-forms in the second line are skew \((i,a)\) tensors. The last term
is the derivative of the induced Hull connection. Integration by parts
differentiates curvature coefficients as well as the test field.

With the covector-first potential, for a lower cochain of degree
\(p=-1,0,1\) and its complementary dual,
\begin{equation}\label{eq:finite-unary-dual-signs}
 B_C(x_p,y_{3-p})=(-1)^{p+1}y_{3-p}[x_p].
\end{equation}
The lower-to-dual signs are \((+,-,+)\). Together with Hessian symmetry
they give \(B_C(d_Cx,y)+(-1)^{|x|}B_C(x,d_Cy)=0\).
Thus the cyclic and formal adjoints have different signs: the upper
arrows in density covectors are \(+G^\dagger,-\mathsf R^\dagger\).

\subsection{The full mixed principal symbol}
\label{app:linear-symbol}

Choose local representatives of the coefficient flag only to write
component operators. This does not split the differential. The order
bounds, with sources as rows and targets as columns, are
\begin{equation}\label{eq:main-unary-order-bounds}
\begin{array}{c|ccc}
 &\mathsf G_0&\mathsf A&\mathsf G_1\\\hline
 \mathsf G_0&1&2&3\\
 \mathsf A&0&1&2\\
 \mathsf G_1&0&0&1
\end{array}.
\end{equation}
Here zero entries denote identically vanishing blocks forced by
\(\mathbb D\)-linearity; in particular the \(\mathsf G_1\to\mathsf G_0\)
bound is negative. The bounds follow directly from
\eqref{eq:linear-gauge-complete} and
\eqref{eq:hessian-euler-inputs}--\eqref{eq:hessian-native-covectors}
and their adjoints. The induced-connection variation has order one.
Its differentiated functional followed by its formal adjoint has order
three. Commuting derivatives replaces two input derivatives by curvature;
an adjoint derivative of a curvature coefficient also lowers the input
order. Thus every \(R_0\), \(\nabla R_0\), and curvature insertion is below
the corresponding mixed principal order. For example
\(R_0\dot\psi_w\) has order zero and its induced-connection adjoint has
order at most one, below the order-three geometric bound.

We use the real symbol convention, replacing \(d,\nabla\) by
\(\zeta\wedge,\zeta\); formal adjoints use \(-\zeta\).
For the absolute field variables \(f=(u,\beta,p)\), set
\begin{equation}\label{eq:finite-geometric-hessian-symbol}
 H_\zeta f=
\begin{pmatrix}
 J(\zeta\wedge\beta)-*(\zeta\wedge u)+p*(\zeta\wedge\varphi_0)\\
 -*\{\zeta\wedge(*Ju-2p\psi_0)\}\\
 *(\zeta\wedge u\wedge\varphi_0-2\zeta\wedge\beta\wedge\psi_0)
\end{pmatrix}.
\end{equation}
This is a square \(57\times57\) block, with the factor \(w/4\) removed.
The geometric field carrier is \(\Lambda^3\oplus\Lambda^2\oplus\Lambda^0\),
equivalently \((z,M,l)\), of dimension \(35+21+1=57\).
It includes the spinor, density and unrestricted gerbe fields.
Put
\[
 (K_\zeta(u,\beta))_{iab}
 =\tfrac12\{\zeta_bH(u)_{ia}-\zeta_aH(u)_{ib}
                            +(\zeta\wedge\beta)_{iab}\},
\quad M_\zeta\alpha=*(\zeta\wedge\alpha\wedge\psi_0),
\quad N_\zeta=\tfrac12K_\zeta^\dagger M_\zeta K_\zeta.
\]
The adjoint includes \(P=-\tr_7\) and the derivative sign. In a skew-matrix
basis with \(P(E_{ab},E_{ab})=2\), it is \(-2K_\zeta^T\).
Order field layers as \((f_0,\alpha,f_1)\) and normalized density layers
as \((e_0,e_{\rm g},e_1)\). Up to the common sign of \(d_C^1=-\cH\),
the complete weighted middle symbol is
\begin{equation}\label{eq:finite-square-symbol}
 \sigma^{\rm mix}(\cH)(\zeta)=\frac w4
 \begin{pmatrix}
 H_\zeta&0&0\\
 0&-\tfrac12M_\zeta&0\\
 N_\zeta&0&H_\zeta
 \end{pmatrix}.
\end{equation}
It is square of size \(261=57+147+57\). There are \(21\) gauge
generators. The gauge output is the normalized amplitude of its
ideal-valued covector, not a second free coefficient. The cubic Hull
block \(N_\zeta\) is retained. It respects the geometric gauge symbol
\[
 G_\zeta(X,\Lambda)
 =(\zeta\wedge\iota_X\varphi_0,\zeta\wedge\Lambda,0):
 \qquad M_\zeta K_\zeta G_\zeta=0,\quad
 G_\zeta^\dagger N_\zeta=0.
\]
Indeed the metric gauge variation gives
\(K_\zeta G_\zeta X=\zeta\kappa_\zeta(X)\) and the gerbe variation gives
zero. Wedging twice with \(\zeta\), and then transposing, proves both
identities. These are the principal-symbol versions of the unary squares.

Each geometric quotient has middle kernel equivalently described by
\begin{equation}\label{eq:geometric-symbol-equations}
\begin{gathered}
 \zeta\wedge(*Ju-2p\psi_0)=0,\qquad
 \zeta\wedge u\wedge\varphi_0=0,\\
 *(\zeta\wedge u)-2p*(\zeta\wedge\varphi_0)+\zeta\wedge\beta=0.
\end{gathered}
\end{equation}
This is obtained by the invertible algebraic diagonal of the independent
equation-to-density map in Theorem~\ref{thm:local-bps-euler}, reduced at
\(\eps=0\). No added equation or restricted \(\beta\) is used.
Take \(|\zeta|=1\) and put \(W=\zeta^\perp\). Its induced \(SU(3)\)
structure has \(\varphi_0=\zeta\wedge\omega+\rho_3\) and
\(\psi_0|_W=\omega^2/2\). Write
\(u=\zeta\wedge\alpha_2+\nu\) and
\(\beta=\zeta\wedge\eta+\zeta_2\), with transverse forms.
The last equation gives \(\nu=2p\rho_3\) and \(\zeta_2=0\), since its
transverse and \(\zeta\)-containing three-form parts vanish separately.
Decompose \(\alpha_2=a\omega+\iota_v\rho_3+\alpha_8\), where
\(\alpha_8\) is primitive of type \((1,1)\). The first equation becomes
\[
 \omega\wedge\alpha_8+\tfrac13p\omega^2=0.
\]
The two \(SU(3)\) summands are independent and multiplication by
\(\omega\) is injective on the primitive summand. Thus
\(p=\alpha_8=0\). The scalar equation then vanishes, and the kernel is
exactly \(u=\zeta\wedge\iota_X\varphi_0\),
\(\beta=\zeta\wedge\eta\), \(p=0\), the \(7+6=13\) dimensional
gauge image. The middle rank is \(44\). The scalar reducibility has
rank one, and the upper ranks are the density-adjoint ranks \(13,1\).

For each gauge generator the quotient symbol complex is
\[
 \Lambda^0\xrightarrow{\zeta\wedge}\Lambda^1
 \xrightarrow{\zeta\wedge(-)\wedge\psi_0}\Lambda^6
 \xrightarrow{\zeta\wedge}\Lambda^7.
\]
The middle map restricts to the isomorphism
\(\Lambda^1W^*\to\zeta\wedge\Lambda^5W^*\) given by multiplication
by \(\omega^2/2\). Its rank is six; the outer ranks are one. Therefore
\[
\begin{array}{c|c|c}
 &\text{dimensions in degrees }-1,\ldots,4&\text{outgoing ranks}\\\hline
 \mathsf G_i&(1,14,57,57,14,1)&(1,13,44,13,1,0)\\
 \mathsf A&(0,21,147,147,21,0)&(0,21,126,21,0,0)\\
 \mathscr E_{\rm var}&(2,49,261,261,49,2)&(2,47,214,47,2,0).
\end{array}
\]
For every nonzero covector the coefficient flag gives short exact
sequences of symbol complexes
\[
 0\to\mathsf G_1\to\ker\eps\to\mathsf A\to0,
 \qquad 0\to\ker\eps\to\mathscr E_{\rm var}\to\mathsf G_0\to0.
\]
Their long exact sequences prove full exactness. \(G_2\) transitivity on
unit covectors and homogeneous weighted rescaling cover all \(\zeta\ne0\).

The field and parameter maps have algebraic diagonals and first-order
corrections \(\mathsf G_0\to\mathsf A\) in the same cochain degree.
Their cotangent maps have first-order corrections
\(\mathsf A\to\mathsf G_1\). Their inverses have the same bounds.
They consequently have mixed order zero and inverse principal symbols.
This proves the statement in both charts. Applying the full mixed
density dual and then \(\tau_{\mathbb D}\) gives the signed transpose at
\(-\zeta\), so the adjoint symbol sequence is exact as well. The free
coefficient pair evaluates as \(a_0b_1+a_1b_0\); the quotient/ideal
pair evaluates as ordinary real duality. The geometric density duality
is distinct from the positive auxiliary metrics used below.
\subsection{Order reduction and the smooth Hodge contraction}
\label{subsec:variational-hodge-proof}

We prove Proposition~\ref{prop:mixed-order-hodge}. Choose positive
bundle metrics and positive invertible elliptic operators \(\Lambda_a\)
of order one. Put
\begin{equation}\label{eq:unary-order-reduction}
 T_p=\diag(\Lambda_a^{s_{p,a}}),\qquad
 D_p=T_{p+1}^{-1}d_pT_p.
\end{equation}
The order bounds make every block of \(D_p\) have order at most zero.
Its principal symbol is the weighted symbol conjugated by the
invertible symbols of \(T_p\). Thus the order-zero symbol complex is
exact, including its endpoints. On \(\mathscr X^p=L^2(Y,E^p_{\mathbb R})\)
the maps \(D_p\) are bounded. Before order reduction the component
spaces are \(H^{k-s_{p,a}}\) at Sobolev level \(k\). The maps
\(T_p^{\pm1}\) preserve smooth sections. The identity \(D_{p+1}D_p=0\)
first holds there and then extends to the bounded Hilbert maps.

Set \(D_p=0\) outside the complex and define
\begin{equation}\label{eq:resolved-laplacian}
 \Delta_p=D_{p-1}D_{p-1}^*+D_p^*D_p.
\end{equation}
At a nonzero covector, a vector in the kernel of its symbol lies in
\(\ker\sigma(D_p)=\im\sigma(D_{p-1})\) and is orthogonal to that image.
It is zero. Thus \(\Delta_p\) is positive, self-adjoint and elliptic
of order zero. The compact pseudodifferential parametrix theorem
\cite[Theorems~8.6--8.7 and~8.11]{Grubb} implies that
\(\mathcal H^p=\ker\Delta_p\) is finite-dimensional and smooth and that
\(\Delta_p\) has closed range \((\mathcal H^p)^\perp\). The theorem
applies to the complexification and hence to the real operator.

Let \(P_p\) be the harmonic projection and invert \(\Delta_p\) on its
orthogonal complement, extending by zero. This gives a bounded real
Green operator with
\begin{equation}\label{eq:resolved-green}
 \Delta_pG_p=G_p\Delta_p=1-P_p,\qquad G_pP_p=P_pG_p=0.
\end{equation}
For a parametrix \(Q_p\Delta_p=1-S_p\),
\(G_pf=Q_p(1-P_p)f+S_pG_pf\). For smooth \(f\), both terms are smooth.
The same identity, the parametrix Sobolev estimates and the bounded
\(L^2\) inverse give continuity on every nonnegative Sobolev space,
hence on \(C^\infty\). The harmonic projection has a smooth finite-rank
kernel. In particular these operators have smooth formal adjoints.

Expanding \(\Delta\) using \(D^2=0\) gives
\(\Delta_{p+1}D_p=D_p\Delta_p\) and the adjoint identity.
The differential and its adjoint kill the appropriate harmonic spaces
and map into the next harmonic orthogonal complement. Uniqueness of
the inverse there gives
\[
 G_{p+1}D_p=D_pG_p,\qquad
 G_{p-1}D_{p-1}^*=D_{p-1}^*G_p.
\]
Identify \(H^p(C)\) with \(\mathcal H^p\) and let \(\widetilde i_p\)
be its inclusion. Define
\begin{equation}\label{eq:resolved-full-sdr}
 \widetilde h_p=D_{p-1}^*G_p,\quad
 i_{0,p}=T_p\widetilde i_p,\quad p_{0,p}=P_pT_p^{-1},\quad
 \mathsf h_{0,p}=T_{p-1}\widetilde h_pT_p^{-1}.
\end{equation}
Then \(di_0=0\), \(p_0d=0\) and \(p_0i_0=1\). Moreover
\[
 D_{p-1}\widetilde h_p+\widetilde h_{p+1}D_p
 =(D_{p-1}D_{p-1}^*+D_p^*D_p)G_p=1-P_p.
\]
Also \(P\widetilde h=\widetilde hP=0\) and
\(\widetilde h_{p-1}\widetilde h_p
 =D_{p-2}^*D_{p-1}^*G_p^2=0\). Cancelling adjacent
\(T^{\pm1}\)'s proves all identities
\eqref{eq:resolved-sdr-identities}. A smooth cycle \(v\) satisfies
\(v-i_0p_0v=d\mathsf h_0v\), and applying \(p_0\) to an exact harmonic
representative shows it vanishes. Thus \(\mathcal H^p\cong H^p(C)\).
The same identity shows closed range on the component Sobolev scales.

\subsection{Perfect duality and cyclic refinement}
\label{subsec:cyclic-transfer}

Set \(B=\tau_{\mathbb D}B_C\), as in
Section~\ref{subsec:cyclic-kuranishi}. It has cochain degree \(-3\) and
\[
 B(y,x)=(-1)^{|x||y|}B(x,y),\qquad
 B(d_Cx,y)+(-1)^{|x|}B(x,d_Cy)=0.
\]
On a free coefficient block,
\(\tau_{\mathbb D}((a_0+\eps a_1)(b_0+\eps b_1))=a_0b_1+a_1b_0\),
with invertible matrix \(\left(\begin{smallmatrix}0&1\\1&0\end{smallmatrix}\right)\).
On a quotient block paired with its ideal dual it is ordinary real
evaluation. Thus \(B\) is nondegenerate on the real cochain bundles.
Boundaries pair trivially with cycles, so it descends to cohomology.

For an operator \(A\) of degree \(r\), define the graded cyclic adjoint by
\[
 B(Ax,y)=(-1)^{r|x|}B(x,A^\sharp y).
\]
Then
\(d_C^\sharp=-d_C\). The Hodge operators have smooth adjoints, and
taking the cyclic adjoint of
\(1-P_0=d_C\mathsf h_0+\mathsf h_0d_C\), where \(P_0=i_0p_0\), gives
\[
 1-P_0^\sharp=d_C\mathsf h_0^\sharp+\mathsf h_0^\sharp d_C.
\]
If a cycle \(z\) pairs trivially with every complementary cycle,
\(B(P_0^\sharp z,y)=B(z,P_0y)=0\) for every smooth \(y\).
Chain-level nondegeneracy gives \(P_0^\sharp z=0\), hence
\(z=d_C\mathsf h_0^\sharp z\). The cohomology pairing has zero radical.
Finite-dimensionality therefore gives
\(H^p_{\rm var}\simeq(H^{3-p}_{\rm var})^*\).

Choose \(i_c=i_0\) and set \(B_H(a,b)=B(i_ca,i_cb)\). Define \(p_c\)
by \(B_H(p_cx,a)=B(x,i_ca)\). Then \(p_ci_c=1\), \(p_cd_C=0\), and
\(P_c=i_cp_c\) is self-adjoint. Put
\[
 A_c=1-P_c,\qquad b=A_c\mathsf h_0A_c,\qquad
 t=\tfrac12(b+b^\sharp),\qquad \mathsf h_c=t d_Ct.
\]
We have \(d_Ct+td_C=A_c\), \(tA_c=A_ct=t\), and
\(d_Ct^2=t^2d_C\). Thus \(d_Ctd_C=d_C\), and
\[
 d_C\mathsf h_c+\mathsf h_cd_C=A_c,\qquad
 p_c\mathsf h_c=\mathsf h_ci_c=0,\qquad
 \mathsf h_c^2=t d_Ct^2d_Ct=0.
\]
The graded adjoint rule gives \(\mathsf h_c^\sharp=\mathsf h_c\).
This is a cyclic strong deformation retract of the actual variational
complex. It preserves smooth sections; locality or
\(\mathbb D\)-linearity is not asserted for it.

For completeness, the cyclicity of transfer can be seen directly in
suspended coordinates. Every nonlinear inclusion tensor is in the
image of \(\mathsf h_c\). Its pairing with \(i_cH\) is zero because
\(\mathsf h_ci_c=0\), and the pairing of two such tensors is zero
because \(\mathsf h_c^2=0\). The transferred inclusion therefore pulls
the constant suspended form back exactly to the induced form on
cohomology. The source Hamiltonian tensor is graded symmetric, and the
intertwining identity transfers this symmetry to the minimal
coderivation. Desuspension gives
\[
 B_H(\mu_n(x_1,\ldots,x_n),x_{n+1})
 =(-1)^{n+|x_1|(|x_2|+\cdots+|x_{n+1}|)}
 B_H(\mu_n(x_2,\ldots,x_{n+1}),x_1).
\]
This is the cyclic convention of Appendix~\ref{app:mixed}; it also
follows from cyclic homotopy transfer
\cite[Theorem~4.2]{ChuangLazarev2009}.

\subsection{Transfer and nilpotent simplicial tests}
\label{subsec:kuranishi-transfer}\label{app:resolved-transfer}

We use the suspension and coefficient conversion of
Appendix~\ref{app:mixed}. On the conilpotent symmetric coalgebra the
source coderivation has components
\[
 q_n(sx_1\cdots sx_n)=(-1)^{\binom n2+\sum_a(n-a)|x_a|}
                     s\ell_n(x_1,\ldots,x_n),\qquad q_1=sd_Cs^{-1}.
\]
Let \(\bar i=s i_cs^{-1}\), \(\bar p=s p_cs^{-1}\) and
\(k=-s\mathsf h_cs^{-1}\). Then
\(q_1k+kq_1=\bar i\bar p-1\),
\(k^2=k\bar i=\bar pk=0\). The transferred inclusion corestrictions
and coderivation corestrictions are obtained from
\begin{equation}\label{eq:resolved-transfer-recursion}
 f_1=\bar i,\quad m_1=0,\qquad
 U_n=\sum_{j=2}^nq_j\mathbf I_n^j,\quad
 f_n=kU_n,\quad m_n=\bar pU_n\quad(n\ge2).
\end{equation}
Here \(\mathbf I_n^j\) is the sum over ordered positive partitions
\(n_1+\cdots+n_j=n\) and their unshuffles of
\(f_{n_1}\odot\cdots\odot f_{n_j}\), with suspended Koszul signs,
divided by \(j!\). Only earlier \(f_a\) occur. Homotopy transfer
\cite[Theorem~1.9]{Bandiera} gives \(\mathbf m^2=0\),
\(\mathbf q\mathbf I=\mathbf I\mathbf m\), and a reverse infinity
morphism \(\mathbf G\) with linear part \(\bar p\). Its hypotheses are
the characteristic-zero uncurved source, the signed contraction, and
a complete compatible filtration: apply it first to \(tL[[t]]\) and
\(tH[[t]]\) with the coefficientwise contraction and \(t\)-adic
filtration. Extract coefficients to obtain the real operations.

Desuspension gives \(\mu_n\) of degree \(2-n\) and inclusion tensors
of degree \(1-n\). A rooted tree of arity \(n\) has
\(\sum_v(\operatorname{valence}(v)-1)=n-1\), so at most \(n-1\)
vertices, each of valence at most \(n\). There are finitely many such
trees at each arity even when the source has all arities. Their vertices
are finite-order smooth Taylor operators and their homotopy edges
preserve smooth sections. Every coefficient is therefore defined.

For a dg-Artin algebra \(R\), use its finite augmentation filtration
on \(L\otimes\mathfrak m_R\). Extend the contraction to its internal
differential and to \(\Omega_n=\Omega^\bullet_{\rm poly}(\Delta^n)\).
For a homogeneous \(v\), the two external terms in
\(d_{\rm tot}(\mathsf h_c\otimes1)
 +(\mathsf h_c\otimes1)d_{\rm tot}\) have coefficients
\((-1)^{|v|-1}\) and \((-1)^{|v|}\), and cancel. Thus all contraction
identities persist and the transfer tensors extend with the total
Koszul signs. The target is exactly
\(H\otimes\mathfrak m_R\otimes\Omega_n\).

The formal Kuranishi theorem \cite[Theorem~1.13]{Bandiera} gives
\[
 \operatorname{MC}(L\otimes\mathfrak m_R\otimes\Omega_n)
 \cong\operatorname{MC}(H\otimes\mathfrak m_R\otimes\Omega_n)
       \times(-\mathsf h_c)((L\otimes\mathfrak m_R\otimes\Omega_n)^1).
\]
The inclusion is the zero section of the second factor. This construction
is natural in \(R\) and simplex maps; its recursion terminates modulo
the augmentation ideal. The auxiliary simplicial vector space is
\[
 \bigoplus_{j\le1}\im(-\mathsf h_c:
       (L\otimes\mathfrak m_R)^j\to(L\otimes\mathfrak m_R)^{j-1})
       \otimes\Omega_\bullet^{1-j}.
\]
Only finitely many \(j\)'s occur. Each \(\Omega_\bullet^k\) is
contractible as a simplicial vector space \cite[Lemma~3.2]{Getzler},
and its linear contraction tensors with every real vector space. The
auxiliary factor is consequently contractible. This proves the weak
equivalence for the entire dg-Artin category, with no density or
identity coordinate removed.

For two contractions, compose the first inclusion with the second
reverse infinity morphism. Its linear term is the identity on abstract
cohomology. Recursion in symmetric length constructs its
\(L_\infty\) inverse. This proves the independence assertion in
Theorem~\ref{thm:finite-cyclic-kuranishi} while leaving the individual
Taylor tensors dependent on the contraction.

\section{The unary physical Hull graph}
\label{app:unary-hull}
\label{app:intrinsic-linear-bps}
\subsection{The independent complex and its field graph}
\label{app:independent-hull-graph}

We use the three-row complex of \cite{QuantumG2}, shifted so that
fields have cochain degree one. Write \(\varphi=\varphi_0\),
\(\psi=*\varphi\), \(w=e^{-2\Phi_0}\) and \(\Psi=w\psi\).
The background has \(\nabla\varphi=0\), \(A_0=\Theta_0=\nabla\),
\(H_0=0\), \(d\Phi_0=0\) and \(F_0=R\). The tangent and gauge
algebras \(\mathfrak g_T,\mathfrak g_G\) are the full marked copies
of \(\mathfrak{so}(E)\). Covariant differentiation, curvature
insertion and the three-row maps preserve these skew-adjoint bundles.
Set \(\mathcal Q=TY\oplus\mathfrak g_G\oplus\mathfrak g_T\).
The graded terms are
\[
\begin{array}{c|ccc}
 p&\text{density row}&\text{coefficient row}&\text{gerbe row}\\\hline
 -1&0&0&\Omega^0\\
 0&0&\Omega^0(\mathcal Q)&\Omega^1\\
 1&\Omega^4_1&\Omega^1(\mathcal Q)&\Omega^2_7\\
 2&\Omega^5_7&\Omega^6(\mathcal Q)&\Omega^3_1\\
 3&\Omega^6&\Omega^7(\mathcal Q)&0\\
 4&\Omega^7&0&0
\end{array}
\]
Geometric rows have coefficients in \(\mathbb D\), and both adjoint
rows have coefficients in \(\mathbb D_0\). Write
\(\jmath_\epsilon:\mathbb D_0\to\mathbb D\),
\(a\mapsto\epsilon a\), for the canonical ideal inclusion.
Every adjoint output uses the reduction modulo $\epsilon$ of its geometric argument.

In the convention \(P=-\operatorname{tr}_7\), the independent
differential is
\begin{equation}\label{eq:hull-independent-differential}
 d_{\mathrm{ind}}=
 \begin{pmatrix}
 \check d&T&0\\
 0&\check{\mathcal D}&S\\
 0&0&-\check d
 \end{pmatrix},\qquad
 \mathcal D
 \begin{pmatrix}y\\a\\k\end{pmatrix}
 =\begin{pmatrix}
 d_\nabla y_m+\dfrac{\jmath_\epsilon}{4}
       \{P(R_m\wedge a)-P(R_m\wedge k)\}\\
 Da+R^m\wedge y_m^{[0]}\\
 Dk+R^m\wedge y_m^{[0]}
 \end{pmatrix}.
\end{equation}
Here \(R_m=\iota_{e_m}R\), and the arguments of
\(\jmath_\epsilon\) are evaluated over \(\mathbb D_0\).
The middle arrows are \(\mathcal D\),
\(\Psi\wedge\mathcal D\) and \(\mathcal D\).
The bottom arrows are \(-d,-\pi_7d,-\pi_1d\), and the top
arrows are the projected exterior derivatives on the displayed graded terms.
The source uses the matrix trace \(\tau=-P\). Its field equation
and weighted pairing fix the coefficient \(1/4\) in
\eqref{eq:hull-independent-differential}, with the curvature one-form
written first. The trace reflection
\[
 (M,k,\alpha)\longmapsto(M,k,2k-\alpha)
\]
intertwines the coefficient operators. No compatibility of this
reflection with their cyclic pairings is used.

The maps between rows in \cite{QuantumG2} are
\[
\begin{aligned}
 (S^{-1}\mu)_a&=-\iota_{e_a}\pi_{14}d\mu,\\
 (S^0q)_a&=-\iota_{e_a}(\pi_7+\pi_{27})dq\wedge\Psi,\\
 (S^1\chi)_a&=-\iota_{e_a}d\chi\wedge\Psi,
\end{aligned}
\qquad
\begin{aligned}
 T^{-1}v&=\pi_1d(v^a\iota_{e_a}\Psi),\\
 T^0y&=-\pi_7^5d\,\pi_{7+27}^4
                 \sum_a\iota_{e_a}(y^a\wedge\Psi),\\
 T^1\eta&=d\,\pi_{14}^5\sum_a\iota_{e_a}\eta^a.
\end{aligned}
\]
Only the geometric coefficient component of \(S\) is nonzero,
and \(\pi_{14}^5=*\pi_{14}^2*\). The superscripts on \(S,T\)
are the source degree conventions before the shift of the total
complex. The mixed square identities proved in \cite{QuantumG2},
together with the instanton and Bianchi identities, give
\(d_{\mathrm{ind}}^2=0\). The specialization to the mixed coefficient
modules preserves these identities.

For the field dictionary, write
\(u=\delta\varphi\), \(\beta=\delta B\), \(p=\delta\Phi\),
\(\alpha=\delta A\), \(h=\delta g\) and
\(M=(h+\beta)/2\). In an orthonormal frame put
\[
 v_a=\tfrac14\langle u,\iota_{e_a}\psi\rangle,\qquad
 m=-\tfrac13\iota_v\varphi,\qquad
 M_{\mathrm{geom}}^a=\tfrac12h_{ia}e^i+\iota_{e_a}m.
\]
Then \(u=\sum_aM_{\mathrm{geom}}^a\wedge\iota_{e_a}\varphi\).
The independent field coordinates are
\begin{equation}\label{eq:hull-independent-fields}
\begin{aligned}
 \rho&=-2p+\tfrac27\operatorname{tr}h,\qquad
 y^a=M_{\mathrm{geom}}^a-\tfrac12\iota_{e_a}\pi_{14}\beta,\\
 q&=-\tfrac12\pi_7\beta-m,\qquad
 a=2k-\alpha,\qquad k=\kappa.
\end{aligned}
\end{equation}
The density-row field is \(\rho\Psi\), and the source dilaton
coordinate is \(-\rho/2\). Its gerbe variation is
\(\mathcal B=-\beta\). With the trace reflection, this gives
\(\delta H_{\mathrm{source}}=-\delta H\).
The algebraic inverse is
\begin{equation}\label{eq:hull-independent-fields-inverse}
\begin{aligned}
 u&=\sum_a y^a\wedge\iota_{e_a}\varphi,\qquad
 \beta=-2q-\sum_a e^a\wedge y_a,\\
 p&=\tfrac27\sum_a y^a(e_a)-\tfrac12\rho,\qquad
 \kappa=k,\qquad\alpha=2k-a.
\end{aligned}
\end{equation}
In form-index-first matrix notation, \(M=y-q\).
Substituting \(k=\KHull M^{[0]}\), with \(\KHull\) given by
\eqref{eq:linear-hull-graph}, defines \(J_1\).
The full induced connection also has the first-order anomaly term.
Only its reduction modulo $\epsilon$ enters the independent adjoint row and the explicitly
weighted trace terms.

For physical parameters define
\(\varkappa(\xi)_{ab}=(\nabla_b\xi_a-\nabla_a\xi_b)/2\).
The parameter and reducibility maps are
\begin{equation}\label{eq:hull-independent-parameters}
\begin{aligned}
 J_0(\xi,\Lambda,\lambda)
 &=\left(\xi,\ \varkappa(\xi^{[0]})-\lambda,
        \ \varkappa(\xi^{[0]}),\ \tfrac12(\Lambda-\xi^\flat)\right),\\
 J_{-1}c&=-c/2.
\end{aligned}
\end{equation}
The degree-zero entries are \((\sigma,a_0,k_0,\mu)\).
On the graph \(k_0=\varkappa(\sigma^{[0]})\), the inverse is
\[
 \xi=\sigma,\qquad\Lambda=2\mu+\sigma^\flat,\qquad
 \lambda=k_0-a_0.
\]
The scalar inverse multiplies by \(-2\).
These maps intertwine the lower differentials. For the geometric
components this follows by differentiating
\eqref{eq:hull-independent-fields}. The adjoint components use
\(\KHull G_M=D\varkappa(\xi)+\iota_\xi R\) and
\(R\in\Omega^2_{14}\). The reducibility square uses \(d^2=0\).

Before Hull restriction, the independent equation coordinates are
\[
\begin{aligned}
 C&=w\,d(\mathcal J_{\varphi_0}u-2p\psi),\qquad
 \Sigma=du\wedge\varphi,\\
 T&=d\beta-\tfrac{\jmath_\epsilon}{2}P(\alpha-\kappa,R)
                   +*du-2*(dp\wedge\varphi),\\
 I_G&=D\alpha\wedge\psi+R\wedge(\mathcal J_{\varphi_0}u)^{[0]},\\
 I_T&=D\kappa\wedge\psi+R\wedge(\mathcal J_{\varphi_0}u)^{[0]}.
\end{aligned}
\]
For an arbitrary physical equation argument \(E=(C,\Sigma,T,I)\),
write
\[
\begin{gathered}
 c=C/w,\qquad\sigma=*\Sigma,\qquad
 A(z)=\tfrac14(v\mapsto v\wedge\varphi)^\dagger z,\qquad
 B(z)=\tfrac13(v\mapsto v\wedge\psi)^\dagger z,\\
 \nu=\tfrac32B(c)-2A(*T).
\end{gathered}
\]
The intrinsic-torsion inverse defines \(t\) by
\[
\begin{aligned}
 A_t&=\sum_{i,a}t_{ia}e^i\wedge\iota_{e_a}\psi
                  =*T+2\nu\wedge\varphi,\\
 B_t&=-\sum_{i,a}t_{ia}e^i\wedge e^a\wedge\varphi
                  =c+2\nu\wedge\psi.
\end{aligned}
\]
The tangent residual on the field graph is
\begin{equation}\label{eq:hull-tangent-equation}
\begin{aligned}
 \widehat r(E)&=\mathscr R_0(E^{[0]})\wedge\psi
       +W\,d\{\langle T^{[0]},\varphi\rangle-\sigma^{[0]}\},\\
 \mathscr R_0(E)_{ab}
 &=\tfrac13\sum_c(\nabla_at_{bc}-\nabla_bt_{ac})\iota_{e_c}\varphi,\\
 (*W(v))_{iab}&=\tfrac14(g_{0,ia}v_b-g_{0,ib}v_a),\qquad
 I_T\big|_{\kappa=\KHull M^{[0]}}=\widehat r(E).
\end{aligned}
\end{equation}
This is the scalar Hull identity
\eqref{eq:linear-scalar-flux-defect} applied to the independent
tangent-instanton equation.

Put
\[
 s=\frac{2\langle T,\varphi\rangle-3\sigma}{28},\qquad
 \tau_E=\frac{\langle T,\varphi\rangle+2\sigma}{28}.
\]
The full equation-coordinate map is
\begin{equation}\label{eq:hull-equation-coordinates}
\begin{aligned}
 \zeta&=\pi_7^5C,\qquad
 \eta_a=w*(-t_{ai}e^i+\tau_Ee^a),\qquad\chi=s\varphi,\\
 \eta_G&=w(2I_T-I_G),\qquad\eta_T=wI_T.
\end{aligned}
\end{equation}
For the geometric blocks this follows by applying the independent
field equations to \eqref{eq:hull-independent-fields}.
The curvature terms use
\[
 A_2\left(0,0,-\tfrac12P(a,R)\right)_{\mathrm{geom},m}
 =-\tfrac w4P(R_m\wedge a)\wedge\psi,\qquad
 R^m\wedge M_{\mathrm{geom},m}\wedge\psi
 =R\wedge\mathcal J_{\varphi_0}u.
\]
Here \(A_2\) denotes the map in
\eqref{eq:hull-equation-coordinates}, before graph restriction.
The first identity is an ordinary real identity before multiplication
by \(\jmath_\epsilon\). The \(\mathfrak g_2\) component of the
two-form acts trivially on \(\psi\), which gives both adjoint rows.

Writing \(L_{ai}=(*(\eta_a/w))_i\), the algebraic inverse is
\begin{equation}\label{eq:hull-equation-inverse}
\begin{aligned}
 \sigma&=2\operatorname{tr}L,\qquad
 t=-L+\left(\tfrac12s+\tfrac14\operatorname{tr}L\right)\id,\\
 \nu&=\tfrac12B(B_t-\zeta/w),\qquad
 C=w(B_t-2\nu\wedge\psi),\\
 T&=*(A_t-2\nu\wedge\varphi),\qquad
 \Sigma=\sigma\operatorname{vol}_0,\\
 I_T&=\eta_T/w,\qquad I_G=(2\eta_T-\eta_G)/w.
\end{aligned}
\end{equation}
Substituting \(I_T=\widehat r(E)\) and \(I_G=I\) into
\eqref{eq:hull-equation-coordinates} defines \(J_2\) and gives
\(d_{\mathrm{ind}}J_1=J_2E\).

For a tangent-valued \(p\)-form \(m\), lower its tangent index with
\(g_0\) and define
\[
 (K_pm)_{ab}=\nabla_bm_a-\nabla_am_b.
\]
Here \(\nabla\) is the background Levi--Civita connection on the
coefficient and exterior slots; the exterior \(p\)-form indices are
unchanged. Tangent-adjoint outputs use the reduction modulo
\(\epsilon\). Let \(j_2(a)=a\wedge\psi\) on \(\Omega^2_7\), and set
\(K_2^\wedge=j_2K_2j_2^{-1}\). Since
\(t=-L+\tau_E\id\), the torsion reconstruction gives
\[
 \mathscr R_0(E)\wedge\psi
 =K_2^\wedge(\eta/w)-4W(d\tau_E),\qquad
 \langle T,\varphi\rangle-\sigma-4\tau_E=12s.
\]
Combining these equalities with \eqref{eq:hull-tangent-equation}
derives the equation graph
\begin{equation}\label{eq:hull-degree-two-graph}
 \eta_T=K_2^\wedge\eta+12w\,W(ds),\qquad\chi=s\varphi.
\end{equation}

\subsection{Upper identities and the Hull-graph isomorphism}
\label{app:hull-upper-graph}

For an arbitrary equation argument define
\[
 n(E)=3B(d*T)+3B(d\nu\wedge\varphi)
                   +d(\sigma-\langle T,\varphi\rangle).
\]
The physical Noether operator of
Proposition~\ref{prop:linear-bps-noether-comparison} has components
\begin{equation}\label{eq:hull-physical-noether}
\begin{aligned}
 N_\Lambda&=-\tfrac14*dC,\\
 N_\lambda&=\tfrac18*(wDI-R\wedge C^{[0]}),\\
 N_X&=\tfrac w4n(E)+\tfrac{\jmath_\epsilon w}{8}
 \{\varkappa^\dagger D^\dagger*(\widehat r-I)
                         +J_R(*(\widehat r-I))\}.
\end{aligned}
\end{equation}
Here \(J_R(q)_i=\sum_jP(q_j,R_{ij})\), and
\(\varkappa^\dagger\) is the density adjoint using
\(P=-\operatorname{tr}_7\). In components,
\((\varkappa^\dagger z)_a=\nabla^bz_{ba}\).
The independent tangent Bianchi row satisfies
\begin{equation}\label{eq:hull-tangent-bianchi}
 N_T(E)=\tfrac18*\{R\wedge C^{[0]}-wD\widehat r(E)\}
       =\tfrac14d\bigl(N_X(E)^{[0]}+N_\Lambda(E)^{[0]}\bigr).
\end{equation}
The metric identifies the two-form on the right with a tangent
skew endomorphism.

To prove the factorization on the curved background, reduce the equation argument modulo $\epsilon$. Parallel \(G_2\) contractions give
\[
 n(E)-*dc=-2(\nabla^it_{ai})e^a+
             d\left(\sigma-\tfrac12\langle T,\varphi\rangle\right).
\]
The Hessian of a scalar is symmetric, so \(DW(df)=0\).
The curvature reconstruction in \eqref{eq:hull-tangent-equation}
therefore gives
\[
 (*D\widehat r)_{ab}
 =\nabla^i(\nabla_at_{bi}-\nabla_bt_{ai}).
\]
Use the convention
\([\nabla_i,\nabla_j]v^a=R_{ij}{}^a{}_bv^b\).
Commuting derivatives on both covariant slots yields
\[
 [\nabla^i,\nabla_a]t_{bi}
       -[\nabla^i,\nabla_b]t_{ai}=R_{ikab}t_{ik}.
\]
The Ricci terms vanish. Pair symmetry and the algebraic Bianchi
identity give the remaining terms. Since
\[
 R\wedge\psi=0,\qquad *(R\wedge\varphi)=-R,\qquad
 c=B_t-2\nu\wedge\psi,
\]
the curvature contraction is \(*(R\wedge c)_{ab}\).
Differentiating the expression for \(n(E)-*dc\) proves
\eqref{eq:hull-tangent-bianchi}. This calculation holds on arbitrary
equation arguments and does not assume \(\nabla R=0\).

Identify density duals with forms using the background volume.
For a degree-three physical element \(n=(n_X,n_\Lambda,n_\lambda)\),
put
\[
\begin{aligned}
 r(n)&=\tfrac14d(n_X^{[0]}+n_\Lambda^{[0]}),\\
 b(n)&=2\{n_X+n_\Lambda-
                   \jmath_\epsilon\varkappa^\dagger(n_\lambda+r(n))\}.
\end{aligned}
\]
The remaining maps are
\begin{equation}\label{eq:hull-upper-maps}
\begin{aligned}
 J_3n=\bigl(&-4*n_\Lambda,\quad
 b(n)^\sharp\otimes\operatorname{vol}_0,\\
 &-8(n_\lambda+2r(n))\operatorname{vol}_0,\quad
 -8r(n)\operatorname{vol}_0\bigr),\\
 J_4n={}&-4n\operatorname{vol}_0.
\end{aligned}
\end{equation}
The entries of \(J_3\) are the density, geometric coefficient,
gauge coefficient and tangent coefficient rows. The tangent entry
satisfies
\begin{equation}\label{eq:hull-degree-three-graph}
 \eta_T^{(7)}=K_3^\wedge\eta^{(7)}
       =-db\otimes\operatorname{vol}_0,\qquad
 \eta^{(7)}=b^\sharp\otimes\operatorname{vol}_0.
\end{equation}
Here \(K_3^\wedge=j_3K_3j_3^{-1}\), with
\(j_3(a)=a\wedge\Psi\) on \(\Omega^3_1\).
The tangent output reads \(b^{[0]}\).

Applying the independently defined differential to
\eqref{eq:hull-equation-coordinates} gives
\[
\begin{aligned}
 d\zeta+T^1\eta&=-4*N_\Lambda,\\
 R^m\wedge\eta_m^{[0]}&=-R\wedge C^{[0]},\\
 D\eta_T+R^m\wedge\eta_m^{[0]}&=-8N_T\operatorname{vol}_0,\\
 D\eta_G+R^m\wedge\eta_m^{[0]}
                       &=-8(N_\lambda+2N_T)\operatorname{vol}_0.
\end{aligned}
\]
Its geometric coefficient component, including \(S^1\chi\), is
\[
 \left\{
 2\left(\tfrac w4n(E)+N_\Lambda\right)
 +\tfrac{\jmath_\epsilon w}{4}J_R(*(\widehat r-I))
 \right\}^{\sharp}\otimes\operatorname{vol}_0.
\]
The identity
\[
 D^\dagger*(\widehat r-I)=\tfrac8w(N_\lambda+N_T)
\]
turns this component into the geometric entry of \(J_3N\).
Equation~\eqref{eq:hull-tangent-bianchi} gives its tangent entry.
Together with the lower squares and the final exterior derivative,
this proves all chain identities:
\begin{equation}\label{eq:hull-all-chain-identities}
\begin{aligned}
 d_{\mathrm{ind}}J_{-1}&=J_0\mathsf R,&
 d_{\mathrm{ind}}J_0&=J_1G,\\
 d_{\mathrm{ind}}J_1&=J_2E,&
 d_{\mathrm{ind}}J_2&=J_3N,\\
 d_{\mathrm{ind}}J_3&=J_4(-\mathsf R^\dagger),&
 d_{\mathrm{ind}}J_4&=0.
\end{aligned}
\end{equation}
In the last nonzero square the top row is explicitly
\(d(-4*n_\Lambda)=J_4(-d^\dagger n_\Lambda)\).

The subcomplex \(\mathcal E^\bullet_{\mathrm{Hull}}\) is defined
by \(k_0=\varkappa(\sigma^{[0]})\) in degree zero,
\eqref{eq:linear-hull-graph} in degree one,
\eqref{eq:hull-degree-two-graph} in degree two and
\eqref{eq:hull-degree-three-graph} in degree three.
Degrees minus one and four have no added constraint.
Equations~\eqref{eq:hull-all-chain-identities} prove closure.

The inverses in degrees minus one and zero are given after
\eqref{eq:hull-independent-parameters}. In degree one use
\eqref{eq:hull-independent-fields-inverse}, in the physical form chart. In degree two use
\eqref{eq:hull-equation-inverse}, with
\(I=(2\eta_T-\eta_G)/w\). For degree-three entries
\((a_6,b_7,g_7,t_7)\), the inverse is
\begin{equation}\label{eq:hull-upper-inverse}
\begin{aligned}
 n_\Lambda&=-*a_6/4,\qquad r=-*t_7/8,\\
 n_\lambda&=(2*t_7-*g_7)/8,\\
 n_X&=(*b_7)^\flat/2-n_\Lambda
                   +\jmath_\epsilon\varkappa^\dagger(n_\lambda+r).
\end{aligned}
\end{equation}
The degree-three graph is precisely
\(r=d(n_X^{[0]}+n_\Lambda^{[0]})/4\).
The inverse in degree four is \(-*/4\).
These are two-sided inverses on the graph. The forward differential
orders in degrees \(-1,\ldots,4\) are at most \(0,1,1,1,2,0\),
and the inverse orders are at most one. Thus
\[
 C_{\mathrm{BPS,N}}
 \cong_{\mathrm{local\ differential}}
 \mathcal E^\bullet_{\mathrm{Hull}}
 \hookrightarrow
 \mathcal E^\bullet_{\mathrm{ind}}.
\]
The first arrow is a chain isomorphism with finite-order local
differential inverse. The second is a strict subcomplex inclusion.
The comparison does not impose \(dB^{[0]}=0\).

\bibliographystyle{amsplain}
\bibliography{references}
\end{document}